%% file: torichyb.tex
\documentclass[a4paper, reqno,12pt]{amsart}
\usepackage[utf8]{inputenc}
\usepackage[T1]{fontenc}
\usepackage[english]{babel}
\usepackage{todonotes}
\usepackage{babelbib}
\usepackage{amsmath}
\usepackage{amsfonts}
\usepackage{amssymb}
\usepackage{amsthm}
\usepackage{hyperref}
\usepackage{color}
\usepackage{amsthm}
\usepackage{graphicx}
\usepackage{textcomp}
\usepackage{array}
\usepackage{pgf}
\usepackage{tikz}
\usepackage{tikz-3dplot}
\usepackage{tikz-cd}
\usetikzlibrary{cd}
\usetikzlibrary{decorations.markings}
\usepackage{float}
\usetikzlibrary{arrows}
\usetikzlibrary{patterns}
\usetikzlibrary{matrix,arrows}
\usepackage{mathrsfs}
\usepackage{bbm,yhmath,stmaryrd}
\usepackage{geometry}
\usepackage{cancel}
\usepackage{blkarray}
\usepackage{multirow}
\usepackage{booktabs}
\usepackage{mathtools}
\usepackage{fancyhdr}
\newtheorem{defn}[equation]{Definition}
\newtheorem{conj}[equation]{Conjecture}
\newtheorem{theo}[equation]{Theorem}
\newtheorem{prop}[equation]{Proposition}
\newtheorem{lem}[equation]{Lemma}

\newtheorem{cor}[equation]{Corollary}
\newtheorem{ex}[equation]{Example}

\newtheorem{rem}[equation]{Remark}
\newtheorem{con}[equation]{Condition}

\numberwithin{equation}{section}
\input macros.tex

\title[Hybrid toric varieties and the NA SYZ fibration on CY hypersurfaces]{Hybrid toric varieties and the non-archimedean SYZ fibration on Calabi-Yau hypersurfaces}
\author{Léonard Pille-Schneider}
\begin{document}
\date{}
\nocite{*}
\maketitle
\begin{abstract}
    Using a construction by Yamamoto of contractions maps from tropical Calabi-Yau varieties to their dual intersection complex, we construct a non-archimedean SYZ fibration for a class of maximally degenerate hypersurfaces in projective space. We furthermore prove that under a discrete symmetry assumption, the potential for the non-archimedean Calabi-Yau metric is constant along the fibers of the retraction.
    \\The proof uses the work of Li on the Fermat degeneration of hypersurfaces, and an explicit description of toric plurisubharmonic metrics on the hybrid space associated to a complex toric variety.
\end{abstract}
\tableofcontents
\section*{Introduction}
Let $X \fl \D^*$ be a degeneration of $n$-dimensional Calabi-Yau manifolds, polarized by a relatively ample line bundle $L$. We furthermore assume that the degeneration is maximal, which means that the algebraic geometry of the fibers $X_t$ degenerates in the most severe possible way (see Definition \ref{def sk ess} for the precise definition). Motivated by mirror symmetry, the SYZ conjecture predicts that as $t \fl 0$, the fiber $(X_t, \omega_t)$ endowed with its unique Kähler Ricci-flat metric $\omega_t \in c_1(L_t)$ should look like the total space of a Lagrangian torus fibration $f_t : X_t \fl B$ onto a certain base $B$ of real dimension $n$. The base $B$ can be naturally realized within the realm of non-archimedean geometry: writing $X^{\an}$ for the Berkovich analytification of $X$ with respect to the $t$-adic absolute value on $K=\C((t))$, there exists a canonical piecewise-linear subspace $\Sk(X) \subset X^{\an}$ called the essential skeleton of $X$, carrying a Lebesgue measure $\mu_0$ which is in a natural sense the limit of the Calabi-Yau measures $\mu_t = \omega_t^n$ on the $X_t$. The measure convergence takes place on the hybrid space $X^{\hyb}$ associated to the degeneration, which can be written as the disjoint union $X^{\hyb} := X \sqcup X^{\an}$, and whose topology is roughly defined in the following way: if $f$ is a holomorphic function on $X$, then the real-valued function $\frac{\log \lvert f \rvert}{\log \lvert t \rvert}$ on $X$ converges on $X^{\hyb}$ to the family of non-archimedean valuations $v_x \mapsto v_x(f)$ of $f$ on $X^{\an}$.
\\The Kontsevich-Soibelman conjecture \cite{KontsevichSoibelman} then predicts the following: after suitably rescaling the diameters, the Calabi-Yau manifolds $(X_t, \omega_t)$ converge in the Gromov-Hausdorff sense to the essential skeleton $\Sk(X)$, containing a dense open subset $\mathcal{R}$ which is a smooth Riemannian manifold, endowed with an integral affine structure, and such that the Riemannian metric is given locally in affine coordinates by the Hessian of a strictly convex potential $\phi$, satisfying the real Monge-Ampère equation:
$$\det( \frac{\partial^2 \phi}{\partial y_i \partial y_j}) = \mu_0,$$
where $(y_1,...,y_n)$ are integral affine coordinates. Heuristically, the integral affine structure should be thought as the limit of the highly degenerating complex structure on $X_t$, while the local potential $\phi$ is the limit of the Calabi-Yau potentials. While it has proven difficult to give a global definition of the real Monge-Ampère equation on the piecewise-affine space $\Sk(X)$, the work of Boucksom-Favre Jonsson \cite{BFJ2}, \cite{BFJ1} implies that it is possible to make sense of the non-archimedean Monge-Ampère measure $\MA(\phi)$ of a continuous plurisubharmonic (psh for short) metric $\phi$ on $L$ on the Berkovich analytification $X^{\an}$, and that there exists a unique continuous psh metric on $L^{\an}$ solving the non-archimedean Monge-Ampère equation:
$$\MA(\phi) = \mu_0,$$
where $\mu_0$ is the Lebesgue measure on the essential skeleton of $X$. This provides a natural strategy towards the proof of the Kontsevich-Soibelman conjecture: first, proving that the non-archimedean Calabi-Yau metric $\phi$ on $L^{\an}$ can be identified in a suitable sense to a convex solution of the real Monge-Ampère equation on $\Sk(X)$, and then show that the latter captures the asymptotic behaviour of the Calabi-Yau manifolds $(X_t, \omega_t)$.
\\The groundbreaking work of Yang Li \cite{Li2} makes substantial progress regarding the second part: assuming that there exists a possibly larger skeleton $\Sk(X) \subseteq S \subset X^{\an}$ and a retraction $\rho : X^{\an} \fl S$, such that the solution of the non-archimedean equation $\phi$ satisfies the \emph{comparison property}: $\phi = \phi \circ \rho$ on the union of open maximal faces of $\Sk(X)$, it can be identified to a solution of the real Monge-Ampère equation on $\Sk(X)$ on this union by \cite{Vil}, which provides an ansatz for the archimedean Calabi-Yau metrics, so that Li is able to show the existence of a special Lagrangian fibration $f_t : U_t \fl \Sk(X)$ on an open subset $U_t \subset X_t$ of asymptotically full Calabi-Yau measure when $\lvert t \rvert \ll 1$. Let us point out however that since the ansatz lives on the disconnected union of open maximal faces of $\Sk(X)$, the above-mentioned invariance property is not strong enough to expect it to imply global Gromov-Hausdorff convergence - it is therefore desirable to define a stronger comparison property to obtain finer convergence results.
\\The goal of this paper is to investigate this invariance property in the particular case of degenerations of hypersurfaces of the following form:
$$X = \{ z_0 ...z_{n+1} +tF =0\} \subset \CP^{n+1} \times \D^*,$$
where $F \in H^0(\CP^{n+1}, \gO(n+2))$ is a generic section. In that case, writing $N$ for the cocharacter lattice of the open torus $\T \subset \CP^{n+1}$ and $N_{\R} = N \otimes_{\Z} \R$, the essential skeleton $\Sk(X)$ can be canonically realized inside the tropical hypersurface $\Trop(X) \subset N_{\R}$ associated to $X$, i.e. the corner locus of the piecewise-affine function on $N_{\R}$ obtained by replacing sums by maxima and products by sums in the equation of the hypersurface. The tropicalization $\Trop(X)$ is also the image of $X^{\an} \cap \T^{\an}$ under the tropicalization map $\val : \T^{\an} \fl N_{\R}$, which is given by the log-modulus of coordinates on the torus. The tropicalization map extends as a map $\val_{\Sigma} :\CP^{n+1, \an} \fl N_{\Sigma}$ onto the manifold with corners $N_{\Sigma}$ associated to the fan $\Sigma$ of $\CP^{n+1}$, which provides a canonical compactification $\overline{\Trop}(X) := \val_{\Sigma}(X^{\an})$ of the tropical hypersurface.
\\Such degenerations of hypersurfaces have been extensively studied in the mirror symmetry literature, in particular in the context of toric degenerations and the Gross-Siebert program \cite{Gro05}, \cite{GrossSiebert2006}. The following will be of particular relevance to us: given a certain choice of branch cuts $\underline{a} =(a_{\tau})_{\tau}$ in $\Sk(X)$, i.e. an interior point $a_{\tau}$ of each face $\tau$, one can define a discriminant locus $\Gamma \subset \Sk(X)$ and an integral affine structure on $\Sk(X) \setminus \Gamma$, capturing the toric nature of the degeneration. The work of Yamamoto \cite{Ya} shows that this integral affine structure can be reconstructed via tropical geometry: there exists a tropical contraction $\delta = \delta_{\underline{a}} : \overline{\Trop}(X) \fl \Sk(X)$ which preserves the integral affines structures on $\overline{\Trop}(X)$ and $\Sk(X) \setminus \Gamma$ respectively. It is thus natural to compose this contraction with the tropicalization map $\val_{\Sigma} : X^{\an} \fl \overline{\Trop}(X)$, our first main result is the following:
\begin{thmA} Let $X$ be a degeneration of hypersurfaces of the form:
$$X = \{ z_0 ...z_{n+1} +tF =0\} \subset \CP^{n+1} \times \D^*$$
where $F \in H^0(\C\CP^{n+1}, \gO(n+2))$ is a generic section. Let $\underline{a}$ be a choice of branch cuts in $\Sk(X)$, and let $\delta_{\underline{a}} : \overline{\Trop}(X) \fl \Sk(X)$ be the associated tropical contraction \cite[thm. 5.1]{Ya}. Then the composition:
$$\rho_{\underline{a}} := \delta_{\underline{a}} \circ \val_{\Sigma} : X^{\an} \fl \Sk(X)$$
is an affinoid torus fibration over $\Sk(X) \setminus \Gamma$. Moreover, the induced integral affine structure on $\Sk(X) \setminus \Gamma$ coincides with the one mentioned above.
\end{thmA}
We will not define affinoid torus fibrations here and refer to def. \ref{defn:affinoid torus fibration}, they should be thought as the non-archimedean analog of smooth torus fibrations, and induce in particular an integral affine structure on their base. In the case of K3 surfaces the above retraction was constructed in a different manner in \cite[thm. A]{MPS}.
\\It is expected that the polarization on the hypersurface should dictate a prefered choice of branch cuts, i.e. such that the non-archimedean Calabi-Yau metric $\phi$ satisfies the comparison property for the retraction $\rho_{\underline{a}}$. When the family of hypersurfaces is the Fermat family, given explicitly by:
$$ X = \{ z_0 ... z_{n+1} + t(z_0^{n+2}+...+z_{n+1}^{n+2}) =0\} \subset \CP^{n+1} \times \D^*,$$
which has the discrete $\mathfrak{S}_{n+2}$-symmetry, there is a canonical choice of branch cuts, namely the barycenters of the faces of $\Sk(X)$. We are then able to show a strong version of the comparison property for the family $X$:
\begin{thmB} Let $X$ be the Fermat family of hypersurfaces, polarized by $L = \gO_{\CP}(n+2)$. Then there exists on $\CP_K^{n+1}$ a toric psh metric on $L^{\an}$ whose restriction $\phi$ to $X^{\an}$ solves the non-archimedean Monge-Ampère equation:
$$\MA(\phi) = \mu_0,$$
where $\mu_0$ is the Lebesgue measure on $\Sk(X)$. Moreover, writing $\phi = \phi_{\FS} + \psi$, where $\phi_{\FS}$ is the Fubini-Study metric on $L$, the continuous function $\psi$ on $X^{\an}$ satisfies the comparison property:
$$\psi = \psi \circ \rho_{\underline{a}}$$
over $\Sk(X) \setminus \Gamma$, where $\rho_{\underline{a}}$ is the retraction from theorem A, $a_{\tau}$ is the barycenter of $\tau$ for each face $\tau \subset \Sk(\X)$, and $\Gamma \subset \Sk(X)$ is the associated discriminant locus of codimension $2$.
\end{thmB}
The first part of this result was obtained independently in \cite{HJMM} for a larger class of hypersurfaces in projective space, via a different argument. We will additionally prove that the strong comparison property also holds for most hypersurfaces considered in their setting, see theorem \ref{theo comp prop}.
\\Our argument builds on the work from \cite{Li1}, together with the hybrid pluripotential theory developed in \cite{PShyb}, where we define a class of plurisubharmonic metrics on the hybrid space $X^{\hyb}$ associated to a degeneration $X$ of complex manifolds. We will in particular use a description of hybrid, fiberwise-toric metrics on a complex toric variety - given a complex projective variety $Z$, the associated hybrid space $p : Z^{\hyb} \fl [0,1]$ is a Berkovich analytic space, such that $p^{-1}((0,1]) \simeq Z \times (0,1]$, and $p^{-1}(0) \simeq Z^{\an}_0$ is homeomorphic to the analytification of $Z$ with respect to the trivial absolute value on $\C$. Given a semi-ample line bundle $L$ on $Z$, a continuous fiberwise-toric metric $\phi$ on $L^{\hyb}$ is then a continuous metric that restricts to a toric metric on $L$ on each fiber $p^{-1}(\lambda)$ for $\lambda \in [0,1]$, such metrics are well-known \cite{BPS} to be equivalent to the data of a continuous function on a finite-dimensional vector space $N_{\R}$, which has the same growth at infinity as the support function $\Psi_P$ of the polytope $P$ encoding the polarization $L$. The metric is then psh if and only the associated function is convex, and the following result states that this also holds in the hybrid setting:
\begin{thmC}
Let $Z$ be a projective complex toric variety, and $L$ a semi-ample line bundle on $Z$. There is a one-to-one correspondence between continuous, fiberwise-toric psh metrics on $L^{\hyb}$ and continuous convex functions:
$$\Phi : N_{\R} \times [0,1] \fl \R$$
such that $(\Phi - \Psi_P)$ extends continuously to $N_{\Sigma} \times [0,1]$.
\end{thmC}
This is slightly different than the set-up considered above, but this result is equivalent to the description of $\bS^1$-invariant, toric psh metrics on the hybrid space associated to the polarized degeneration $(Z \times \D^*, p_1^*L)$, which will be sufficient for our purposes.
\\ \textbf{Notation and conventions.} All rings are assumed to be unitary and commutative.
\\Throughout this text, whenever we say that $X \xrightarrow{\pi} \D^*$ is a degeneration of complex manifolds, we mean that $X$ is a smooth complex manifold and $\pi$ a holomorphic submersion (which will often be omitted from notation). We will furthermore always assume that the degeneration is meromorphic at $0$, i.e. that there exists a normal complex analytic space $\X \xrightarrow{\pi}  \D$ such that $\X_{| \D^*} = X$.
\\ \textbf{Organization of the paper.} In section \ref{sec Berko}, we recall various facts on Berkovich spaces and hybrid spaces, and pluripotential theory thereon. Section \ref{sec toric metrics} is devoted to the study of toric metrics on hybrid spaces, and the proof of theorem C. In section \ref{sec SYZ}, we discuss the SYZ heuristic and define a class of non-archimedean torus fibrations, as well as a strong comparison property for them.In section \ref{sec syz super}, we describe Yamamoto's tropical contraction in the hypersurface case, and study various models of our degeneration, providing a proof of Theorem A. Finally, in section \ref{sec Fermat}, we focus on the Fermat degeneration of hypersurfaces and prove Theorem B, and prove that the generic SYZ fibrations on the complex fibers converge in the hybrid topology to the retraction from theorem A.
\\ \textbf{Acknowledgements.} I am grateful to Sébastien Boucksom, Antoine Ducros and Yang Li for their comments on the first version of this paper. I also thank Enrica Mazzon and Yuto Yamamoto for helpful communications.
\section{Non-archimedean geometry} \label{sec Berko}
Let $K$ be a complete valued field, and $X/K$ a separated scheme of finite type. The Berkovich analytification $X^{\an}$ can be defined as the set of couples $x=(\xi, v_x)$, where $\xi = \xi_x$ is a scheme-theoretic point of $X$ and $v_x$ is a real valuation on the function field of $\overline{\xi}_x$, extending the valuation on $K$. 
\\The set $X^{\an}$ is endowed with the coarsest topology such that the maps $f \mapsto \lvert f \rvert_x$ are continuous on $U^{\an}$ whenever $f \in \gO_X(U)$, as well as the map $x \mapsto \xi_x$ from $X^{\an}$ to $X^{\sch}$. The topological space $X^{\an}$ satisfies nice properties: it is Hausdorff, locally path-connected, locally contractible, and is compact if and only $X/K$ is proper.
\\In the sequel, the cases relevant to us will be the following: either $K=\C$ endowed with the Euclidean absolute value, either $K$ is a discretely valued (hence non-archimedean) field of equicharacteristic zero. In the former case, the Gelfand-Mazur theorem ensures that the Berkovich analytification $X^{\an}$ of $X$ is nothing but the associated complex analytic space, that we denote by $X^{\hol}$, so that we will now focus on the latter.
\subsection{Models and Berkovich skeleta}\label{section DVF}
Let $K$ be a complete, discretely-valued field of equicharacteristic zero, with valuation $v$; by Cohen's structure theorem we may assume that $K =k((t))$ is the field of Laurent series over its residue field, endowed with the valuation $v= \ord_0$. We denote by $R= k[[t]]$ its valuation ring.
\\Throughout this section, we let $X$ be a separated $K$-scheme of finite type, and write $n = \dim(X)$. We will call \emph{model} of $X$ a flat, separated $R$-scheme $\X$, together with an isomorphism of $K$-schemes $\X \times_{R} K \simeq X$. We will denote by $\X_0 := \X \times_{R} k$ its special fiber and by $\Div_0(\X)$ the group of Weil divisors on $\X$ supported on the special fiber.
\\Given a model $\X$ of $X$ and writing $\X_0 = \sum_{i \in I} a_i D_i$ as the sum of its irreducible components, we say $Y \subset \X_0$ is a \emph{stratum} if there exists $J \subset I$ such that $Y$ is a (nonempty) connected component of $D_J := \cap_{j \in J} D_j$.
\\If $\X$, $\X'$ are two models of $X$, we will say that $\X'$ dominates $\X$ if there exists a (necessarily unique) $R$-morphism $f: \X' \fl \X$ whose base change to $K$ induces the identity on $X$. 
\\Assume that $X/K$ is proper, and let $\X/R$ be a proper model of $X$. By the valuative criterion of properness, any point $x \in X^{\an}$ admits a center $c_\X(x)$ on $\X$ (see \cite[§2.2.2]{MN}), which is furthermore contained in $\X_0$ since $v_x(t) =1$. The map $c_{\X} : X^{\an} \fl \X_0$ is anticontinuous, which means that the preimage of a closed subset of $\X_0$ by $c_{\X}$ is open in $X^{\an}$.
\\In the case where $X/K$ is smooth, we say a model $\X/R$ has simple normal crossing singularities if $\X/R$ is smooth, and the special fiber $\X_0$ is a simple normal crossing divisor inside $\X$. Such models always exists when $X/K$ is projective, by Hironaka's theorem on resolution of singularities; more precisely, any model $\X/R$ can be dominated by an snc model.
For reasons that will appear more clearly later on, we will sometimes need to work with a broader class of models:
\begin{defn} \label{def dlt}
Let $\X /R$ be a model of $X$. We say that $\X$ is a \emph{dlt} (divisorially log terminal) model of $X$ if the following conditions hold:
\begin{itemize}
\itemsep0em
    \item[-] the pair $(\X, \X_{0, \red})$ is log canonical in the sense of the Minimal Model Program (see \cite{KM});
    \item[-] the pair $(\X, \X_{0, \red})$ is simple normal crossing at the generic points of log canonical centers of $(\X, \X_{0,\red})$.
\end{itemize}
Moreover, a dlt model $\X$ is \emph{good} if each irreducible component of $\X_{0,\red}$ is $\Q$-Cartier.
\end{defn}
We will not give a precise definition of log canonical centers here, and refer the reader to \cite{KM}. However, if $\X$ is a dlt model and $X$ is defined over an algebraic curve - which is the most relevant case for applications - then the log canonical centers of $(\X, \X_{0, \red})$ are precisely the strata of $\X_0$ by \cite[4.16]{Kollar2013}, so that a dlt model $\X$ is simple normal crossing at the generic points of the strata of $\X_0$.
\\To every dlt model $\X$ of $X$, with special fiber $\X_0 = \sum_{i \in I} a_i D_i$, we can associate a cell complex that encodes the combinatorics of the intersections of the irreducible components of $\X_0$:
\begin{defn} \label{def dual cpx}
Let $\X$ be a dlt model of $X$. 
To each stratum $Y$ of $\X_0$ which is a connected component of $D_J$, we associate a simplex:
$$ \tau_Y = \{ w \in \R_{\geqslant 0}^{|J|} \, | \sum_{j \in J} a_j w_j =1 \}.$$
The cell complex $\mathcal{D}(\X_0)$ is defined as the union of the above $\tau_Y$, with the following incidence relations: $\tau_Y$ is a face of $\tau_{Y'}$ if and only if $Y' \subset Y$.
\end{defn}
Given any dlt model $\X$ of $X$ over $R$, there exists a natural embedding $i_{\X}$ of the dual complex $\mathcal{D}(\X_0)$ into $X^{\an}$, given as follows. 
The vertices $v_i$ of $\mathcal{D}(\X_0)$ are in one-to-one correspondence with irreducible components $D_i$ of the special fiber $\X_0 = \sum_{i \in I} a_i D_i$, so that we set $$i_{\X}(v_{i}) = v_{D_i} := a^{-1}_i \ord_{D_i},$$ where the valuation $\ord_{D_i}$ associates to a meromorphic function $f \in \mathcal{K}(X) \simeq \mathcal{K}(\X)$ its vanishing order along $D_i$ - the normalisation by $a^{-1}_i$ ensuring that $v_{D_i}(t) = 1$. The higher-dimensional faces are embedded using the following:
\begin{prop}{\cite[Proposition 2.4.4]{MN}} \label{prop casi mono}
Let $Y \subset \X_0$ be a stratum which is a connected component of $D_J$ for a $J \subset I$, with generic point $\eta$, and furthermore fix a local equation $z_j \in \gO_{\X,\eta}$ for $D_j$, for any $j \in J$. 
\\For any $w \in \tau_Y = \{ w \in \R^{|J|}_{\geqslant 0} \,| \sum_{j \in J} a_j w_j =1 \}$, there exists a unique valuation $v_w$ on $\mathcal{K}(X)$ on such that for every $f \in \gO_{\X,\eta}$, with expansion $f = \sum_{\beta \in \N^{|J|}} c_{\beta} z^{\beta}$ (with $c_{\beta}$ either zero or unit), we have:
$$v_w(f) = \min \{ \langle w , \beta \rangle \,| \beta \in \N^{|J|},  c_{\beta} \neq 0 \},$$
where $\langle \; , \; \rangle$ is the usual scalar product on $\R^{|J|}$.
The above valuation is called the quasi-monomial valuation associated with the data $(Y, w)$. 
\end{prop}
The map:
\begin{align*}
    i_\X: & \quad \mathcal{D}(\X_0) \rightarrow X^{\text{an}} \\
    & \quad \tau_Y \ni w \mapsto v_w
\end{align*}
is continuous and injective, and its image is called the \emph{skeleton} of $\X$, written as $\Sk(\X) \subset X^{\an}$.  By compactness of $\mathcal{D}(\X_0)$, the map $i_{\X}$ induces a homeomorphism between $\mathcal{D}(\X_0)$ and $\Sk(\X)$, so that we will sometimes abusively identify $\mathcal{D}(\X_0)$ with $\Sk(\X)$, which is therefore a cell complex of dimension at most $\dim X$.
\begin{defn}
Let $Y$ be a stratum of $\X_0$. We define $\Star(\tau_Y)$ as the union of open faces in $\Sk(\X)$ whose closure contains $\tau_Y$.
\end{defn}
Let us now assume that $\X$ is a \emph{good} dlt model of $X$, i.e. that the irreducible components of $\X_0$ are $\Q$-Cartier. We can now define a retraction for the inclusion $\Sk(\X) \subset X^{\an}$ as follows: for any $v \in X^{\an}$, there exists a minimal stratum $Y \subseteq  \cap_{j \in J} D_j$ of $\X_0$ such that the center $c_{\X}(v)$ of $v$ is contained in $Y$. Since $\X$ is good, for each $j \in J$, there exists $e_j >0$ such that $e_j D_j$ is Cartier at $c_{\X}(v)$. We now associate to $v$ the quasi-monomial valuation $\rho_{\X}(v)$ corresponding to the data $(Y, w)$ with $w_j = e_j^{-1}v(z_j)$, where $z_j$ is a local equation of $e_j D_j$ at the generic point of $c_{\X}(v)$. This should be seen as a monomial approximation of the valuation $v$ at the generic point of $Y$, with respect to the model $\X$ (which is snc there).
\begin{defn}
The above map $\rho_{\X} : X^{\an} \longrightarrow \Sk(\X)$ is the Berkovich retraction associated with the model $\X/R$.
\end{defn}
The Berkovich retraction is continuous and restricts to the identity on $\Sk(\X)$; it furthermore induces a homotopy equivalence between $X^{\an}$ and $\Sk(\X)$ by \cite{Th}, \cite{Ber2}.
\begin{rem} \label{retract local}
The Berkovich retraction is a local construction on $\X$, in the following sense. Let $Y$ be a stratum of $\X_0$, the formal completion $\fX_Y := \widehat{\X_{/Y}}$ of $\X$ along $Y$ is a formal $R$-scheme, and admits a generic fiber $\mathfrak{X}^{\eta}_Y$ in the sense of Berkovich, which is a Berkovich analytic space and can be explicitly described as the following open subset of $X^{\an}$:
$$\mathfrak{X}^{\eta}_Y = \{ x \in X^{\an} \lvert \, c_{\X}(v_x) \in Y \} = \rho^{-1}_{\X}(\Star(\tau_Y)).$$
The global construction from above can in fact be made on $\mathfrak{X}^{\eta}_Y$ \cite{Ber2}, so that there exists a retraction: 
$$\rho_Y : \mathfrak{X}^{\eta}_Y \fl \Star(\tau_Y), $$
which coincides with the restriction of the retraction $\rho_{\X}$. Thus, the restriction of $\rho_{\X}$ over $\Star(\tau_Y)$ only depends on the formal completion $\widehat{\X_{/Y}}$.
\end{rem}
\subsection{Hybrid spaces} \label{sec hyb}
Recall that a Banach ring $(A, \lvert \cdot \rvert)$ is a ring endowed with a submultiplicative norm, such that $A$ is a complete metric space with respect to the induced distance.
\begin{defn} Let $A$ be a Banach ring.
\\The Berkovich spectrum $\M(A)$ is the set whose points $x \in \M(A)$ are multiplicative semi-norms $\lvert \cdot \rvert_x : A \fl \R_{\ge 0}$ satisfying $\lvert \cdot \rvert_x \le \lvert \cdot \rvert$. 
\\It is equipped with the topology of pointwise convergence on $A$, which makes it into a non-empty Hausdorff compact topological space by \cite{Ber}.
\end{defn}
For instance, if $A=k$ is a complete valued field, then $\M(k)$ is reduced to the point $\lvert \cdot \rvert$. Any point $x \in \M(A)$ admits a residue field $\mathscr{H}(x)$, which is a complete valued field defined as follows. Since the semi-norm $\lvert \cdot \rvert_x$ is multiplicative, its kernel $\mathfrak{p}_x$ is a prime ideal of $A$, we write $\kappa(x) = \Frac(A/\mathfrak{p}_x)$ the algebraic residue field. The semi-norm $\lvert \cdot \rvert_x$ descends as an absolute value to $\kappa(x)$, and we denote by $\mathscr{H}(x)$ the associated completion.
\\If $X$ is a scheme of finite type over $A$, one can define an \emph{analytification functor} $X \mapsto X^{\text{an}}$ from the category of $A$-schemes of finite type to the category of $A$-analytic spaces, constructed in \cite{LP}. Any $A$-analytic space comes with a sheaf of analytic functions, as defined in \cite[def. 1.5.3]{Ber}. As in the case where $A$ is a valued field, the space $X^{\an}$ satisfies nice topological properties: if $X/A$ is separated, then $X^{\an}$ is Hausdorff; and $X^{\an}$ is compact if and only $X/A$ is proper.
\begin{defn}
Let $\C^{\hyb}$ be the Banach ring obtained by equipping the field $\C$ with the norm $\lvert \cdot \rvert_{\hyb}$, defined for non-zero $z \in \C$ by:
$$ \lvert z \rvert_{\hyb} = \max \{ 1 , \lvert z \rvert \}.$$
\end{defn}
One can show \cite[ex. 1.1.15]{LP} that the elements of $\M(\C^{\text{hyb}})$ are of the form $\lvert \cdot \rvert^{\lambda}$, for $\lambda \in [0,1]$, where $\lvert \cdot \rvert^0 = \lvert \cdot \rvert_0$ denotes the trivial absolute value on $\C$. This yields a homeomorphism $\lambda : \M(\C^{\text{hyb}}) \xrightarrow{\sim} [0,1]$.
\\Thus, if $Z$ is a scheme of finite type over $\C$, its analytification with respect to $\lvert \cdot \rvert_{\text{hyb}}$, which we denote by $Z^{\text{hyb}}$, comes with a structure morphism $\pi : Z^{\text{hyb}} \fl [0,1]$. Moreover, for any $\lambda \in [0,1]$ the fiber $\pi^{-1}(\lambda)$ is canonically homeomorphic to the analytification of $Z$ with respect to the absolute value $\lvert \cdot \rvert^{\lambda}$ - notably, if $\lambda>0$, then $\pi^{-1}(\lambda)$ is homeomorphic to $Z^{\hol}$ by the Gelfand-Mazur theorem, and the analytic structure is the same up to a scaling factor. There is even a homeomorphism: 
$$p: \pi^{-1}((0,1]) \xrightarrow{\sim} (0,1] \times Z^{\hol},$$
compatible with the projections to $(0, 1]$.
\\On the other hand, the fiber $\pi^{-1}(0)$ is homeomorphic to the analytification $Z^{\text{an}}_0$ of $Z$ with respect to the trivial absolute value on $\C$. As a result, the space $Z^{\text{hyb}}$ allows us to see the rescalings of the complex variety $Z^{\hol}$ degenerate to its trivially-valued counterpart.
\\We now want to perform a similar construction for degenerations of complex varieties. Let $X \xrightarrow{\pi} \D^*$ be a quasi-projective, holomorphic family of $n$-dimensional complex manifolds, where $\D^* =\{ \lvert t \rvert <1 \}$ is the punctured unit disk in $\C$. This allows us to view (the base change of) $X$ as a quasi-projective scheme over the non-archimedean field $K = \C((t))$ of Laurent series, we will write $X^{\an}$ for the Berkovich analytification of $X$ with respect to the $t$-adic valuation on $K$.
\\We fix a radius $r \in (0,1)$, and consider the following Banach ring, which we call the \emph{hybrid ring}:
$$A_r = \{ f = \sum_{n \in \Z} a_n t^n \in K  \; / \; \lVert f \rVert_{\text{hyb}} := \sum_{n} \lVert a_n \rVert_{\text{hyb}} \; r^n < \infty \}.$$
\begin{lem}{(\cite[prop. A.4]{BJ})} \label{cercle hyb}
\\The map $\tau: t \mapsto \lvert \cdot \rvert_t$ defined by:
\[ \lvert f \rvert_t = \left\{ \begin{array}{ll} r^{\ord_0(f)} & \mbox{if $t =0$},\\
 r^{\log \lvert f(t) \rvert/ \log \rvert t \rvert} & \mbox{if $t \neq 0$}
\end{array} \right. \]
\\for $f \in A_r$ induces a homeomorphism from $\bar{\D}_r$ to $\M(A_r)$.
\end{lem}
The Berkovich spectrum $\M(A_r)$ is called the hybrid circle and denoted by $C^{\hyb}(r)$, and the above lemma identifies the hybrid circle of radius $r$ with the closed disk of radius $r$. In particular, up to scaling factors, the image of $0 \in \bar{\D}_r$ is the $t$-adic absolute value on $A_r \subset \C((t))$, while $\log \lvert \cdot \rvert_t$ is the Euclidean norm of the evaluation of $f$ at $t$.
As a result, the Berkovich analytification $X^{\hyb} := (X_{A_r})^{\an}$ of the base change $X_{A_r}$ fibers over $\bar{\D}_r$, and allows us to see the complex manifold $X$ degenerate to its non-archimedean analytification, as a consequence of the following:
\begin{prop}{\cite[thm. 1.2]{Fav}} \label{topo hyb}
Let $X$ be a degeneration of complex manifolds, $X_{A_r}$ the associated $A_r$-scheme, and denote the associated hybrid space by $\pi_{\hyb} : X_r^{\hyb} \fl C^{\hyb}(r)$. Then:
\begin{itemize}
\item $\pi_{\hyb}^{-1}(0)$ can be canonically identified with $X^{\an}$,
\item there exists a homeomorphism $\beta : X_{ | \D^*_r} \xrightarrow{\sim} \pi_{\hyb}^{-1}(\tau(\D^*_r))$, satisfying $\pi_{\hyb} \circ \beta = \tau \circ \pi$,
\item if $\phi$ is a rational function on $X_{A_r}$, then $\lvert \phi (\beta(z)) \rvert = \lvert \phi (z) \rvert^{\frac{\log r}{\log \lvert t \rvert}}$ for $z$ not in the indeterminacy locus of $\phi$.
\end{itemize}
\end{prop}
\subsection{Pluripotential theory}
Let $K$ be a complete valued field and $X$ a projective scheme over $K$, endowed with a semi-ample line bundle $L$. By the Gelfand-Mazur theorem, either $K = \R$ or $\C$ with (a power of) the Euclidean absolute value, either $K$ is non-archimedean.
\\One defines a continuous metric $\phi$ on the Berkovich analytification of $(X, L)$ as an operation that associates to any local section $s \in H^0(U, L)$, where $U \subset X$ is a Zariski open, a continuous function:
$$\lVert s \rVert_{\phi} : U^{\an} \fl \R,$$
compatible with restriction of sections and such that $\lVert fs \rVert_{\phi} = \lvert f \rvert \lVert s \rVert_{\phi}$ whenever $f \in \gO_X(U)$. We then want to define a class of (continuous) plurisubharmonic metrics on $L$ as in the complex case, where a continuous metric $\phi$ on $L$ is said to be psh if its curvature current $dd^c \phi := \frac{i}{2\pi} \ddbar \phi$ is semi-positive. While the work of Chambert-Loir - Ducros \cite{CLD} provides a notion of semi-positive currents on Berkovich spaces, and hence a local definition of psh metrics, we will rather adopt a global point of view here, following \cite{BFJ2}, \cite{BJv4}, \cite{PShyb}.
\\Given an integer $m \ge 1$, a family $(s_0,...,s_N)$ of global sections of $mL$ without common zeroes, and real constants $c_0,...,c_N$, the metric:
$$\phi = m^{-1} \max_{i=0,...,N} (\log \lvert s_i \rvert +c_i)$$
is called a \emph{tropical Fubini-Study} metric on $L$. The class of continuous psh metrics is then defined as the closure of the class of Fubini-Study metrics for uniform convergence:
\begin{defn} \label{def psh} A continuous metric $\phi$ on $L$ is plurisubharmonic (psh for short) if it can be written as the uniform limit of a net of tropical Fubini-Study metrics on $L$. We write $\CPSH(X, L)$ for the set of continuous plurisubharmonic metrics on $L$.
\end{defn}
For instance, if $K = \C$ endowed with the Euclidean absolute value and $X$ is a normal complex variety, then it follows mostly from Demailly's regularization theorem \cite{Dem} that this is consistent with the usual definition of a psh metric on $L$, see \cite[thm. 2.11]{PShyb}. Similarly, if $K$ is a complete non-archimedean field, then the above definition is consistent with the local Chambert-Loir - Ducros definition of psh metrics on $L$ by \cite[thm. 7.14]{BE}.
\\Let $L_1,...,L_n$ be $n$ semi-ample line bundles on $X$, where $n = \dim X$. One of the central objects of pluripotential theory is the Monge-Ampère operator, associating to a tuple $(\phi_1,..., \phi_n)$ of continuous psh metrics on the $L_i$ a measure $\MA(\phi_1,..., \phi_n)$. Assume for a moment that $K = \C$ and that $X$ is smooth, by the seminal work of Bedford-Taylor \cite{BT} the mixed Monge-Ampère pairing:
$$(\phi_1,...,\phi_n) \mapsto dd^c \phi_1 \wedge ... \wedge dd^c \phi_n,$$
\emph{a priori} defined when each $\phi_i$ is a smooth Hermitian metric on $L_i$, extends in a unique way to psh, locally bounded metrics - hence in particular to continuous psh metrics.
\\We now assume $K$ is non-archimedean: the role of smooth metrics is now played by tropical Fubini-Study metrics. Given a tuple $(\phi_1,..., \phi_n)$ of tropical Fubini-Study metrics on the $L_i$, one can define their mixed Monge-Ampère measure $\MA(\phi_1,...,\phi_n)$, with similar properties as in the complex analytic case - the pairing $(\phi_1,...,\phi_n) \mapsto \MA(\phi_1,...,\phi_n)$ is for instance symmetric and multilinear, and the total mass:
$$\int_{X^{\an}} \MA(\phi_1,...,\phi_n) = (L_1 \cdot ... \cdot L_n).$$
The Monge-Ampère measure can then be extended to more general metrics:
\begin{theo}{\cite{CL}, \cite[thm. 8.4]{BE}} 
\\Let $K$ be a complete valued field, $X/K$ an $n$-dimensional projective scheme and $L_1$,...,$L_n$ line bundles on $X$. Then the Monge-Ampère operator:
$$(\phi_1,..., \phi_n) \mapsto \MA(\phi_1,...,\phi_n)$$
admits a unique extension to continuous psh metrics on the $L_i$. The extension is furthermore continuous with respect to the topology of uniform convergence and the weak topology of Radon measures.
\end{theo}
Let $X$ be a smooth $n$-dimensional complex manifold, and $L$ an ample line bundle on $X$. Then Yau's celebrated solution to the Calabi conjecture \cite{Y} asserts that for any smooth volume form $\mu$ on $X$, normalized to have total mass $1$, there exists a unique (up to an additive constant) smooth, positive definite metric $\phi$ on $L$ such that:
$$(dd^c \phi)^n = (L^n) \mu.$$
\\Throughout the years, various generalizations of the above theorem in a more singular setting have appeared in the literature: let us simply mention Kołodziej's result \cite{Kolo}, that states that under the same assumptions on $(X, L)$, then the statement of the above theorem holds for a much wider range of probability measures on $X$ (for instance, measures $\mu$ with $L^p$-density for some $p>1$), when we don't require for the solution to be smooth - here the solution is a continuous psh metric $\phi \in \CPSH(X, L)$, and the equality $(dd^c \phi)^n = \mu$ is understood in the sense of \cite{BT}.
\\We now let $K= k((t))$ be a discretely-valued field of equicharacteristic zero. The following result can be understood as an analog of Kołodziej's result over $K$:
\begin{theo}{(\cite[thm. A]{BFJ1}, \cite{BGJKM})} \label{theo sol MA}
\\Let $(X, L)$ be a smooth polarized variety over $K$. Let $\mu$ be a probability measure on $X^{\an}$, supported on the skeleton of some snc $R$-model of $X$. Then there exists a unique (up to an additive constant) continuous, semi-positive metric $\phi$ on $L$ satisying the non-archimedean Monge-Ampère equation:
$$\MA(\phi) = \mu.$$
\end{theo}
The above theorem was proved when $X$ is defined over the function field of a curve over $k$ in \cite{BFJ1}, and then extended to varieties over non-archimedean fields of residual characteristic zero in \cite{BGJKM}.
\\Let now $X \fl \D^*$ a projective degeneration of complex manifolds, which we may view as a projective scheme over the Banach ring $A_r$ from the previous section. We may view the associated hybrid space $X^{\hyb}$ as the Berkovich analytification of the latter scheme, so that one can define tropical Fubini-Study metrics on $L^{\hyb}$ in a similar manner to the case of a field: those are continuous metrics of the form:
$$\phi = m^{-1} \max_{i=0,...,N} (\log \lvert s_i \rvert +c_i),$$
where $m \ge 1$, the global sections $s_0,...,s_N \in H^0(X, mL)$ have no common zeroes and $c_0,...,c_N$ are real constants. A continuous, plurisubharmonic metric on $L^{\hyb}$ is now by definition a continuous family of metrics on the fibers of $\pi_{\hyb}$, that can be written as the uniform limit of a net of tropical Fubini-Study metrics, see \cite[§2.4]{PShyb}. In particular, for any $t \in C^{\hyb}(r)$, the restriction $\phi_t$ of $\phi$ to the analytic space (over a field) $\pi^{-1}_{\hyb}(t)$ is a continuous plurisubharmonic metric, in the sense of definition \ref{def psh}.
\\The following theorem asserts than given a continuous, plurisubharmonic metric on $(X^{\hyb}, L^{\hyb})$, the associated family of Monge-Ampère measures is weakly continous on $X^{\hyb}$:
\begin{theo}{\cite[thm. 4.2]{Fav}, \cite[thm. 4.10]{PShyb}} \label{theo cont MA}
\\Let $(X, L) \fl \D^*$ be a polarized degeneration of complex manifolds, and $\phi$ a continuous psh metric on $L^{\hyb}$. Then the following convergence of measures:
$$(dd^c\phi_t)^n \xrightarrow[t \rightarrow 0]{} \MA(\phi_0)$$
holds on $X^{\hyb}$ in the weak sense.
\end{theo}
\section{Hybrid toric metrics} \label{sec toric metrics}
Let $(Z, L)$ is a polarized toric variety over $\C$. In this section, we investigate toric metrics on $(Z^{\hyb}, L^{\hyb})$. When working over $\C$ equipped with either the archimedean or trivial absolute value, semi-positive toric metrics on $(Z^{\an}, L^{\an})$ are well-know to correspond to a certain class of convex functions on the vector space $N_{\R}$ where the fan of $Z$ lives. The main result of this section, theorem \ref{theo toric}, states that semi-positive metrics on $(Z^{\hyb}, L^{\hyb})$ that are \emph{fiberwise-toric}, correspond to families parametrized by $\M(\C^{\hyb}) \simeq [0,1]$ of convex functions on $N_{\R}$ as above, which vary in a convex way with respect to the coordinate $\lambda$ on the base.
\\Throughout this section, $N$ denotes a lattice of rank $n$, $M = \Hom(N, \Z)$ its dual and $N_{\R} = N \otimes_{\Z} \R$, $M_{\R} = M \otimes_{\Z} \R= \Hom_{\R} (N_{\R}, \R)$. We denote by $\langle \cdot, \cdot \rangle$ the duality pairing $M_{\R} \times N_{\R} \fl \R$.
\subsection{Convex geometry}
We say a function $\phi : N_{\R} \fl \R$ is piecewise-affine if there exists a rational polytopal subdivision $\mathcal{T}$ of $N_{\R}$ such that $\phi$ is affine on each face of the subdivision; it is said to be rational piecewise-affine if each affine piece of $\phi$ has rational slopes (note that we allow the constant term to be any real number). We will sometimes say $\phi$ is \emph{adapted} to the subdivision $\mathcal{T}$.
\\A (rational) convex polytope $P \subset M_{\R}$ is the convex hull in $M_{\R}$ of finitely many points in $M_{\Q}$. The polytope is furthermore integral if we can choose the points to be in $M$.
\\If $P \subset M_{\R}$ is a convex polytope, a \emph{face} of $P$ is any intersection of $P$ with a halfspace such that none of the interior points of the polytope lie on the boundary of the halfspace.
The normal fan $\Sigma = \Sigma_P$ is the complete fan in $N_{\R}$ whose cones are the normal cones to the faces of $P$. More explicitly, $\Sigma = \{ \sigma_F \}_{F \; \text{face of} \; P}$, where:
$$\sigma_F = \{ x \in N_{\R} / \forall u \in F, \langle u, x \rangle = \max_P u \}.$$
\begin{defn} Let $P \subset M_{\R}$ be a convex polytope. Its support function $\Psi_P : N_{\R} \fl \R$ is defined by the formula:
$$\Psi_P(x)  = \sup_{u \in P} \langle u, x \rangle.$$
It is a rational piecewise-linear convex function on $N_{\R}$, linear on each face of the normal fan $\Sigma_P$.
\end{defn}
The fact that $\Psi_P$ is piecewise-linear follows from the fact that for fixed $x$, the function $u \mapsto \langle u, x \rangle$ reaches its maximum at a vertex of $P$, hence the equality:
$$\Psi_P(x)  = \max_{u \in V} \langle u, x \rangle,$$
where $V \subset P$ is the set of vertices of $P$. Note that the data of $(\Sigma_P, \Psi_P)$ recovers the polytope $P$ uniquely.
\begin{defn}
Let $\phi : N_{\R} \fl \R$ be a convex function. We define the Legendre transform $\phi^*$ of $\phi$ by the formula:
$$\phi^*(u) = \sup_{x \in N_{\R}} \big( \langle u, x \rangle -\phi(x) \big),$$
for $u \in M_{\R}$, and we write its domain $P(\phi) :=\{ \phi^* < + \infty \} \subset M_{\R}$.
\end{defn}
\begin{defn} \label{def conv adm} Let $P \subset M_{\R}$ be a convex polytope. We say that the convex function $\phi : N_{\R} \fl \R$ is $P$-admissible if and only the equality $P(\phi)= P$ holds. This is equivalent to the following condition:
$$\sup_{x \in N_{\R}} \lvert \phi(x) - \Psi_P(x) \rvert < + \infty.$$
We will write $\Ad_P(N_{\R})$ or simply $\Ad_P$ for the set of $P$-admissible convex functions.
\end{defn}
\begin{ex} \label{ex vertices}
Let $J \subset P\cap M_{\Q}$ be a finite subset, and $\phi : N_{\R} \fl \R$ be the function defined by:
$$\phi(x) = \max_{j \in J} (\langle u_j, x \rangle + c_j),$$
where the $c_j$'s are real constants. Then $\phi \in \Ad_P(N_{\R})$ if and only if $J$ contains the set of vertices of $P$.
\end{ex}
\begin{prop}
Let $\phi \in \Ad_P(N_{\R})$. Then the Legendre transform $\phi^* : P \fl \R$ is convex and continuous. Moreover, the map $\phi \mapsto \phi^*$ induces a bijection from $\Ad_P(N_{\R})$ to the set of continuous convex functions on $P$, whose inverse map is given by $\beta \mapsto \beta^*$, where:
$$\beta^*(x) = \sup_{u \in P} \big(\langle u, x \rangle - \beta(u) \big)$$
for $\beta : P \fl \R$ a continuous convex function.
\end{prop}
\begin{proof}
If $\phi$ is $P$-admissible, its Legendre transform is upper semi-continuous on $P$ as a supremum of continuous functions, while it follows from \cite{How} that the Legendre transform $\phi^*$ is convex on $P$, hence lower semi-continuous on $P$. This proves the first item, and the second item follows from \cite[thm. 12.2]{Rock}.
\end{proof}
The following proposition states that $P$-admissible functions can be monotonously approximated by rational piecewise-affine functions:
\begin{prop} \label{lem cvx} Let $P$ be a convex polytope in $M_{\R}$, and let $\phi \in \Ad_P(N_{\R})$ be a $P$-admissible convex function.
\\Then there exists a decreasing sequence $(\phi_j)_{j \in \N}$ of rational piecewise-affine, $P$-admissible convex functions converging pointwise to $\phi$.
\end{prop}
\begin{proof}
Let $\beta : P \fl \R$ be the Legendre transform of $\phi$, and let $(x_j)_{j \in \N}$ be a sequence in $N_{\Q}$ that is dense in $N_{\R}$. For each $j \in \N$, let $(r_{l,j})_{l \in \N}$ be a sequence of rational numbers decreasing to $\phi(x_j)$, uniformly with respect to $j$ (uniformity can be achieved by a diagonal extraction argument). We set:
$$\beta_j(u) = \max_{l \le j}\big( \langle u, x_l \rangle - r_{l,j} \big),$$
so that as $j \rightarrow \infty$, the function $\beta_j$ is close to the maximum of the affine functions with gradients in $\{ x_l \}_{l \le j}$ that are cutting out supporting hyperplanes of the graph of $\beta$. It is clear that $\beta_j$ is rational piecewise-linear, and increasing pointwise to:
$$\beta^{**}(u) = \sup_{x \in N_{\R}} \big( \langle u, x \rangle - \phi(x) \big),$$
which is none other than $\beta$. Continuity of $\beta$ furthermore implies that the convergence is uniform, by Dini's lemma.
\\We now set: 
$$\phi_j(x) = \beta_j^*(x)= \sup_{u \in P} \big(\langle u, x \rangle - \beta_j(u)\big).$$
It defines an decreasing sequence of $P$-admissible rational piecewise-affine convex functions on $N_{\R}$, by the following lemma. The fact that the $\phi_j$'s decrease (uniformly) to $\phi$ is now a consequence of the monotonicity of the Legendre transform.
\end{proof}
\begin{lem} \label{lem Leg pa} Let $P \subset M_{\R}$ a convex polytope and $\beta : P \fl \R$ a function that is piecewise-affine, adapted to a rational polyhedral subdivision of $P$. Then the Legendre transform $\beta^*$, defined for $x \in N_{\R}$ by:
$$\beta^*(x) = \sup_{u \in P} \big( \langle u, x \rangle - \beta(u) \big)$$
can be written as:
$$\beta^* = \max_{i \in I} (u_i + c_i),$$
where $I$ is a finite set, $u_i \in P\cap M_{\Q}$ and $c_i \in \R$. \\Moreover, $\beta^* \in \Ad_P(N_{\R})$.
\end{lem}
\begin{proof} Let $\mathcal{T}$ be a rational polytopal decomposition of $P$ such that $\beta$ is affine on each face of $\mathcal{T}$ (which exists as follows for instance from \cite[lem. 2.5.3]{BPS}), and write the set $E$ of vertices of $\mathcal{T}$ as $E=\{ u_i, i \in I \}$. Then we have:
$$\beta^*(x) = \max_{i \in I} \big(\langle u_i, x \rangle - \beta(u_i)\big),$$
as for fixed $x$, the function $u \mapsto \big( \langle u, x \rangle - \beta(u) \big)$ is affine on each face of $\mathcal{T}$, hence achieves its maximum at a vertex of $\mathcal{T}$.
\\The second item follows from example \ref{ex vertices}, since $E$ contains the set of vertices of $P$ by construction.
\end{proof}
The following proposition will be used during the proof of theorem \ref{theo toric}:
\begin{prop} \label{convex hybrid} Let $P \subset M_{\R}$ be a convex polytope with support function $\Psi_P$ and normal fan $\Sigma$, and let $\Phi : N_{\R} \times [0,1] \fl \R$ be a continuous convex function, such that the function $(\Phi- \Psi_P)$ extends continuously to $N_{\Sigma} \times [0,1]$.
\\Then there exists a sequence $(\Phi_j)_{j \in \N}$ of piecewise-affine convex functions on $N_{\R} \times [0,1]$, such that: 
\begin{itemize}
\item for all $j \in \N$, the function $(\Phi_j- \Psi_P)$ extends continuously to $N_{\Sigma} \times [0,1]$;
\item the sequence $(\Phi_j)_j$ decreases to $\Phi$ over $N_{\Sigma} \times [0,1]$;
\item each $\Phi_j$ has rational slopes in the $N_{\R}$-direction, i.e. we may write:
$$\Phi_j(x, \lambda) = \max_{\alpha \in A} \big( \langle u_{\alpha}, x \rangle + b_{\alpha} \lambda + c_{\alpha} \big),$$
where the $u_{\alpha} \in P\cap M_{\Q}$, and the $a_{\alpha}$, $b_{\alpha} \in \R$.
\end{itemize}
\end{prop}
\begin{proof} Let $\Phi^* : P \times \R \fl \R$ be the Legendre transform of $\Phi$, then by lemma \ref{lem ext toric} the function $(\Phi^* - \max\{0, a \})$ extends continuously to $P \times [-\infty, + \infty]$, where $a$ is the coordinate on $\R$.
\\Let $((x_j, \lambda_j))_{j \in \N}$ be a dense sequence of rational points in $N_{\R} \times [0,1]$, and let $(r_{j,l})_{l \in \N}$ be a sequence of rational numbers decreasing to $\phi(x_j, \lambda_j)$, uniformly with respect to $j$. We furthermore assume that $\lambda_0 =0$ and $\lambda_1 =1$.
\\We set:
$$\beta_j(u,a) = \max_{l \le j} \big( \langle u, x_l \rangle + a\lambda_l -r_{j,l} \big)$$
for $(u,a) \in P \times \R$, it is a sequence of piecewise-affine convex functions increasing pointwise to $\Phi^*$. The convergence also holds on $P \times [- \infty, + \infty]$ after substracting $\max \{0, a \}$, since for $j \ge 1$ and $\lvert a \rvert \gg 1$ depending on $u$, we have that $(\beta_j(u,a) - \max\{0, a \})$ does not depend on $a$. As a result, the convergence is in fact uniform, by Dini's lemma.
\\We now take $\Phi_j$ to be the Legendre transform of $\beta_j$, i.e.:
$$\Phi_j(x, \lambda) = \sup_{(u, a) \in P \times \R} \big(\langle u, x \rangle + a\lambda - \beta_j(u,a) \big),$$
the $\Phi_j$'s are of the form described in the statement by the argument from lemma \ref{lem Leg pa}, as $(\beta_j(u,a) - \max\{0, a \})$ does not depend on $a$ for $\lvert a \rvert \gg 1$.
Since the $\beta_j$ increase uniformly to $\Phi^*$, the $\Phi_j$ decrease uniformly to $\Phi^{**}$, which is equal to $\Phi$ by \cite[thm. 12.2]{Rock}. 
\\It remains to prove that $(\Phi_j - \Psi_P)$ extends continuously to $N_{\Sigma} \times [0,1]$. For $\sigma \in \Sigma$, and $m_{\sigma} \in M$ such that $(\Psi_P)_{| \sigma} \equiv m_{\sigma}$, since the convex hull $\Conv(u_{\alpha})_{\alpha} =P$, we have that $\Phi_j(x, \lambda) = \langle m_{\sigma}, x \rangle + b_{j_0} \lambda + c_{j_0}$ for $x \in \sigma$ and $\lvert x \rvert \gg 1$ independent on $\lambda$, where $j_0$ is such that $u_{j_0} = m_{\sigma}$ This proves that $(\Phi_j - \Psi_P)$ extends continuously to $N_{\Sigma} \times [0,1]$, and decreasing convergence happens over $N_{\Sigma} \times [0,1]$ since it is uniform, which concludes the proof.
\end{proof}
\begin{lem} \label{lem ext toric} Let $P \subset N_{\R}$ be a convex polytope, and let $\Phi : N_{\R} \times [0,1] \fl \R$ satisfying the assumptions of prop. \ref{convex hybrid}. Then the Legendre transform:
$$\Phi^* : P \times \R \fl \R$$
is continuous. Moreover, writing $a$ the coordinate on $\R$, the function $(\Phi^* - \max \{0, a \})$ extends continuously to $P \times [- \infty, + \infty]$.
\end{lem}
\begin{proof}
Let $(u, a) \in M_{\R} \times \R$, then:
$$\Phi^*(u, a) = \sup_{N_{\R} \times [0,1]} \big(\langle u, x \rangle + a \lambda - \Phi(x, \lambda) \big)$$
$$ = \sup_{N_{\R} \times [0,1]} \big(\langle u, x \rangle + a \lambda - \Psi_P(x) +(\Psi_P(x) - \Phi(x, \lambda)) \big),$$
which is finite if and only if $u \in P$ since $(\Psi_P- \Phi)$ is continuous, hence bounded on $N_{\Sigma} \times [0,1]$. We also infer that:
$$\Phi^*(u, a) = \sup_{N_{\Sigma} \times [0,1]} \big(\langle u, x \rangle + a \lambda - \Phi(x, \lambda) \big)$$
whenever $u \in P$, since $\Phi(x, \lambda) = - \infty$ for $x \in N_{\Sigma} \setminus N_{\R}$. As a result, $\Phi^*$ is continuous on $P \times \R$ as a fiberwise-supremum over a compact set of a continuous function.
\\We now prove the second item, we assume $a \ge 0$ as the other case is treated in the same way. The convex function $\Psi_u : a \mapsto (\Phi^*(u, a)-a)$ on $\R_{\ge 0}$ can be written as:
$$\Psi_u = \sup_{\lambda \in [0,1]} \big ((\lambda-1)a + \eta_u(\lambda) \big),$$
where $\eta$ is a continuous function of $(u, \lambda)$. As a result, $\Psi_u$ is a bounded convex function, decreasing on $\R_{\ge 0}$, and whose limit at $+ \infty$ is $\sup_{[0,1]} \eta_u$, which is continuous with respect to $u$. Thus, the family $(\Psi(u, a))_{a \ge 0}$ of continuous convex functions on $P$ decreases with respect to $a$ to a continuous function, so that they converge uniformly by Dini's lemma, and $(\Phi^* -a)$ extends continuously to $P \times [0, + \infty]$.
\end{proof}
\subsection{Toric varieties}
Recall that $N$ is a lattice of rank $n$, with dual lattice $M$. We write $\T = \Sp \Z[M]$ the associated split algebraic torus over $\Z$. We furthermore have that $M = \Hom(\T, \G_m)$ is the character lattice of $\T$, so that each $m \in M$ defines a regular invertible function on $\T$, which we denote by $\chi^m$.
\\Let $ \Sigma = \{ \sigma \}_{\sigma \in \Sigma}$ be a fan inside $N_{\R}$; this is a finite collection of strictly convex rational polyhedral cones inside $N_{\R}$, stable under intersection and such that each face of a cone in $\Sigma$ is itself in $\Sigma$.
\\If $k$ is any field,  one can associate to the fan $\Sigma$ a normal $k$-scheme of finite type as follows: for $\sigma \in \Sigma$, set $M_{\sigma} = M \cap \sigma^{\vee}$, which is a finitely generated semi-group by Gordan's lemma. We then define the affine toric variety associated to $\sigma$ by:
$$Z_{\sigma} := \Sp k [M_{\sigma}];$$
one can then see that for $\sigma' \subset \sigma$, the ring $k [M_{\sigma'}]$ is a localization of $k [M_{\sigma}]$, so that the collection of $\{ Z_{\sigma} \}_{\sigma \in \Sigma}$ can be glued to define a variety over $k$:
$$Z_{\Sigma} = \bigcup_{\sigma \in \Sigma} Z_{\sigma}.$$
Since the cone $\{ 0_{N_{\R}} \} \in \Sigma$, the variety $Z_{\Sigma}$ contains the torus $\T = \Sp k[M]$ as a dense open subset. Its complement $\Delta_Z := Z \setminus \T$ is a reduced, anticanonical Weil divisor which we call the toric boundary of $Z$. We write it as the sum of its irreducible components $\De_Z = \sum_{l \in L} Z_l$.
\begin{defn}{\cite[def. 3.3.4]{BPS}} Let $L$ be a line bundle on $Z$. A toric structure on $L$ is the data of an isomorphism $L \simeq \gO_Z(D)$, where $D$ is a toric Cartier divisor; together with a nowhere-vanishing section $s_D$ of $L_{|\T}$, which is a monomial on $\T$ under the previous isomorphism $L_{\T} \simeq \gO_{\T}$.
\\A toric line bundle on $Z$ is a line $L$ endowed with a toric structure.
\end{defn} 
We now assume that $Z$ is smooth, so that each cone $\sigma \in \Sigma$ can be written as:$$\sigma = \sum_{l=1}^{\dim \sigma} \R_{\ge 0} v_l,$$
where $v_l \in N$ are the (linearly independent) primitive generators of the extremal rays of $\sigma$.
\begin{defn} Let $D = \sum_{l \in L} a_l Z_l$ be a toric divisor on $Z$. For each maximal cone $\sigma \in \Sigma$, we let:
$$\Psi_{D} : \sigma \fl \R$$
be the unique linear map such that $\Psi_D(v_l) = a_l$ for each primitive generator $u_l$ of an extremal ray of $\sigma$.
\\This defines a piecewise-linear function:
$$\Psi_D : \Sigma \fl \R,$$
called the support function of $D$. 
\end{defn}
Note that our definition differs from the one in \cite{Fu} by a minus sign - with our convention, $\Psi_D$ is a convex function if and only $\gO_Z(D)$ is globally generated. If $D_1, D_2$ are two toric divisors such that $\gO_Z(D_1) \simeq \gO_Z(D_2)$, then there exists $m \in M$ such that $D_1-D_2 =\xdiv(\chi^m)$, so that the difference $\Psi_{D_1} - \Psi_{D_2} = \langle m, \cdot \rangle$ is a linear form on $N_{\R}$.
\begin{defn} Let $D$ be a toric divisor on $Z$. The associated convex polyhedron is defined by:
$$P_D = \{ u \in M_{\R} \mid \langle u, \cdot \rangle \le \Psi_D \}.$$
\end{defn}
Similarly to the previous remark, shifting $D$ by a principal toric divisor amounts to translating $P_D$ inside $M_{\R}$.
The polytope $P_D$ computes the space of sections of $\gO_Z(D)$:
\begin{prop}{\cite[lem. p. 66]{Fu}} \label{prop toric sections} If $m \in M \cap P_D$, then the monomial $\chi^{-m}$ defines a regular section $s_m = \chi^{-m}  s_D \in H^0(Z, \gO_Z(D))$. Moreover,
$$H^0(Z, \gO_Z(D)) = \bigoplus_{ m \in P_D \cap M} k \cdot s_m.$$
\end{prop}
In the case where $\gO_Z(D)$ is semi-ample (or equivalently, globally generated), the convex function $\Psi_D$ can be written as a finite maximum of linear functions:
$$\Psi_D = \max_{j \in J} \langle m_j, \cdot \rangle,$$
so that $P_D$ is the convex hull of the $\{ m_j, j \in J \}$ and is a convex polytope, whose support function equates the support function $\Psi_D$ of $D$. Hence, we will say that a convex function $\phi : N_{\R} \fl \R$ is $D$-admissible if $(\phi - \Psi_D) =O(1)$, i.e. $\phi \in \Ad_{P_D}(N_{\R})$.
\begin{lem} \label{lem bpf} Let $J \subset (P_D \cap M)$ be a finite set. Then the $s_{m_j}$'s, for $j \in J$, have no common zeroes if and only $J$ contains the set of vertices of $P_D$, if and only $\max_{j \in J} \langle m_j, \cdot \rangle$ is $D$-admissible.
\end{lem}
\begin{proof} The first equivalence is \cite[ex. p. 69]{Fu}, while the second one is example \ref{ex vertices}.
\end{proof}
\subsection{Analytification of toric varieties}\label{toric anal}
Let $A$ be a Banach ring, and let $\T = \Sp A[M]$ be a split algebraic torus over $A$. Its Berkovich analytification $\T^{\an}$ comes with a canonical valuation map:
$$\val : \T^{\an} \fl N_{\R} = \Hom(M, \R),$$
$$ x \mapsto (m \mapsto -\log \lvert \chi^m(x) \rvert).$$
In other words, $\langle m, \val(x) \rangle = v_x(\chi^m)$, hence the notation.
Fixing coordinates $T_1,...,T_n$ on the torus, which amounts to fixing a basis of $M$, the above map is given by:
$$\val(x) = (-\log \lvert T_1(x) \rvert,..., -\log \lvert T_n(x) \rvert) \in \R^n.$$
Now let $\Sigma$ be a fan inside $N_{\R}$, and let $Z_{\Sigma}/A$ be the associated toric $A$-scheme, defined as in a case of a field, by patching together the $\Sp A[M_{\sigma}]$ for $\sigma \in \Sigma$. For instance, if there exists a subfield $k \subset A$, then $Z_{\Sigma, A} = Z_{\Sigma, k} \times_k A$ is the base change of the $k$-toric variety to $A$.
\\We want to define a partial compactification $N_{\Sigma}$ of $N_{\R}$ so that the map $\val$ defined above extends as a continuous map:
$$\val_{\Sigma} : Z^{\an}_{\Sigma} \fl N_{\Sigma}.$$
To that extent, we set $\overline{\R} = \R \cup \{+ \infty \}$, it is a semi-group for the standard addition. 
\\For $\sigma \in \Sigma$ a cone, the corresponding toric affine chart is given by $Z_{\sigma} = \Sp A [M_{\sigma}]$. It follows from the universal property of the ring $A[M_{\sigma}]$ that the map:
$$Z_{\sigma}(A) \fl \Hom_{\sg} (M_{\sigma}, (A, \times)),$$
$$z \mapsto ( m \mapsto \chi^m(z))$$
is a bijection. It is thus natural to set:
$$N_{\sigma} := \Hom_{\sg} (M_{\sigma}, (\overline{\R}, +)),$$
since the semi-group homomorphism $v_A = -\log \lVert \cdot \rVert: A \fl \overline{\R}$ induces naturally a map:
$$\val : Z_{\sigma}(A) \fl N_{\sigma},$$
$$ z \mapsto (m \mapsto - \log \lvert \chi^m(z) \rvert).$$
extending the usual logarithm on the torus (the above expression being defined since $\chi^m$ is regular on $Z_{\sigma}$). We define the topology on $N_{\sigma}$ as the coarsest making the evaluation maps $\ev_m : N_{\sigma} \fl \overline{\R}$ for $m \in M_{\sigma}$ continuous.
\begin{prop}[{\cite[§1.2, prop. 2]{Fu}, \cite[§4.1]{BPS} }] If $\sigma' \subset \sigma$ are two cones of $\Sigma$, then there is a dense open immersion $N_{\sigma'} \hookrightarrow N_{\sigma}$. Thus, the $\{ N_{\sigma} \}_{\sigma \in \Sigma}$ glue together to yield a topological space $N_{\Sigma}$ containing $N_{\R}$, together with a continuous $\val$ map:
$$\val_{\Sigma} : Z^{\an}_{\Sigma} \fl N_{\Sigma}$$
restricting to the above $\val$ on each $Z_{\sigma}$.
\\The topological space $N_{\Sigma}$ is called the tropical toric variety associated to $\Sigma$.
\end{prop}
Note that the above definition of $N_{\Sigma}$ is independent of the base ring $A$.
\\Now assume that $A=(K, v_K)$ is a non-archimedean field. Then the map $\val$ admits a canonical section, defined as follows:
\begin{defn} \label{def Gauss}
Let $x \in N_{\R}$. The valuation $\gamma(x)$ on $K(M)$ defined on elements of $K[M]$ by the formula:
$$v_{\gamma(x)} \bigg( \sum_{m \in M} a_m \chi^m \bigg) = \min_{m \in M} \big( v_K(a_m) + \langle m, x\rangle \big)$$
is called the Gauss point associated to $x$. This defines a continuous embedding:
$$\gamma : N_{\R} \fl \T^{\an},$$
satisfying $\val \circ \gamma = \Id_{N_{\R}}$.
\\Furthermore, for any fan $\Sigma$ in $N_{\R}$, the map $\gamma$ extends as a continuous embedding $\gamma : N_{\Sigma} \fl Z^{\an}_{\Sigma}$, satisfying the property:
$$\val_{\Sigma} \circ \gamma = \Id_{N_{\Sigma}}.$$
\end{defn}
Thus, the tropical toric variety $N_{\Sigma}$ can be naturally realized as a space of monomial valuations, inside $Z_{\Sigma}^{\an}$. The fact that $v_{\gamma(x)}$ is indeed a valuation - i.e. the associated absolute value is multiplicative - follows for instance from \cite[prop. 4.2.12]{BPS}.
\begin{defn} \label{def toric rho}
We write:
$$\rho_{\Sigma} : Z^{\an}_{\Sigma} \fl Z^{\an}_{\Sigma}$$
the composition $\rho_{\Sigma} = \gamma \circ \val_{\Sigma}$, and call its image the (toric) skeleton $\Sk(Z_{\Sigma})$ of $Z_{\Sigma}$.
\end{defn}
In the case $K=\C$ with the Euclidean absolute value, there also exists a canonical section of $\val$, given as follows: since $\T(\C) = \Hom(M, \C^*)$, there is an embedding $\iota : N_{\R} \hookrightarrow \T^{\hol}$ given by:
$$x \longmapsto \iota(x) := \big( m \mapsto e^{-\langle m, x \rangle} \big).$$
It is clear that this map extends as an embedding $\iota : N_{\sigma} \hookrightarrow Z_{\sigma}^{\hol}$, using the real exponential as a section of $-\log \lvert \cdot \rvert.$
In terms of coordinates on $(\C^*)^n$, the map $\iota$ is simply the section of $-\log \lvert \cdot \rvert$ given by the real exponential (up to a sign).
\\We now move to the case $A = \C^{\hyb}$, where we want to patch together the sections from above to construct a section of the continuous map:
$$ \Val := (\val \times \pi) : Z^{\hyb}_{\Sigma} \fl N_{\R} \times \M(A).$$
To that purpose, we define $\iota_{\lambda} : N_{\R} \hookrightarrow \T^{\hol}$ by the formula:
$$x \longmapsto \iota_{\lambda}(x) := \big( m \mapsto e^{-\lambda^{-1} \langle m, x \rangle} \big),$$
and extend it to $N_{\Sigma}$ cone by cone as above; this yields a continuous embedding $\iota_{\lambda} : N_{\Sigma} \hookrightarrow Z_{\Sigma}^{\hol}$. We then compose these rescaled embeddings with the homeomorphism:
$$p : Z^{\hol} \times (0,1] \xrightarrow{\sim} \pi^{-1}((0,1]),$$
to embed $N_{\Sigma}^{\hyb}:= N_{\Sigma} \times [0,1]$ naturally inside the hybrid space:
\begin{prop} \label{iota hyb} Let $\iota_{\hyb} : N_{\Sigma} \times [0,1] \fl Z^{\hyb}$ be defined by the formulas:
$$\iota_{\hyb}(x, \lambda) =p_{\lambda} (\iota_{\lambda}(x)) \; \text{for} \; \lambda \neq 0,$$
$$\iota_{\hyb}(x, 0) = \gamma(x).$$
Then $\iota_{\hyb}$ is a continuous embedding, and is a section of $\Val$.
\\Its image $\Sk^{\hyb}(Z_{\Sigma}):= \iota_{\hyb}(N_{\Sigma} \times [0,1])$ is called the hybrid toric skeleton of $Z_{\Sigma}$.
\end{prop}
\begin{proof}
By definition of the topology on $Z^{\hyb}$, in order to prove that $\iota_{\hyb}$ is continuous, it is enough to prove that if $f \in \gO(U)$ is an algebraic function on a Zariski open subset $U$ of $Z$, then $(x, \lambda) \mapsto \log \lvert f\rvert_{\iota_{\hyb}(x, \lambda)}$ is continuous on $\iota_{\hyb}^{-1}(U^{\hyb})$. We may assume that $f \in \C[M]$ is a polynomial, which we write as $f =\sum_{j \in J} a_j \chi^{m_j}$, where the $a_j \in \C^*$, $m_j \in M$ and the index set $J$ is finite.
\\We now have:
$$\log \lvert f\rvert_{\iota_{\hyb}(x, \lambda)} =\lambda \log \big\lvert \sum_{j \in J} a_j e^{-\lambda^{-1} \langle m_j, x \rangle} \big \rvert$$
for $\lambda \neq 0$, while:
$$\log \lvert f\rvert_{\iota_{\hyb}(x, 0)} = -\min_{j \in J} \langle m_j, x \rangle.$$
It is straightforward to check that this indeed defines a continuous function on $N_{\R} \times [0,1]$, and similarly on $N_{\sigma} \times [0,1]$ for $\sigma \in \Sigma$: the function $f$ is defined on $Z_{\sigma}$ if and only if $\langle x, m_j \rangle \ge 0$ for $j \in J$ and $x \in N_{\sigma}$, so that we are merely allowing the exponentials in the above sum to vanish.
\\The fact that $\Val \circ \iota_{\hyb} = \Id_{N_{\Sigma}^{\hyb}}$ follows from the fact that $\val_{| Z^{\hyb}_{\lambda}} = \lambda \val_{\C}$ under the homeomorphism $Z^{\hyb}_{\lambda} \simeq Z^{\hol}$, so that $\iota_{\lambda}$ is a section of $\val_{| Z^{\hyb}_{\lambda}}$ for $\lambda >0$.
\\This finally implies that $\iota_{\hyb}$ is an embedding, the inverse map being given by the restriction of $\Val$ to $\Sk^{\hyb}(Z_{\Sigma})$.
\end{proof}
\begin{rem} Let $Z$ be a toric variety over $\C$, and $Z_K$ its base change to $K = \C((t))$. Then we have a commutative diagram:
\begin{center}
    \begin{tikzcd}
    Z^{\hyb}_K \arrow[r, "Q"] \arrow[d, "\Val_K"]  & Z^{\hyb}_0 \arrow[d, "\Val_0"] \\
   N_{\Sigma} \times \bar{\D}_r  \arrow[r, "\lambda"] & N_{\Sigma} \times [0,1],
   \end{tikzcd}
\end{center}
so that writing $t : [0,1] \fl \bar{\D}_r$ the map given by $t(\lambda) = r^{1/\lambda}$, we see that the map:
$$\iota_{\hyb, K} := \iota_{\hyb, 0} \circ (\Id_{N_{\Sigma}} \times t)$$
defines a continuous section of $\Val_K$.
\end{rem}
\subsection{Fiberwise-toric metrics}
Throughout this section, we work with a toric scheme $Z=Z_{\Sigma}$ over a Banach ring $A$.
The Berkovich analytification $Z^{\an}$ comes with two continuous maps: a $\val$ map:
$$\val_{\Sigma} : Z^{\an} \fl N_{\Sigma},$$
and a structure morphism:
$$\pi : Z^{\an} \fl \M(A).$$
Writing $N_{\Sigma}(A) = N_{\Sigma} \times \M(A)$, this yields a continuous map:
$$\val_A : Z^{\an} \fl N_{\Sigma}(A)$$
defined by the formula $\val_A(x) = (\val(x), \pi(x))$. In particular, the torus over $A$ is endowed with a map:
$$\val_A : \T^{\an}_A \fl N_{\R} \times \M(A).$$
\begin{defn}
Let $(L, s)$ be a toric line bundle on $Z$, and $\phi$ a metric on $L$. Then we say that $\phi$ is a fiberwise-toric metric on $L$ if and only the function:
$$\lVert s \rVert_{\phi} : \T^{\an}_A \fl \R$$
is such that $\lVert s(x) \rVert_{\phi} = \lVert s(y) \rVert_{\phi}$ for any two $x, y \in \T^{\an}$ such that $\val_A(x) = \val_A(y)$.
\end{defn}
Here a toric line bundle on $Z$ means, as in the case of a field, a line bundle $L$ on $Z$, together with an isomorphism $L \simeq \gO_Z(D)$ for a toric Weil divisor $D$ and a global section $s$ that is a monomial on $\T_A$ under the latter isomorphism.
\begin{ex} Assume that $A=K$ is a complete valued field. Then if $K=\C$, a metric is toric if and only if it is invariant under the action of the maximal compact torus $\bS \subset \T$. 
\\If $K$ is a non-archimedean discretely-valued field, then our definition of a toric metric matches the one from \cite[def. 4.3.2]{BPS}.
\end{ex}
The following description is a direct consequence of the definition:
\begin{prop}
Let $\phi$ be a continuous metric on the toric line bundle $L$. Then $\phi$ is a toric metric if and only for any $b \in \M(A)$, the restriction $\phi_b$ of $\phi$ to $Z_{\Sigma, \mathscr{H}(b)}^{\an}$ is a toric metric in the sense of \cite[def. 4.3.2]{BPS}.
\end{prop}
For continuous toric metrics over a valued field, the semi-positivity of the metric can be read off in the combinatorial world:
\begin{theo}{\cite[thm.4.8.1]{BPS}}  \label{theo toric field} Let $Z$ be a proper toric variety over a complete valued field $K$, and $(\gO_Z(D), s)$ a semi-ample toric line bundle on $Z$. We let $\iota : N_{\Sigma} \hookrightarrow Z^{\an}$ be the canonical section of $\val$ as given in section \ref{toric anal}. Then the mapping:
$$\lVert \cdot \lVert_{\phi} \longmapsto \big(\Phi := -\log \lVert s \circ \iota  \rVert_{\phi} \big)$$
sets up a correspondence between continuous, semi-positive toric metrics $\phi$ on $L$ and continuous, $D$-admissible convex functions $\Phi$ on $N_{\R}$.
\end{theo}
The function $\Phi$ will sometimes be called \emph{toric potential} for $\phi$. Note that it follows from the proof of prop. \ref{lem cvx} that if $\Phi$ is convex and $D$-admissible, then the function $(\Phi- \Psi_D)$ extends continuously to $N_{\Sigma}$. More generally, continuous toric metrics on $L$ are in bijection with continuous convex functions on $N_{\R}$ satisfying the latter condition.
\\We expect a similar picture to hold over more general Banach rings $A$: a continuous, fiberwise-toric metric should induce a family of convex functions on $N_{\R}$, parametrized by the Berkovich spectrum $\M(A)$, and which satisfy a suitable growth condition fiberwise over $\M(A)$. The semi-positivity of the metric in the direction of the base should be equivalent to the fact that this family of convex functions varies in a psh way on $\M(A)$. What this concretely means in general seems unclear at the moment; however in the case of the hybrid space attached to a complex toric variety, theorem \ref{theo toric} shows that the semi-positivity on the base translates into convexity with respect to the coordinate $\lambda \in [0,1]$.
\\Note that when $K$ is non-archimedean, this is proved only in the case where $K$ is discretely-valued field in \cite{BPS}, but the general case follows roughly from the argument of the proof of theorem \ref{theo toric}, together the following generalization of \cite[cor. 4.7.2]{BPS}, so that we omit the details.
\begin{prop}
Let $K$ be a complete non-archimedean field, and let $\phi \in \FS^{\tau}(Z_{\Sigma}, L)$ be a toric metric. Then there exists a finite set $(s_j)_{j \in J}$ of toric sections of $mL$ such that:
$$\phi = m^{-1} \max_{j \in J} \big( \log \lvert s_j \rvert + c_j \big).$$
In other words, toric Fubini-Study metrics on $L$ are precisely the Fubini-Study metrics associated to toric sections of powers of $L$.
\end{prop}
\begin{proof}
Write:
$$\phi = m^{-1} \max_{i \in I} \big( \log \lvert s_i \rvert + c_i \big)$$
for arbitrary sections $(s_i)_{i \in I}$ of $mL$, without common zeroes. Each $s_i$ may be written in a basis of toric sections as a finite sum $s_i = \sum_j \lambda_{i, j} \chi^{m_j}$ with $\lambda_{i,j} \in K$, so that if $x \in \T^{\an}$ we have:
$$\log \lvert s_i(\rho(x)) \rvert = \max_j (v_K(\lambda_{i,j}) + \log \lvert \chi^{m_j}(\rho(x))\rvert )$$
by definition of the monomial valuation $\rho(x)$ (see def. \ref{def toric rho}). Since $\phi$ is toric, the equality $\phi(x) = \phi(\rho(x))$ holds for every $x \in \T^{\an}$, and we infer:
$$\phi = \max_{i \in I} \bigg( \max_{j \in J_i} \big(v_K(\lambda_{i,j}) + \log \lvert \chi^{m_j} \rvert \big) +c_i \bigg),$$
hence the result since the $(\chi_{m_j})_{j \in \cup_i J_i}$ have no common zeroes.
\end{proof}
Let $K$ be a complete valued field. As a consequence of thm. \ref{theo toric field}, any semi-ample toric line bundle $L = \gO_Z(D)$ admits a canonical continuous psh metric $\phi_{\can}$, such that the associated convex function $\Phi_{\can} = \Psi_{D}$. The canonical metric can be described explicitly as follows: let $V \subset P_D$ be the set of vertices, then we have:
$$\phi_{\can} = \max_{m \in P_D} \log \lvert s_m \rvert.$$
Note that the $(s_m)_{m \in V}$ have no common zeroes, as a semi-ample line bundle on a toric variety is globally generated \cite[p. 68]{Fu}. This definition can thus be generalized to a toric scheme over any Banach ring:
\begin{defn} \label{def metr can} Let $A$ be a Banach ring, and $(Z, L)$ a toric scheme together with a semi-ample line bundle over $A$. The canonical metric on $L$ is the tropical Fubini-Study metric defined by the formula:
$$\phi_{\can} = \max_{m \in P_D} \log \lvert s_m \rvert,$$
where the $s_m$ are defined as in prop. \ref{prop toric sections}.
\end{defn}
The canonical metric is in particular fiberwise-toric.
\subsection{Hybrid toric metrics}
If $\phi$ is a toric hybrid metric on $L^{\hyb}$, it follows from the previous discussions that it is entirely determined by the function:
$$\Phi(w) = -\log \lVert s(\iota_{\hyb}(w)) \rVert_{\phi},$$
where $w \in N^{\hyb}_{\Sigma}$. Our next result asserts that in the case of continuous metrics, the semi-positivity of $\phi$ can also be read on the hybrid tropical toric variety:
\begin{theo} \label{theo toric}
Let $Z=Z_{\Sigma}$ be a toric variety over $\C$, and $L = (\gO_Z(D), s_D)$ a semi-ample toric line bundle on $Z$. We write $\Psi_D : N_{\R} \times [0,1] \fl \R$ the associated support function, extended trivially in the $\lambda$-direction.
\\There is a one-to-one correspondence between continuous, semi-positive, fiberwise-toric metrics $\phi$ on $L^{\hyb}$, and continuous convex functions:
$$\Phi : N_{\R} \times [0,1] \fl \R$$
such that the function $(\Phi - \Psi_D)$ extends as a continuous function to $N_{\Sigma} \times [0,1]$.
\\The correspondence is explicitly given by:
$$ \phi \longmapsto \bigg( (x, \lambda) \mapsto -\log \lVert s_D(\iota_{\hyb}(x, \lambda)) \rVert_{\phi} \bigg).$$
\end{theo}
\begin{ex} \label{ex toric prod} Assume that $\Phi (x, \lambda) \equiv \Phi(x)$ is independent of $\lambda$. Then unraveling the definitions and using the isomorphism $Z^{\hyb}_{\lambda} \simeq Z^{\hol}$ obtained by rescaling the absolute value, the induced family of metrics on $Z^{\hol}$ is given by:
$$\phi_{\lambda}(z) = \lambda^{-1} \Phi( -\lambda \log \lvert z \rvert)),$$
while the non-archimedean metric $\phi_0$ on $Z^{\an}_0$ is described by the convex function $\Phi$, in the setting of theorem \ref{theo toric field}.
\\Similarly, pulling back on the trivial degeneration $X = Z \times \D^*$, our theorem states in particular that the psh family of metrics on $p_1^* L$:
$$\phi_t = (-\log \lvert t \rvert) \cdot \Phi( \frac{\log \lvert z \rvert}{\log \lvert t \rvert})$$
extends as a continuous psh metric to $X^{\hyb}$, with $\phi_0 = \Phi \circ \val$. Note that the $\phi_t$ are psh with respect to $(z,t)$ since $\Phi$ can be written as the decreasing limit of a sequence of piecewise-affine functions.
\end{ex}
\begin{proof}[Proof of theorem \ref{theo toric}]
Let $\phi \in \CPSH(Z^{\hyb}, L^{\hyb})$ be a continuous psh, fiberwise-toric metric, and write:
$$\Phi : N_{\R} \times [0,1] \fl \R,$$
$$(x, \lambda) \mapsto -\log \lVert s_D(\iota_{\hyb}(x, \lambda)) \rVert_{\phi},$$
it is a continuous function by prop. \ref{iota hyb}.
Then writing $\phi = \lim_j \phi_j$ as the decreasing limit of a net in $\FS^{\tau}$, we have that $\Phi = \lim_j \Phi_j$ as a decreasing limit, so that to prove that $\Phi$ is convex on $N_{\R} \times [0,1]$, it is enough to prove that $\Phi_j$ is convex for $\phi_j \in \FS^{\tau}(L^{\hyb})$. Dropping the $j$ subscript, we write:
$$\phi = m^{-1} \max_{\alpha \in A} (\log \lvert s_{\alpha}\rvert + c_{\alpha}),$$
with $s_{\alpha} \in H^0(Z, mL)$ and $c_{\alpha} \in \R$.
By convexity of the maximum of finitely many convex functions, we are reduced to the case:
$$\phi = m^{-1} \log \lvert s \rvert,$$
where $s \in H^0(Z, mL)$ for some $m>0$. Using the isomorphism $\gO_{\T}(D) = s_D \gO_{\T}$, we may write $s = f \times (s_D)^m$ on $\T$ for $f \in K[M]$, so that:
$$\Phi(x, \lambda) = m^{-1} \log \big\lvert \frac{s_D^m}{s}((\iota_{\hyb}(x, \lambda)) \big\rvert =m^{-1} \log \lvert f(\iota_{\hyb}(x, \lambda)) \rvert.$$
By the proof of prop. \ref{iota hyb} and with the same notation, we have:
$$\Phi(x, \lambda) = \lambda m^{-1}  \log \big\lvert \sum_{j \in J} a_j e^{-\lambda^{-1} \langle m_j, x \rangle} \big\rvert$$
for $\lambda \neq 0$, while:
$$\Phi(x, 0) =-m^{-1} \min_{j \in J} \langle m_j, x \rangle.$$
The fact that $\Phi$ is convex on $N_{\R} \times [0,1]$ is now elementary. 
\\We now move back to the case where $\phi \in \CPSH(L^{\hyb})$ is an arbitrary continuous psh, fiberwise-toric metric on $L^{\hyb}$. To prove that $(\Phi-\Psi_D)$ extends continuously to $N_{\Sigma} \times [0,1]$, we argue as in the proof of \cite[prop. 4.3.10]{BPS}. Let $\sigma \in \Sigma$ be a cone, then there exists $m_{\sigma} \in M$ such that the support function $\Psi_{| \sigma} \equiv m_{\sigma}$. Moreover, the section $s_{\sigma} := \chi^{-m_\sigma} s_D$ is a nowhere vanishing section of $L$ on $Z_{\sigma}$. As a result, the function:
$$(x, \lambda) \mapsto - \log \lVert s_{\sigma} (\iota_{\hyb}(x, \lambda)) \rVert_{\phi}$$
is continuous on $N_{\sigma} \times [0,1]$. However the above expression is easily seen to be equal to $\Phi(x, \lambda) -\langle m_{\sigma}, x \rangle = \Phi(x, \lambda) -\Psi_D(x)$, so that $(\Phi - \Psi_D)$ extends continuously to $N_{\Sigma} \times [0,1]$.
\\Conversely, let $\Phi : N_{\R} \times [0,1] \fl \R$ be a continuous convex function, such that $(\Phi - \Psi_D)$ extends continuously to $N_{\Sigma} \times [0,1]$. We want to produce a continuous, semi-positive hybrid metric $\phi$ on $L$ such that $\Phi = -\log \lVert s_D \circ \iota_{\hyb} \rVert$. We start by applying prop. \ref{convex hybrid} to produce a decreasing sequence $(\Phi_j)_{j \in \N}$ of continuous convex functions on $N_{\R} \times [0,1]$, converging to $\Phi$ pointwise, and such that the function $(\Phi_j - \Psi_D)$ extends continuously to $N_{\Sigma} \times [0,1]$ for all $j \ge 0$. Additionally, we may write $\Phi_j$ as:
$$\Phi_j =d_j^{-1} \max_{\alpha \in A} \big(\langle m_{\alpha}, \cdot \rangle + b_{\alpha} \lambda + c_{\alpha} \big),$$
with the $m_{\alpha} \in M \cap P_D$. Since $\Phi_j(\cdot, \lambda)$ is $D$-admissible for any $\lambda$, it follows from lemma \ref{lem bpf} that $(s_{m_{\alpha}}, \alpha \in A)$ is a family of toric sections of $mL$ without common zeroes, so that:
$$\phi_j =m^{-1} \max_{\alpha \in A} \big(\log \lvert s_{m_{\alpha}} \rvert + b_{\alpha} \log \lvert e \rvert + c_{\alpha} \big)$$
is a Fubini-Study hybrid metric on $L$, inducing the convex function $\Phi_j$ on $N_{\R} \times [0,1]$ by the lemma below. Moreover, since the $(\Phi_j - \Psi_D)_j$ decrease to $(\Phi - \Psi_D)$ over $N_{\R} \times [0,1]$, the same holds over $N_{\Sigma} \times [0,1]$ by uniform convergence, which is easily seen to imply that the $\phi_j$ decrease to a continuous metric $\phi$ on $L^{\hyb}$, which concludes the proof.
\end{proof}
\begin{lem} Let $L = (\gO_Z(D), s_D)$ be a toric line bundle on $Z$, and $m \in M \cap P_D$. We write $\phi  = \log \lvert s_m \rvert$ the singular metric on $L^{\hyb}$ induced by prop. \ref{prop toric sections}. Then for $(x, \lambda) \in N_{\R} \times [0,1]$, the equality:
$$-\log \lVert s_D(\iota_{\hyb}(x, \lambda)) \rVert_{\phi} = \langle m, x \rangle$$
holds.
\end{lem}
\begin{proof}
By definition, we have:
$$-\log \lVert s_D(\iota_{\hyb}(x, \lambda)) \rVert_{\phi} = \log (\frac{\lvert s_{m} \rvert}{\lvert s_{D} \rvert})(\iota_{\hyb}(x, \lambda)),$$
but under the isomorphism $\gO_{\T}(D) = s_D \cdot \gO_{\T}$ we have $s_m /s_D = \chi^{-m}$, which concludes.
\end{proof}
\subsection{Hybrid family of Haar measures}
We conclude this chapter with a discussion on Haar measures on hybrid tori, elaborating on the framework from \cite[§4.2]{BPS}.
\\Let $Z$ be a complex projective toric variety, and $ \pi : Z^{\hyb} \fl [0,1]$ the associated hybrid space. For each $\lambda >0$, there exists a canonical homeomorphism $p_{\lambda} : Z^{\hol} \xrightarrow{\sim} Z_{\lambda} = \pi^{-1}(\lambda)$, where $Z^{\hol}$ denotes the complex analytification  of $Z$. 
\\For $x \in N_{\R}$, let $\bS_x = \val^{-1}(x) \subset \T^{\hol}$, and $\iota_{\lambda} : \T^{\hol} \hookrightarrow Z^{\hyb}$ be the composition of the embedding $\T^{\hol} \subset Z^{\hol}$ with the isomorphism $Z^{\hol} \simeq Z^{\hyb}_{\lambda}$. The real compact torus $\bS_x$ carries a canonical Haar probability measure $\mu_{\bS}$, which we can pushforward to the complex torus, and then to $Z^{\hyb}$ via the embedding $\iota_{\lambda}$. We write the resulting measure as:
$$\mu_{x, \lambda} := (\iota_{\lambda})_* \mu_{\bS_x}$$
for $\lambda >0$, and define:
$$\mu_{x, 0} = \delta_{\gamma(x)}$$
the Dirac mass at the Gauss point $\gamma(x)$.
\\It is well-understood that the natural non-archimedean analog of the Haar measure on a real compact torus is the Dirac mass at the Gauss point of the affinoid torus \cite[prop. 4.2.10]{BPS}, the following theorem shows that the analogy is furthermore continuous on the hybrid space:
\begin{theo} \label{theo Haar}
Let $Z$ be a complex projective toric variety, and $Z^{\hyb}$ the associated hybrid space. Then the family of measures $(\mu_{x, \lambda})_{(x, \lambda) \in N_{\R} \times [0,1]}$ on $Z^{\hyb}$ defined above is weakly continuous.
\end{theo}
\begin{proof} We need to prove that if $f \in \mathcal{C}^0(Z^{\hyb})$, then the real-valued function:
$$(x, \lambda) \mapsto \int_{Z^{\hyb}} f(z) d\mu_{x, \lambda}(z)$$
is continuous on $N_{\R} \times [0,1]$. If $\lambda >0$, then:
$$(x, \lambda) \mapsto \int_{Z^{\hol}} f(z, \lambda) d\mu_{\bS_x}(z)$$
is continuous on $N_{\R} \times (0,1]$ by dominated convergence, since $f$ is continuous and bounded on $Z^{\hol} \times (0,1]$.
\\We now need to examine what happens when $\lambda \fl 0$. By density of $\DFS(Z^{\hyb}) \subset \mathcal{C}^0(Z^{\hyb})$, it is enough to prove the statement for $f \in \DFS(Z^{\hyb})$, so that we may assume that there exists a line bundle $L$ on $Z$ such that $f= f_1-f_2$, with $f_i \in \FS(Z^{\hyb}, L^{\hyb})$. Since every line bundle on $Z$ is isomorphic to a toric one, we may and will assume that $L$ is a toric line bundle. We thus write:
$$f_i =m^{-1} \max_{\alpha \in A} (\log \lvert s_{\alpha} \rvert + c_{\alpha}),$$
for $i=1,2$. Since the measures $\mu_{x, \lambda}$ are supported on $\T^{\hyb}$, on which each section $s_{\alpha} \in H^0(Z, mL)$ can be written as $P_{\alpha} \times s_D$, where $P_{\alpha} \in \C[M]$, we are reduced to the case where $L =\gO_Z$ and $f = g_1-g_2$, where:
$$g_i =m^{-1} \max_{\alpha \in A} (\log \lvert P_{\alpha} \rvert + c_{\alpha}),$$
with $P_{\alpha} \in \C[M]$. The result thus boils down to the lemma below.
\end{proof}
\begin{lem}
Let $(P_{\alpha})_{\alpha \in A}$ be a finite family of polynomials in $\C[z_1,...,z_n]$. For $x \in \R^n$ and $\lambda \in \R_{>0}$, we define the polydisk $\D_{\lambda, x} := \{ \lambda \log \lvert z_i \rvert < -x_i \; \forall i \le n \}$ and $\bS_{\lambda, x}$ its distinguished boundary. Then there exists a locally bounded function $B$ of $x$ such that:
$$\bigg\lvert \int_{\bS_{\lambda, x}} \max_{\alpha \in A}( \lambda \log \lvert P_{\alpha}(\theta)  \rvert +c_{\alpha}) d\mu(\theta) -\max_{\alpha \in A} (-v_{\gamma(x)}(P_{\alpha}) + c_{\alpha})\bigg\rvert \le \lambda B(x),$$
where the Haar measure $\mu$ on $\bS_{\lambda, x}$ is normalized to have mass $1$.
\end{lem}
\begin{proof}
The left integral can be written as:
$$\int_{(\bS^1)^n} \big(\max_{\alpha \in A} \log \lvert P_{\alpha}(\exp(-\lambda^{-1} x + i \theta))\rvert+ c_{\alpha} \big)  d \theta,$$
where $d\theta$ is the Haar probability measure on $(\bS^1)^n$.
\\However if $P = \sum_{m \in M} a_m \chi^m$ and $\theta \in (\bS^1)^n$, we have:
$$ \lambda \log \lvert P(\exp(-\lambda^{-1}x+i \theta) \rvert + \min_{a_m \neq 0} \langle m, x\rangle = \lambda \log \big\lvert \sum_{m \in M} a_m e^{-\frac{ \langle m-m_0 , x \rangle}{\lambda} +i \theta} \big\rvert,$$
where $m_0$ achieves the minimum of $\langle m, x \rangle$ for $a_m \neq 0$. We infer that: 
$$ \lambda \log \lvert P(\exp(-\lambda^{-1}x+i \theta) \rvert + \min_{a_m \neq 0} \langle m, x\rangle = \lambda \log \big\lvert a e^{i\theta} +\sum_{m'} a_{m'} e^{-\frac{\langle m'-m_0, x\rangle}{\lambda} + i \theta} \big\rvert,$$
where $a$ is a constant and $\langle m'-m_0, x\rangle >0$ for all $m'$. Taking the maximum and averaging over $(\bS^1)^n$, this concludes the proof.
\end{proof}
\section{The non-archimedean SYZ picture} \label{sec SYZ}
We will now move on to our main motivation for the content of this paper: maximal degenerations of Calabi-Yau manifolds - here by Calabi-Yau manifold we mean a projective complex variety with trivial canonical bundle, which in particular includes hyperkähler and abelian varieties.
\subsection{The classical SYZ picture} \label{sec classical SYZ}
The SYZ conjecture \cite{SYZ} arises as an attempt to provide mathematical explanation for the following phenomenon, known as mirror symmetry: a Calabi-Yau manifold $X$ should admit a mirror dual $\check{X}$, whose symplectic geometry encodes in a natural way the complex geometry of $X$, and vice-versa. We start with a definition:
\begin{defn} Let $(X, \omega, \Omega)$ be a Calabi-Yau manifold of dimension $n$, and $B$ an $n$-dimensional (real) topological manifold. 
\\A special Lagrangian fibration $f : X \fl B$ is a continuous map which is a smooth torus fibration outside a closed, codimension $2$ subset $\Gamma \subset B$, and such that for all $b \in B \setminus \Gamma$, $\omega_{| f^{-1}(b)} \equiv 0$ and $\Im(\Omega)_{| f^{-1}(b)} \equiv 0$.
\end{defn}
If $f: X\fl B$ is a special Lagrangian torus fibration, we write $B^{\sm} := B \setminus \Gamma$ its smooth locus, each smooth fiber $X_b = f^{-1}(b)$ for $b \in B^{\sm}$ is a smooth, special Lagrangian torus inside $X$. 
\\Loosely speaking, the SYZ conjecture predicts that if $X$, $\check{X}$ are two mirror dual Calabi-Yau manifolds, each of them should admit a special Lagrangian torus fibration onto the same half-dimensional base $B$, and those two fibrations should be dual to each other, in the sense that for $b \in B^{\sm}$, the fibers $X_b$ and $\check{X}_b$ are dual tori, i.e. $\check{X}_b = H^1(X_b, \R/\Z)$. Hence, given a Calabi-Yau manifold $X$ and a special Lagrangian fibration $f: X \fl B$, the SYZ heuristic provides a program to reconstruct the mirror $\check{X}$, by dualizing the fibration $f$ over $B^{\sm}$ and then extend this fibration over $B$. The problem of compactifying a smooth Lagrangian torus fibration into a singular one has however proven to be a difficult problem, and has led to many developments over the past years; we refer the reader to the series of recent or upcoming papers \cite{RZ21a,RZ21,RZ} for some progress on these questions.
\\The fibration induces the two following structures on $B^{\sm}$: an integral affine structure $\nabla^{\Z}$, and a Riemannian metric $g_B$ called the McLean metric. We start with a definition:
\begin{defn}
An \emph{integral affine structure} $\nabla^{\Z}$ on a topological
manifold is an atlas of charts (called $\Z$-affine charts) with transition functions in $\textrm{GL}_n(\mathbb{Z}) \ltimes \mathbb{R}^n$.
\end{defn}
If $B$ is a manifold endowed with an integral affine structure, it makes sense to speak about integral affine functions on $B$ - those are continuous functions on $B$ that are, locally in $\Z$-affine coordinates, of the form $f(x_1,\ldots,x_n)=a_1 x_1 + \ldots + a_n x_n +b $, with $a_i \in \Z$ and $b \in \R$. We denote by $\Aff_{B}$ the sheaf of integral affine functions on $B^{\sm}$, which recovers the integral affine structure by \cite[2.1]{KontsevichSoibelman}.
\\Without giving the precise definition as we will not need them, the integral affine structure on $B^{\sm}$ is induced by action-angle coordinates, while the McLean metric $g$ is such that there locally exists a convex function $K : B^{\sm} \fl \R$ such that in $\Z$-affine coordinates, it can be written as:
$$g_{ij} = \frac{\partial^2 K}{\partial y_i \partial y_j}.$$
We refer the reader to the survey \cite{Gro} for a more thorough discussion. Conversely, starting from the data of $(B, \nabla^{\Z}, K)$ a smooth $\Z$-affine manifold together with a (multi-valued) strictly convex function, one can construct a toy model of Lagrangian fibration with base $B$ as follows: set $X = X(B)$ to be the tangent torus bundle, that is, $X(B)= TB/ \Lambda$, $\Lambda$ being the local system of lattices associated to $\nabla^{\Z}$. 
\\In term of local $\Z$-affine coordinates $y_j$ on $B$, we can pick fiber coordinates $x_j = dy_j$, and the complex structure on $X$ is defined by requiring the functions:
$$z_j = e^{2i \pi (x_j + i y_j)}$$
to be holomorphic. In particular, the map $f$ is now given by: 
$$f(z_1,...,z_n) = \frac{1}{2} (-\log \lvert z_1 \rvert,..., - \log \lvert z_n \rvert).$$
\\The Kähler form on $X$ is defined in terms of its local potential:
$$\omega = i \ddbar (K \circ f),$$
which restricts to zero on any torus fiber - we even have a stronger statement, that the Kähler metric $g_{\omega}$ restricts to a flat metric to each of the torus fibers, hence the name \emph{semi-flat metric}. Moreover, the complex Monge-Ampère measure:
$$\omega^n = \big(dd^c (K \circ f) \big)^n$$
is directly related to the real Monge-Ampère measure of $K$ in affine coordinates, as we have:
$$\big( dd^c (K \circ f) \big)^n =i^{n^2} \det( \frac{\partial^2 K}{\partial y_i \partial y_j} ) dz \wedge d\bar{z},$$
where $dz= dz_1\wedge ... \wedge dz_n$.
\\Later on, it was realized that the SYZ picture should rather occur as an asymptotic phenomenon on the moduli space of Calabi-Yau manifolds, near the locus of the boundary where the manifolds degenerate in the worst possible way - in the terminology of mirror symmetry, this means that the manifolds approach the "large complex structure limit". Thus, we now let $X \fl \D^*$ be a maximal degeneration of Calabi-Yau manifolds (see def. \ref{def sk ess}), polarized by a relatively ample line bundle $L$. We fix a relative trivialization $\Omega \in H^0(X, K_{X/\D^*})$, and write $\omega = (\omega_t)_{t \in \D^*}$ the family of Kähler-Ricci flat metrics on $L$. 
\\In this setting, a strong version of the SYZ conjecture would predict the following: for any $\lvert t \rvert \ll 1$, the fiber $X_t$ admits a special Lagrangian fibration:
$$f_t : X_t \fl B,$$
onto a compact base $B$ of real dimension $n$, and whose fibers are of size $(\log \lvert t \rvert^{-1})^{-1/2}$ with respect to the rescaled Calabi-Yau metric $\tilde{\omega}_t \in \frac{c_1(L_t)}{\big\lvert \log \lvert t \rvert \big\rvert}$. These rescaled Calabi-Yau metrics indeed have bounded diameters by \cite{LT}.
\\In particular, the rescaled Ricci-flat Kähler manifolds $(X_t, \tilde{\omega}_t)$ should collapse in the Gromov-Hausdorff sense to the base $B$, endowed with an asymptotic version of the McLean metric alluded to above. This is known in the case of abelian varieties by the work of \cite{Odaka2018} (see also \cite{GO}), and a slightly weaker statement is known for Fermat hypersurfaces in $\CP_{\C}^{n+1}$ by \cite{Li1}.
\subsection{The essential skeleton}
In this section, we will explain how the base $B$ of the SYZ fibration can be realized in the non-archimedean world, as before we let $X \fl \D^*$ be a projective degeneration of smooth Calabi-Yau manifolds. We use the same notation $X$ for the base change to $K= \C((t))$.
\begin{defn}
Let $X/K$ be a Calabi--Yau variety. A minimal model of $X$ is a dlt model $\X/R$, such that the logarithmic relative canonical divisor is trivial, i.e.
$$ K^{\log}_{\X/R} := K_{\X/R} + \X_{0, \red} - \X_0 \sim \gO_\X.$$
We say that $\X$ is a good minimal model of $X$ if the components of $\X_0$ are $\Q$-Cartier.
\end{defn}
Building on the general MMP machinery, the existence of such models is known when $X$ is defined over an algebraic curve (and is expected to hold in the general case):
\begin{theo}[{\cite[Theorem 1.13]{NXY}}]
Let $X/K$ be a projective Calabi--Yau variety, and assume that $X$ is defined over an algebraic curve.
Then there exists a minimal model $\X/R$ of $X$. Furthermore, there exists a finite extension $K'/K$ such that the base change $X_{K'}$ admits a minimal model with reduced special fiber.
\end{theo}
Such models are however not unique, but they turn out to have the same skeleton inside $X^{\an}$ by \cite{NX} (even though the triangulation may differ).
\begin{defn} \label{def sk ess}
Let $X/K$ be an $n$-dimensional Calabi--Yau variety. The essential skeleton $\Sk(X) \subset X^{\an}$ is the Berkovich skeleton $\Sk(\X)$ of any minimal model $\X/R$ of $X$.
\\We will say that $X/K$ is maximally degenerate if $\dim \Sk(X) =n$.
\end{defn}
Fixing a meromorphic at zero, nowhere-vanishing section $\Omega \in H^0(X, K_{X/\D^*})$, the essential skeleton can alternatively be defined as the locus where the weight function $\wt_{\Omega} : X^{\an} \fl \R$ from \cite[thm. 3.3.3]{NX} reaches its minimum. The weight function can be expressed as $\wt_{\Omega} = \log \lvert \Omega \rvert_{A_X}$, where $A_X$ is the canonical metric on $K_X$, as defined in \cite{Tem}.
\\The nowhere-vanishing $n$-form $\Omega \in H^0(X, K_{X/\D^*})$ induces a continuous family of volume forms $\nu_t := i^{n^2} \Omega_t \wedge \bar{\Omega}_t$ on the $X_t$'s, and thus a continuous family of probability measures:
$$\mu_t := \frac{\nu_t}{\nu_t(X_t)}.$$
We now have the following result, which computes the weak limit on the Calabi-Yau measures inside the hybrid space $X^{\hyb}$. The support of the limiting measure is precisely the essential skeleton:
\begin{theo}{\cite[thm. B]{BJ}} \label{theo cv mesures}
\\Let $X \fl \D^*$ be a degeneration of Calabi-Yau manifolds, and $\Omega \in H^0(X, K_{X/ \D^*})$ a relative trivialization of the canonical bundle. Writing $\nu_t = i^{n^2} \Omega_t \wedge \bar{\Omega}_t$, we have:
$$\nu_t(X_t) \sim C \lvert t \rvert^{\kappa} (\log \lvert t \rvert^{-1})^d,$$
 as $t \rightarrow 0$, where $\kappa \in \Q$ and $d = \dim \Sk(X) \le n$. Moreover, the family of rescaled measures:
$$\mu_t = \frac{\nu_t}{\nu_t(X_t)}$$
converge weakly on $X^{\hyb}$ to a Lebesgue-type measure supported on $\Sk(X) \subset X^{\an}$.
\end{theo}
This result provides heuristic explanation as to why the base $B$ of the SYZ fibration should be the essential skeleton $\Sk(X)$: since the Calabi-Yau metric is obtained by solving the complex Monge-Ampère equation $(\omega_t)^n = C_t \mu_t$, one would expect the solution to the non-archimedean Monge-Ampère equation $\MA(\phi) = \mu_0$ provided by theorem \ref{theo sol MA} to provide a non-archimedean incarnation of the family of Calabi-Yau metrics; and since $\mu_0$ is supported on $\Sk(X)$ one would hope to translate this equation into a real Monge-Ampère equation on $\Sk(X)$ via a non-archimedean avatar $\rho : X^{\an} \fl \Sk(X)$ of the SYZ fibration. We will come back to this in section \ref{sec comp MA}.
\subsection{Affinoid torus fibrations}
Let $(X, L) \fl \D^*$ be a polarized maximal degeneration of Calabi-Yau manifolds. Since the metric spaces $(X_t, \tilde{\omega}_t)$ are conjectured to look like the total space of a special Lagrangian torus fibration over $\Sk(X)$, submersive away from a singular locus $\Gamma_t$ of real codimension 2, it is reasonable to expect that the essential skeleton $\Sk(X)$ should be equipped with an integral affine structure away from a discriminant locus of real codimension $2$. We will now explain how to try and construct such an affine structure in a purely non-archimedean setting, we write $X^{\an}$ the Berkovich analytification of $X$ viewed as a variety over $K = \C((t))$.
\\As we saw in section \ref{sec classical SYZ}, our local semi-flat model for the SYZ fibration is given by $f : (\C^*)^n \fl \R^n$, the map $f$ being simply $f(z_1,...,z_n) = (- \log \lvert z_1 \rvert,..., -\log \lvert z_n \rvert)$. 
\\The non-archimedean counterpart of this map is the tropicalization map:
$$ \val : \T^{\text{an}} \fl N_{\R}$$
from section \ref{toric anal}, where $\T$ is an $n$-dimensional torus over $K$. This leads to the following definition:
\begin{defn} \label{defn:affinoid torus fibration}
Let $\rho : X^{\an} \fl B$ be a continuous map to a topological space $B$. For any point $b \in B$, we say that $\rho$ is an \emph{affinoid torus fibration at $b$} if there exists an open neighbourhood $U$ of $b$ in $B$, such that the restriction to $\rho^{-1}(U)$ fits into a commutative diagram:
$$\begin{tikzpicture}
\matrix(m)[matrix of math nodes, row sep=3em,
    column sep=2.5em, text height=1.5ex, text depth=0.25ex]
{\rho^{-1}(U) & \val^{-1}(V)  \\
    U  & V, \\ };
	\path 
	(m-1-1) edge[->] node[auto] {$\simeq$} (m-1-2)
			edge[->] node[auto] {$ \rho $} (m-2-1)
         (m-1-2) edge[->] node[auto] {$\val$} (m-2-2)
         (m-2-1) edge[->] node[auto] {$\sim$} (m-2-2);
\end{tikzpicture}$$
$V$ being an open subset of $\R^n$, the upper horizontal map an isomorphism of analytic spaces, the lower horizontal map a homeomorphism, and the map $\val$ defined as in section \ref{toric anal}.
\end{defn}
Note that the above definition implies that $B$ is a topological manifold at $b$; in the case where $\X$ is a good dlt model of $X$, the base $B = \Sk(\X)$ and $\rho$ is the Berkovich retraction $\rho_{\X}$, this does not necessarily hold at every point of $\Sk(\X)$.
\\Given a continuous map $\rho : X^{\text{an}} \fl B$, we denote by $B^{\textrm{sm}}$ the (possibly empty) open locus of points in $B$ where $\rho$ is an affinoid torus fibration at, and we call $B \setminus B^{\textrm{sm}}$ the \emph{discriminant} or {singular locus} of $\rho$. As in the archimedean case, the (possibly disconnected) topological manifold $B^{\textrm{sm}}$ is endowed with an integral affine structure, whose charts are the continuous maps from definition \ref{defn:affinoid torus fibration}. An alternative, more intrinsic description of this structure is given in \cite[4.1, Theorem 1]{KontsevichSoibelman}: let $U \subset B^{\textrm{sm}}$ be a connected open subset. Then if $h$ is an invertible analytic function on $\rho^{-1}(U)$, its modulus $\lvert h \rvert$ is constant on the fibers of $\rho$ by the maximum principle, so that it defines a continuous function on the base. We now have:
$$ \mathrm{Aff}_{ B^{\textrm{sm}}}(U) = \{\log\lvert h \rvert \, | \, h \in \mathcal{O}^{\times}_{X^{\an}} (\rho^{-1}(U))\}.$$
\begin{ex} \label{ex:affinoid torus over max face}
If $\X/R$ is a good dlt model of $X$ with reduced special fiber, then the Berkovich retraction $\rho_{\X}: X^{\an} \rightarrow \Sk(\X)$ 
is an affinoid torus fibration over the interior of the maximal faces $\tau$ of $\Sk(\X)$. Indeed, the retraction over $\mathring{\tau}$ only depends on the formal completion of $\X$ along the corresponding 0-dimensional stratum $p$, and we have that $\rho_{\X}^{-1}(\mathring{\tau})$ is the generic fiber of $\Spf \widehat{\gO}_{\X, p}$. By the dlt condition, the pair $(\X, \X_0)$ is snc at $p$, so that: 
$$\widehat{\gO}_{\X, p} \simeq  R[[z_0,..., z_n]] / (t- z_0...z_n).$$
The generic fiber $\fX_p^{\eta}$ embeds in the torus: 
$$\T^{\an} = \big(\Sp K[z_0,...,z_n] / \{t=z_0...z_n\} \big)^{\an},$$
and the Berkovich retraction is the restriction of the $\val$ map over the open simplex:
$$\Int(\tau) = \{x_0+...+x_n =1 \} \subset (\R_{>0})^{n+1}.$$
\end{ex}
\begin{defn} \label{def adm} Let $X/K$ be a smooth projective variety, $B = \Sk(\X)$ the skeleton of a dlt $R$-model of $X$. We say a map:
$$\rho : X^{\an} \fl B$$
is an \emph{admissible} retraction if the following hold:
\begin{itemize}
\item there exists a piecewise-affine locus $\Gamma \subset B$ of codimension $\ge 2$ such that $\rho$ is an affinoid torus fibration over $\Sk(X) \setminus \Gamma$;
\item the retraction $\rho$ is piecewise-affine: for any piecewise-affine function $\psi : B \fl \R$, the pullback $\psi \circ \rho : X^{\an} \fl \R$ is a piecewise-affine function on $X^{\an}$ (or model function \cite[def. 5.11]{BE}).
\end{itemize}
\end{defn}
One way to produce a map satisfying the second item of the definition is to choose an snc model $\X' \fl \X$ dominating $\X$, and a piecewise-affine map $\pi : \Sk(\X') \fl \Sk(\X)$. Then the retraction $\rho = \pi \circ \rho_{\X'}$ is piecewise-affine, examples of such retractions were constructed in \cite{MPS}. It would be interesting to know whether or not all piecewise-affine maps $\rho : X^{\an} \fl \Sk(\X)$ arise this way.
\\As suggested by the definition, affinoid torus fibrations are related to the local toric geometry of $X$ and its models. We now describe a more global class of examples, that we will be using as building blocks later on. Following \cite{NXY}, we will say that an $R$-scheme of finite type $\cZ$ is toric if there exists a complex toric variety $\mathcal{Z}$, together with a toric morphism $t: \mathcal{Z} \fl \A^1$, such that $\cZ \simeq \mathcal{Z} \times_{\A^1} R$. 
\\Let $\X /R$ be a regular toric model of $\T$, i.e. a regular toric $R$-scheme such that $\X \times_R \; \Sp K = \T$, which we assume to have reduced special fiber. Such a model is described by a regular fan $\hat{\Sigma} \subset \hat{N}_{\R} := N_{\R} \times \R_{\ge 0}$, whose cones intersect $N_{\R} \times \{0 \}$ only at the origin.
\\We consider the following open subset of $\T^{\an}$:
$$\widehat{\X}_{\eta} := \{ v_x \in \T^{\an} \lvert \, v_x \; \text{has a center on} \; \X \},$$
which admits a Berkovich retraction:
$$\rho_{\X} : \widehat{\X}_{\eta} \fl \Sk(\X)$$
defined as in remark \ref{retract local}. In this case, the map $\rho_{\X}$ can be described explicitly as follows: let $\Sigma_1$ be the polyhedral complex in $N_{\R}$ obtained by intersecting the fan $\hat{\Sigma}$ with $N_{\R} \times \{1\}$. There is a natural identification between $\Sigma_1$ and $\mathcal{D}(\X_0)$, sending a vertex of $\cD(\X_0)$ to the primitive generator of the corresponding ray of $\hat{\Sigma}$, and then extending on each face by linearity. Moreover, it follows from \cite[Theorem A.4]{GJKM} that $\gamma ( \lvert \Sigma_1 \rvert) = \Sk(\X) \subset \T^{\an}$, where $\gamma : N_{\R} \fl \T^{\an}$ is the Gauss embedding from definition \ref{def Gauss}.
\begin{prop}[{\cite[Example 3.5]{NXY}}, {\cite[prop. 1.5.2]{MPS}}] \label{prop: toric retraction}
The equality:
$$\widehat{\X}_{\eta} = \val^{-1} (\lvert \Sigma_1 \rvert) $$
holds, and $\rho_{\X} = \val_{| \widehat{\X}_{\eta}}$, which is thus an affinoid torus fibration over the relative interior of $\Sk(\X)$.
\end{prop}
This also holds when $X$ is a regular proper toric variety over $K$ and $\X/R$ a regular proper toric model of $X$, by \cite[Theorem A.4]{GJKM}.
\\As our main interest lies in degenerations of Calabi-Yau manifolds, the above class of examples is too restrictive for us. We are thus naturally lead to consider degenerations which are only locally formally toric, so that we may apply remark \ref{retract local} to produce Berkovich retractions which are locally affinoid torus fibrations. Let us make this more precise:
\begin{defn}
\label{toricdef}
Let $\X$ be a normal $R$-scheme of finite type, such that $(\X, \X_0)$ is a good dlt pair, and let $Y$ be a stratum of $\X_0$.
We say that $\X$ is \emph{toric along Y} if there exists a toric $R$-scheme $\cZ$, a stratum $W$ of $\cZ_0$ and a formal isomorphism over $R$ 
$$\widehat{\X_{/Y}} \simeq \widehat{\cZ_{/W}}.$$
\end{defn}
The upshot is now that of $\X$ is toric along $Y$, then the Berkovich retraction $\rho_{\X} : X^{\an} \fl \Sk(\X)$ is an affinoid torus fibration over $\Star(\tau_Y)$, as follows from \cite[Theorem 6.1]{NXY} (end of the proof) and \cite[\S 3.4]{NXY}. 
\\An obvious necessary condition for $\X$ to be toric along $Y$ is for $Y$ to be a toric variety itself; the following result states that under a certain positivity assumption on the conormal bundle of $Y \subset \X$, the converse holds:
\begin{theo} \label{THEO TORIC STRAT} \cite[thm. B]{MPS}
Let $X/K$ be a smooth projective variety of dimension $n$, and $\X/R$ be a dlt model of $X$ with reduced special fiber $\X_0= \sum_\alpha D_\alpha$, such that every $D_\alpha$ is a Cartier divisor.
\\Let $Z= D_0 \cap D_1 \cap \ldots \cap D_{n-r}$ be an $r$-dimensional stratum of $\X_0$, such that:
\begin{itemize}
\itemsep0pt
\item $\mathring{Z} \subset Z$ is a torus embedding, where $\mathring{Z}= Z \setminus \cup_{\alpha \neq 0,1,\ldots,n-r}D_\alpha$;
\item the conormal bundle $\nu_{Z/ \X}^*$ is a nef vector bundle on $Z$;
\item for each $\alpha \notin \{0,...,n-r\}$, the intersection $D_{\alpha} \cap Z$ is either empty or connected.
\end{itemize}
Then the formal completion $\widehat{\X_{/Z}}$ is isomorphic to the formal completion of the normal bundle $\mathcal{N}=\nu_{Z/\X}$ along the zero section. In particular, $\X$ is toric along $Z$ (in the sense of definition \ref{toricdef}).
\end{theo}
Under the same assumptions, we have:
\begin{cor}
The Berkovich retraction $\rho_\X: X^{\an} \rightarrow \Sk(\X)$ is an $n$-dimensional affinoid torus fibration over $\Star(\tau_Z)$. In particular, the integral affine structure induced by $\rho_\X$ over the interior of the maximal faces of $\Sk(\X)$ extends to $\Star(\tau_Z)$ with no singularities.
\end{cor}
\subsection{The Monge-Ampère comparison property} \label{sec comp MA}
In this section, we explain the expected connection between the non-archimedean Monge-Ampère operator and the real Monge-Ampère operator on skeleta.
\\Let $N$ be a lattice, $\phi : U \fl \R$ a convex function defined on a convex open subset $U \subset N_{\R}$, and let $d\lambda$ be the Lebesgue measure on the dual $M_{\R} :=\Hom(N, \R)$.
\\For $x_0 \in U$, we define the gradient image of $\phi$ at $x_0$:
$$ \nabla \phi (x_0) = \{ l \in M_{\R} / \phi(x_0) + l(x-x_0) \le \phi(x) \; \forall x \in U \}.$$
Geometrically, this is the set of covectors cutting out affine hyperplanes in $N_{\R}$ that meet the graph $\Gamma$ of $\phi$ at $x_0$, and are below $\Gamma$ on all of $U$. 
\\For instance, if $\phi$ has $\mathcal{C}^1$-regularity, it follows immediately from the convexity of $\phi$ that the gradient image of $\phi$ at $x_0$ only contains the differential $d\phi_{x_0}$ of $\phi$ in the usual sense.
\\More generally, if $E \subset U$ is a Borel set, we set $\nabla \phi(E) = \cup_{x_0 \in E} \nabla \phi(x_0)$.
\begin{defn}
Let $\phi : U \fl \R$ be a convex function defined on an open subset $U \subset N_{\R}$.
\\The real Monge-Ampère measure of $\phi$ is defined as:
$$ \M (\phi) (E) = d\lambda(\nabla \phi(E)),$$
for any Borel set $E \subset U$.
\end{defn}
If $\phi$ has $\mathcal{C}^2$-regularity, then this is nothing but the density measure $\det(\nabla^2 \phi) d\lambda^{\vee}$.
Recall that in the semi-flat setting of section \ref{sec classical SYZ}, the Kähler Ricci-flat metric $\omega$ is given by:
$$\omega = dd^c (K \circ f)$$
for a locally defined convex function $K : B^{\sm} \fl \R$, and $f: X \fl B$ the special Lagrangian fibration. The Calabi-Yau condition was then translated into the real Monge-Ampère equation for $K$:
$$\det (\frac{\partial^2 K}{\partial y_i \partial y_j}) = C,$$
with $C >0$ a real constant. As a result, the semi-flatness condition can be expressed as the fact that the local potential for the Calabi-Yau metric is constant along the fibers of the SYZ map, which reduces the complex Monge-Ampère equation to a real Monge-Ampère equation on the base of the fibration.
\\We now move back to the non-archimedean picture. In this setting, the analog of the Calabi-Yau metric is the unique psh metric $\phi$ on $L^{\an}$ solving the non-archimedean Monge-Ampère equation:
$$\MA(\phi) = \mu_0,$$
where $\mu_0$ is the limit measure from thm. \ref{theo cv mesures} - when $X$ has semi-stable reduction, this is simply the Lebesgue measure on the essential skeleton $\Sk(X)$.
\begin{defn}
Let $X/K$ be a smooth projective variety over $K$, $\X/R$ a dlt model of $X$, and let $\rho : X^{\an} \fl \Sk(\X)$ be an admissible retraction (defn. \ref{def adm}), with discriminant locus $\Gamma$.
\\We say that a continuous, psh metric $\phi$ on $L^{\an}$ is a \emph{semi-flat} metric with respect to $\rho$ if there exists an open cover $\Sk(X) \setminus \Gamma = \bigcup_{U \in \mathcal{U}} U$ by charts as in def. \ref{defn:affinoid torus fibration}, such that $\phi$ restricts to a toric metric over any $U \in \mathcal{U}$.
\end{defn}
More precisely, this means that after pulling back $L^{\an}$ to the open $\val^{-1}(V)$ as in defn. \ref{defn:affinoid torus fibration} and choosing a trivialization, the metric $\phi$ extends to $\T^{\an}$ as a toric metric on a toric line bundle extending $L^{\an}$. 
\\When $\phi$ is a toric metric, it turns out that the non-archimedean Monge-Ampère equation can be reduced to a real Monge-Ampère equation:
\begin{theo}{\cite[thm. 4.4.4]{BPS}} Let $(Z, L)$ be a polarized toric variety over $K$.
\\Let $\phi$ be a semi-positive toric metric on $L$, given by a convex function which we denote by $\Phi : N_{\R} \fl \R$.
\\Then the non-archimedean Monge-Ampère measure of the metric $\phi$ can be computed as follows:
$$\MA(\phi) = n! \gamma_* (\M(\Phi)),$$
where $\gamma : N_{\R} \hookrightarrow Z^{\an}$ is the embedding given by the Gauss section.
\end{theo}
As a result, by locality of the non-archimedean Monge-Ampère operator (see \cite{CLD}, \cite[cor. 5.2]{BFJ1}), if $\phi$ is a semi-flat metric with respect to an admissible retraction $\rho$, this implies that away from the discriminant locus of $\rho$, the non-archimedean Monge-Ampère measure $\MA(\phi)$ is supported on $\Sk(\X)$, and can be expressed locally as the real Monge-Ampère measure on $\Sk(\X)$ of locally defined convex functions - which may be however fairly hard to compute explicitly in terms of the metric $\phi$. \\In general, one would hope to write $\phi = \phi_{\Ld} + \psi$ for a certain model metric $\phi_{\Ld}$ on $L$, and $\psi : X^{\an} \fl \R$ a continuous function, and be able to rephrase the semi-flat condition into the invariance property:
$$\psi = \psi \circ \rho$$
over $\Sk(\X) \setminus \Gamma$. We would then obtain, locally in toric charts, an equality of measures of the form:
$$\MA(\phi) = n! \iota_* (\M(\Phi_{\Ld} + \psi)),$$
where $\Phi_{\Ld}$ is a local toric potential for the model metric $\phi_{\Ld}$. While it is unclear how to do this in general because it depends on the explicit description of the toric charts for the retraction, we are able to work this out in the case of degenerations of hypersurfaces, see remark \ref{rem MA hyper}.
\\If $\rho = \rho_{\X}$ is the Berkovich retraction associated to an snc model $\X/R$ and we restrict our attention to the interior of the maximal faces of $\Sk(\X)$, then things get simpler and we have the following result:
\begin{theo}{\cite{Vil}} \label{theo Vil}
Let $(X, L)$ be a smooth projective variety over $K$, and $(\X, \Ld)$ be a semi-stable $R$-model of $(X, L)$. Let $\phi \in \CPSH(X, L)$ such that:
$$\phi = \phi_{\Ld} + \psi,$$
where $\psi \in \mathcal{C}^0(X^{\an})$. Let $\tau$ be a maximal face of $\Sk(\X)$, and assume that the invariance property $\psi = \psi \circ \rho_{\X}$  holds over $\Int(\tau)$. Then we have the equality of measures:
$$\mathbf{1}_{\Int(\tau)} \MA(\phi) = n! \M(\psi_{| \Int(\tau)}).$$
\end{theo}
We now let $X/K$ be a maximal degeneration of Calabi-Yau manifolds, polarized by the ample line bundle $L$. We write $\phi \in \PSH(X^{\an}, L^{\an})$ the unique solution to the NA Monge-Ampère equation:
$$\MA(\phi) = \mu_0$$
\begin{defn}{(Weak comparison property, \cite[def. 3.11]{Li2})} We say that $\phi$ satisfies the weak NAMA/real MA comparison property if there exists an snc model $\X/R$ of $X$, together with a model $\Ld$ of $L$ on $\X$, such that the function $\psi: X^{\an} \fl \R$ defined by the formula:
$$\phi = \phi_{\Ld} + \psi$$
satisfies $\psi = \psi \circ \rho_{\X}$ over the interior of the maximal faces of $\Sk(\X)$.
\end{defn}
The motivation for this definition resides in the following theorem, which proves that the above condition is sufficient for the weak SYZ conjecture to hold:
\begin{theo}{\cite[thm. 1.3]{Li2}} Let $X \fl \D^*$ be maximal degeneration of CY manifolds, polarized by an ample line bundle $L$. Assume the NAMA/real MA comparison property holds. 
\\Then for any $\delta >0$, there exists $\eps_{\delta}$ such that for all $\lvert t \rvert \le \eps_{\delta}$, there exists a special Lagragian torus fibration on an open subset of $X_t$ of (normalized) Calabi-Yau measure greater than $(1-\delta)$.
\end{theo}
Very loosely, the idea of the proof of the above theorem is that the solution to the non-archimedean Monge-Ampère equation $\phi$ provides a solution to the real Monge-Ampère equation on an open region of $\Sk(X)$ that has full measure; and pulling back this solution by the map $\Log_{\X}$ from \cite[prop. 2.1]{BJ} provides an ansatz for the Calabi-Yau potentials on a region of the $X_t$ with asymptotically full measure: if $K \subset \Int(\tau)$ is a compact subset, then Li obtains bounds for $\lvert \phi_{CY, t} - \phi \circ \Log_{\X} \rvert$ over $K$ using uniform Skoda estimates on the highly degenerate Calabi-Yau manifolds $X_t$ \cite{Li3}. The estimates on the potential are then enough to prove the existence of a special Lagrangian fibration over $\Log_{\X}^{-1}(K)$ which is a perturbation of $\Log_{\X}$, by the results of Y. Zhang \cite{ZhSL}.
\\While the above theorem provides us with a asymptotic SYZ fibration, it does not describe completely the Gromov-Hausdorff limit of the family: we only obtain that the disjoint union of maximal faces of $\Sk(X)$ endowed with a real Monge-Ampère metric embeds into the Gromov-Hausdorff limit. We expect that in order to obtain Gromov-Hausdorff convergence to the metric completion of $\Sk(X) \setminus \Gamma$ for a piecewise-affine locus $\Gamma$ of codimension $2$, we need a stronger version of the comparison property in order to obtain a solution to the real Monge-Ampère equation in codimension $1$:
\begin{defn}{(Strong comparison property)}
Let $X/K$ be a smooth projective, maximally degenerate Calabi-Yau variety, polarized by an ample line bundle $L$.
\\Let $\phi$ be the unique positive metric on $L$ satisfying the non-archimedean Monge-Ampère equation:
$$\MA(\phi) = \mu_0.$$
We say that $\phi$ satisfies the strong comparison property if there exists an admissible retraction $\rho : X^{\an} \fl \Sk(X) $ such that $\phi$ is a semi-flat metric with respect to $\rho$.
\end{defn}
Assume that the above statement holds, and write $\Gamma \subset \Sk(X)$ the discriminant locus of $\rho$. Then $\big( \Sk(X) \setminus \Gamma \big) = \bigcup_{\alpha \in A} U_{\alpha}$ such that $\rho$ is an affinoid torus fibration over $U_{\alpha}$, we write $\phi_{\alpha} : U_{\alpha} \fl \R$ the convex, toric potential (in the sense of thm. \ref{theo toric field}) for $\phi$ on $U_{\alpha}$. Then by the previous discussion there is an equality of measures on $U_{\alpha}$:
$$\mathbf{1}_{U_{\alpha}} \mu_0 = n! \M(\phi_{\alpha}),$$
where the real Monge-Ampère measure on the right-hand side is computed with respect to the integral affine structure induced by $\rho$. Moreover, for $\alpha \neq \beta$, the difference $\phi_{\alpha} - \phi_{\beta}$ must be affine on $U_{\alpha} \cap U_{\beta}$, so that we may define locally in affine coordinates $g =\nabla^2 \phi := \sum_{i, j} \frac{\partial^2 \phi_{\alpha}}{\partial y_i \partial y_j} dy_i dy_j$ whenever $\phi_{\alpha}$ is smooth.
\begin{conj}{(Kontsevich-Soibelman conjecture)}
Let $(X, L)$ be a polarized, maximal degeneration of Calabi-Yau manifolds. Then there exists an admissible retraction:
$$ \rho : X^{\an} \fl \Sk(X),$$
with discriminant locus $\Gamma \subset \Sk(X)$, such that the solution $\phi$ to the non-archimedean Monge-Ampère equation:
$$\MA(\phi) = \mu_0$$
is a semi-flat metric with respect to $\rho$. Additionally, the collection of local potentials $\phi = (\phi_{\alpha})_{\alpha}$ is a smooth, strictly convex solution of the real Monge-Ampère equation on $\Sk(X) \setminus \Gamma$.
\\Moreover, the rescaled Calabi-Yau metrics $(X_t, \tilde{\omega}_t)$ converge in the Gromov-Hausdorff sense as $t \rightarrow 0$ to a compact $n$-dimensional base $(B, g_B)$, such that:
\begin{itemize}
\item $(B, g_B)$ contains the smooth Riemannian manifold $(\Sk(X) \setminus \Gamma, \nabla^2 \phi)$, whose complement is of Hausdorff codimension at least 2,
\item $B$ is homeomorphic to $\Sk(X)$.
\end{itemize}
\end{conj}
Building on the results from \cite{Li1}, we are able to prove the strong comparison property for degenerations of Fermat hypersurfaces, see section \ref{sec Fermat}.
\\Let us point out that assuming the strong comparison property, the strict convexity of the local potential $\phi$ can not be deduced from a local PDE argument: there are non-smooth, non-strictly convex solutions to real Monge-Ampère equations of the form $\M(\phi) = d\lambda$ \cite{Moo}, so that higher regularity of the solution would have to be obtained through a 'global' argument.
\section{The non-archimedean SYZ fibration for hypersurfaces} \label{sec syz super}
\subsection{Degenerations of Calabi-Yau hypersurfaces} \label{sec super setup}
We are interested in degenerations of Calabi-Yau manifolds obtained in the following way:
$$X = \{ z_0...z_{n+1} + tF =0 \} \subset \CP^{n+1} \times \D^*,$$
where $F \in H^0(\CP^{n+1}, \gO_{\CP}(n+2))$ is a general section (in a sense that will be specified later); we also denote by $X$ the base change to $K = \C((t))$. We will see that in this case, the essential skeleton $\Sk(X)$ can be realized explicitly using tropical geometry, and there is a natural way of endowing it with an integral affine structure, which is singular in codimension 2. The construction we outline here is a special case of \cite{Gro05}, but since the hypersurface case is far less technical we will outline the details.
\\We start by identifying the essential skeleton $\Sk(X)$ with the Berkovich skeleton of the model $\X$:
\begin{prop}\label{prop model dlt} Let $F \in H^0(\CP^{n+1}, \gO_{\CP}(n+2))$ be a general section. Then the hypersurface:
$$\X = \{ z_0...z_{n+1} +tF =0 \} \subset \CP^{n+1}_R$$
is a minimal dlt model of the smooth, maximally degenerate Calabi-Yau manifold $X = \X_K$; in particular $\Sk(X) = \Sk(\X)$.
\end{prop}
More precisely, the statement of the above proposition holds whenever $F$ is not identically zero on each toric stratum $Z$ of $\CP^{n+1}$, and the intersection $Z \cap \{ F=0 \}$ is smooth; such $F$ are called \emph{admissible} in \cite{HJMM}.
\begin{rem} As shown by the proof, the model $\X$ will not be good even for general choice of $F$ - the irreducible components of $\X_0$ are not $\Q$-Cartier at the singular points of $\X$ as soon as $n \ge 2$. Thus, there does not exist a Berkovich retraction $\rho_{\X} : X^{\an} \fl \Sk(\X)$, which is consistent with the statement of theorem \ref{THEO TORIC STRAT}: if such a retraction existed it would be an affinoid torus fibration over the whole $\Sk(X) \simeq \bS^n$, yielding an integral affine structure with no singularities on an $n$-sphere, which is impossible.
\end{rem}
\begin{proof} It follows from an elementary computation that for general $F$, the singular locus of the total space $\X$ is given by:
$$\Sing(\X) = \bigcup_{i \neq j} \{ z_i = z_j =t =F =0 \} \subset \X_0,$$
which is of codimension $3$ in the total space $\X$. The points where the singularity is the most severe are the intersection of the strata curves of $\X_0$ with the locus $\{ F=0 \}$, since $F$ is general this locus does not meet the zero-dimensional strata of $\X_0$. Thus, we may focus on the curve $\{z_1=...=z_n =0 \} \subset \X_0$, where the singularity is étale-locally of the form:
$$\cU = \{ z_1....z_n +tw =0 \} \subset \A^{n+1}_R.$$
 so that by \cite[prop. 11.4.24]{CoxLittleSchenck2011}, the pair $(\cU, \cU_0)$ is log canonical. We also infer that at the generic point of a stratum, $\X$ is smooth while the components of $\X_0$ meet transversally, which proves that $\X$ is a dlt model of $X$. 
 \\Moreover, $\X_0$ is reduced and $K_{\X/R} = \gO_{\X}$ by adjunction, which concludes.
\end{proof}
From now on, we will assume that $F$ is such that the statement of the above lemma holds. We view $Z_{\Sigma} =\CP^{n+1}$ as a smooth toric Fano variety of dimension $(n+1)$, with simple normal crossing boundary $D= \sum_{l=0}^{n+1} D_l$ the sum of coordinate hyperplanes. We additionally let $L=\gO_{\CP}(n+2)$ be the anticanonical polarization on $\CP^{n+1}$, and we write $P \subset M_{\R}$ the associated polytope. If $(v_1,...,v_{n+1})$ is a basis of $N$, we set $v_0 := -(v_1+...+v_{n+1})$ and the $v_l$'s for $l=0,...,(n+1)$ are the primitive generators of the rays of $\Sigma$.
\\Writing $P^* \subset N_{\R}$ the polar dual of $P$, that is:
$$P^* = \{ x \in N_{\R} / \langle u, x \rangle \le 1 \; \forall u \in P \},$$
we see that here $P^* = \Conv(v_0,...,v_n)$.
\\One of the reasons such degenerations provide an interesting class of examples is that their essential skeleton can be realized in the toric world, as we will now explain. The dual complex of the toric boundary $D \subset \CP^{n+1}$ can naturally be realized inside $N_{\R}$: if $D_l$ is a coordinate hyperplane, it corresponds to $e_l \in N$ the primitive integral generator of the corresponding ray of $\Sigma$. Since $\X_0 = D$, we see that we have a canonical isomorphism of simplicial complexes between the dual complex $\cD(\X_0)$ and the boundary $\partial P^*$, sending a vertex $v_{D_l} \in \cD(\X_0)$ to $v_l \in N$. Realizing $\cD(\X_0) \simeq \Sk(\X) \subset X^{\an}$ inside the Berkovich analytification, we have the more precise statement:
\begin{prop} \label{prop sk N} Let $X \subset \CP^{n+1}_K$ as above.
Then the composition: 
$$X^{\an} \subset \CP^{n+1, \an} \xrightarrow{\val_{\Sigma}} N_{\Sigma}$$
induces by restriction a isomorphism of simplicial complexes from $\Sk(\X)$ to $\partial P^*$.
\end{prop}
\begin{proof}It is enough to prove that if $D_{l_1},..., D_{l_{n+1}}$ are $(n+1)$ components of $D = \X_0$ meeting at a point $p$, then $\val$ maps the corresponding maximal face $\tau_p$ of $\Sk(\X)$ homeomorphically to: 
$$\tau_{\sigma} = \Conv(v_{l_1},...,v_{l_{n+1}}) \subset N_{\R},$$
where $\sigma = \sum_{i=1}^{n+1} \R_{\ge0} v_{l_i}$. Thus, we let $v \in \tau_p$ be a quasi-monomial valuation, which we see as a semi-valuation on $\CP^{n+1}$. Using prop. \ref{prop casi mono}, we identify $v$ with the monomial weight $w =(w_1,...,w_{n+1})$, where $w_i =v(f_i)$ for any local equation $f_i$ of $D_{l_i}$ at $p$. We let $l_0$ be the index such that $p \notin D_{l_0}$, and take as a local equation:
$$f_i = \frac{z_{l_i}}{z_{l_0}},$$
where $[z_0:...:z_{n+1}]$ are standard homogeneous coordinates on $\CP^{n+1}_K$. Then by definition of the $\val$ map, we have:
$$\langle \val(v), m \rangle = v(\chi^m)$$
for any $m \in M$. We let $(m_1,..., m_{n+1})$ be the basis of $M$ dual to $(v_{l_1},..., v_{l_n})$, so that $\chi^{m_i}$ is an equation for $D_{l_i}$ on the toric affine chart $\CP_{\sigma}$. We now have $\chi^{m_i} =f_i$, so that writing $m = \sum_{i=1}^{n+1} \lambda_i m_{i}$, we infer from prop. \ref{prop casi mono}:
$$v(\chi^m) = \sum_{i=1}^{n+1} w_i \lambda_i,$$
or equivalently $\val(v) = \sum_{i=1}^{n+1} w_i v_{l_i}$, which concludes the proof.
\end{proof}
We are thus naturally led to study the image inside $N_{\Sigma}$ of $X^{\an}$ by the valuation map $\val_{\Sigma}$. This is well-known to be the (extended) tropicalization $\overline{\Trop}(X)$ of the hypersurface, defined as follows.  Write the generic section $F$ as a sum of monomials:
$$F(z) = \sum_{m \in P_{\Z}} c_m z^m$$
and let $L: N_{\R} \fl \R$ be the piecewise-affine function obtained by replacing, in the equation $(tF +z_0...z_{n+1})$ for $X$, the products by sums and the sums by minima. More explicitly, $L$ is defined by the formula:
$$L(x) = \min \{ 0, \min_{c_m \neq 0} (1-\langle m , x \rangle ) \},$$
and $\Trop(X) \subset N_{\R}$ the locus where the minimum is achieved by at least two different terms. The following result is standard:
\begin{prop}
The closure $\overline{\Trop}(X) \subset N_{\Sigma}$ is the image of $X^{\an}$ via the map:
$$\val_{\Sigma} : Z_{\Sigma}^{\an} \fl N_{\Sigma}.$$
Moreover, the complement of $\Trop(X)$ in $N_{\R}$ admits one bounded connected component, which is the interior of $P^*$.
\end{prop}
\begin{proof} The first item follows from the fundamental theorem of tropical geometry \cite[thm. 3.1.1]{McS}, while the second one follows from the fact that the polytope $P^*$ is the locus where $L=0$.
\end{proof}
We will now construct a piecewise-affine subset $\Gamma \subset \Sk(X)$ of codimension 2, and an integral affine structure on $\Sk(X) \setminus \Gamma$, compatible with the integral affine structure on the interior of the maximal faces of $\Sk(\X)$, and compatible with the fan structure in the neighbourhood of each vertex. The construction is non-canonical, as it depends on the following choice: for each face $\tau$ of $\Sk(\X)$ of dimension $>0$, we choose a point $a_{\tau} \in \Int(\tau)$. In the sequel, we will call the collection $\underline{a} = (a_{\tau})_{\tau}$ a \emph{choice of branch cuts}.
\\We let $\widetilde{\mathscr{P}}$ be the subdivision of $\Sk(\X)$ whose faces are of the form:
$$\tau' = \Conv(a_{\tau_0},..., a_{\tau_l}),$$
whenever $\tau_0 \subset ... \subset \tau_l$ is an ascending chain of faces of $\Sk(\X)$.
\\The discriminant locus in now given as:
$$\Gamma := \bigcup \Conv(a_{\tau_0},..., a_{\tau_l}),$$
where the union is taken over ascending chains of faces $\tau_0 \subset ... \subset \tau_l$ of $\Sk(\X)$, with $\dim \tau_0 \ge 1$ and $\dim \tau_l \le (n-1)$. It is clear from the definition that this defines a codimension $2$ piecewise-affine subset of $\Sk(X)$, which we will sometimes write $\Gamma = \Gamma_{\underline{a}}$ whenever we wish to emphasize the dependency on the choice of branch cuts $\underline{a}$.
\begin{ex} Assume $n=2$, then we are choosing a point $a_e \in \Int(e)$ in each edge of $\Sk(\X)$, in addition to a point $a_{\tau} \in \Int(\tau)$ in each maximal face $\tau$ of $\Sk(\X)$. The discriminant locus $\Gamma$ consists of the set $\{ a_e \}_e$, which is finite.
\\Assume now that $n=3$, then we are choosing an interior point in each $d$-cell of $\Sk(\X)$ for $d=1,2,3$. The discriminant locus $\Gamma$ is a trivalent graph, contained in the union of the $2$-cells of $\Sk(\X)$, and its intersection with the 2-face $\tau =\Conv(v_1,v_2, v_3)$ is depicted below:
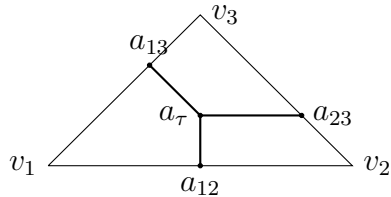
\tdplotsetmaincoords{60}{35}
\begin{figure}[H]
\begin{center}
\begin{tikzpicture}[scale =2]
\node[left] at (-1,0,0) {$v_{1}$};
\node[right] at (1,0,0) {$v_2$};
\node[right] at (0,1,0) {$v_{3}$};
\node[below] at (0,0,0) {$a_{12}$};
\node[right] at (2/3,1/3,0) {$a_{23}$};
\node[above] at (-1/3,2/3,0) {$a_{13}$};
\node[left] at (0,1/3,0) {$a_{\tau}$};
\filldraw (0,0,0) circle (0.4pt);
\filldraw (2/3,1/3,0) circle (0.4pt);
\filldraw (-1/3,2/3,0) circle (0.4pt);
\filldraw (0,1/3,0) circle (0.4pt);
\draw (-1,0,0)--(1,0,0)--(0,1,0)--(-1,0,0);
\draw[thick] (0,0,0)--(0,1/3,0)--(2/3,1/3,0);
\draw[thick] (0,1/3,0)--(-1/3,2/3,0);
\end{tikzpicture}
\end{center}
\end{figure}
\end{ex}
Let $v \in \Sk(\X)$ be a vertex, and $\Sigma_v$ be the fan of $D_v$ in $\R^n$. We let:
$$\phi_v : \Star(v) \fl \R^n$$
be the unique piecewise-affine homeomorphism sending $v$ to $0 \in \R^n$, and each vertex $v_l \in \Star(v)$ to the primitive generator of the corresponding ray in $\Sigma_v$.
\\Moreover, if $\tau \subset \Sk(\X)$ is a maximal face, we let $\phi_{\tau}: \tau \fl \R^n$ be an integral affine homeomorphism onto a standard simplex.
\\As mentioned above, the affine charts $\phi_v$ do not glue to define an integral affine structure on $\Sk(\X)$, i.e. the transition functions between them are not integral affine. To remedy this, we simply shrink the stars of the vertices using the subdivision $\widetilde{\mathscr{P}}$ so that they don't intersect anymore. More precisely, for a vertex $v$ of the skeleton, let $\widetilde{\Star}(v)$ be the star of $v$ in the polyhedral complex $\widetilde{\mathscr{P}}$. Then the singular affine structure on $\Sk(\X)$ is defined as follows:
\begin{defn} \label{def Z aff}
The collection of charts $\Big( (\Int(\widetilde{\Star}(v)), \phi_v)_v \cup (\Int(\tau), \phi_{\tau})_{\tau} \Big)$ is our induced integral affine structure on $\Sk(X) \setminus \Gamma$.
\end{defn}
From now on, we will make the following assumption: 
\begin{con} \label{condition} The section $F = \sum_{m \in P_{\Z}} c_m z^m$ is such that prop. \ref{prop model dlt} holds, and whenever $m \in P_{\Z}$ is a vertex, we have $c_m \neq 0$ - equivalently, the Newton polytope of the section $F$ is equal to the whole $P$.
\end{con}
This immediately implies the equality:
$$L(x) = \min \{ 0, \min_{m \in V(P)} (1-\langle m , x \rangle) \}.$$
Note that this is true for a generic choice of $F$, and this ensures notably that the tropicalization $\overline{\Trop}(X)$ is invariant under the action of the symmetric group $\mathfrak{S}_{n+2}$ on $N_{\Sigma}$.
The main theorem of this section is the following:
\begin{theo} \label{theo affinoid Yuto}
Assume that condition \ref{condition} holds. Let $\underline{a} =(a_{\tau})_{\tau}$ be a choice of branch cuts in $\Sk(\X)$, and $\delta_{\underline{a}} : \overline{\Trop}(X) \fl \Sk(X)$ the associated tropical contraction \cite[thm. 5.1]{Ya}. Then the composition:
$$\rho_{\underline{a}} := \delta_{\underline{a}} \circ \val_{\Sigma} : X^{\an} \fl \Sk(X)$$
is an affinoid torus fibration over $\Sk(X) \setminus \Gamma$. Moreover, the induced integral affine structure on $\Sk(X) \setminus \Gamma$ coincides with the one from definition \ref{def Z aff}.
\end{theo} 
Let $v= v_i \in \Sk(\X)$ be a vertex, and let $U_i \subset \Sk(\X)$ be the open star of $v_i$ with respect to the simplicial decomposition $\widetilde{\mathscr{P}}$. Then we need to prove that the map $\rho = \rho_{\underline{a}}$ is an affinoid torus fibration over each $U_i$, as well as over the interior of each maximal face $\tau \subset \Sk(\X)$. The proof will proceed in two steps:
\begin{itemize}
\item (section \ref{section tropical cont}) we mainly review the construction of the tropical contraction, following \cite{Ya}. The contraction is in fact constructed face by face, and we will describe it in particular near a vertex and over a maximal face. The latter, in combination with prop. \ref{prop sk N}, will readily imply that $\rho$ is an affinoid torus fibration over a maximal face of $\Sk(\X)$.
\item (section \ref{local models}) we then focus on a vertex $v_i \in \Sk(\X)$, and construct a small log resolution $\X_i \fl \X$, which is an isomorphism over the corresponding component $D_i$. In particular, we may apply theorem \ref{THEO TORIC STRAT} to conclude that the associated retraction $\rho_{\X_i}$ is an affinoid torus fibration over $\Star(v_i)$, and we conclude by proving that the equality $\rho = \rho_{\X_i}$ holds over $U_i$.
\end{itemize}
\subsection{The tropical contraction} \label{section tropical cont}
We let $X \subset \CP^{n+1}_K$ be a maximal degeneration of Calabi-Yau hypersurfaces as above, satisfying condition \ref{condition}. The goal of this section is to describe in a fairly simple way the construction of a certain tropical contraction $\delta : \overline{\Trop}(X) \fl \Sk(X)$, due to Yamamoto \cite{Ya}. Here tropical contraction means that the map $\delta$ is locally modeled on projections of the form $N_{\Sigma} \fl N_{\Sigma'}$, where $\Sigma'$ is the fan of a (closed) orbit in $Z_{\Sigma}$.
\begin{theo} {\cite[thm. 1.2]{Ya}} For any choice of branch cuts $\underline{a}= (a_{\tau})_{\tau}$, there exists a tropical contraction:
$$\delta_{\underline{a}} : \overline{\Trop}(X) \fl \Sk(X),$$
where we view $\Sk(X) \subset \overline{\Trop}(X)$ via the embedding of prop. \ref{prop sk N}. The tropical contraction preserves the integral affine structures, in the sense that:
$$\delta_* \Aff_{\overline{\Trop}(X)} = j_*(\Aff_{\Sk(X) \setminus \Gamma}),$$
where $\Aff$ denotes the sheaf of integral affine functions on $\overline{\Trop}(X)$ (see \cite[def. 2.10]{Ya}) and $\Sk(X) \setminus \Gamma$ respectively, and $j : \Sk(X) \setminus \Gamma \hookrightarrow \Sk(X)$ the inclusion.
\end{theo}
Let us point out that the original construction is far more general, as it applies to a wide range of Calabi-Yau complete intersections inside toric varieties. However the codimension one case is substantially simpler, so that we will outline the details of the construction in this case.
\\We fix once and for all a choice $\underline{a}$ of branch cuts, and consider the associated subdivision $\widetilde{\mathscr{P}}$ of $\Sk(\X)$. Our first step will be to describe $\overline{\Trop(X)}$ more explicitly. Recall that the open subset $\Trop(X) \subset \overline{\Trop(X)}$ is given as the corner locus of the function:
$$L(x) = \min \{ 0, \min_{m \in V(P)} (1- \langle m, x\rangle ) \},$$
hence the following natural stratification. If $x \in \Trop(X)$, we let $J(x) \subset I := \{0\} \cup V(P)$ the subset of indices that realize the minimum; note that $\lvert J(x) \rvert \ge 2$ since $x \in \Trop(X)$. Given a subset $J \subset I$, the associated stratum $T_J \subset \Trop(X)$ is the set of $x \in \Trop(X)$ such that $J(x) \supseteq J$. We furthermore write $\overline{T}_J \subset N_{\Sigma}$ its closure, which is contained in $\overline{\Trop}(X)$.
\begin{defn} Let $\tau \subset \partial P^*$ be a face, and write $\tau = \Conv(e_{i_1},...,e_{i_r})$. The associated cone $\sigma(\tau) \in \Sigma$ is given by:
$$\sigma(\tau) = \Cone(e_{i_1},...,e_{i_r}),$$
and its closure $\bar{\sigma}(\tau) \subset N_{\Sigma}$.
We additionally set $\tilde{\tau} := \tau + \bar{\sigma}(\tau) \subset N_{\Sigma}.$
\end{defn}
Let $J' \subset V(P)$ be a non-empty subset of vertices of $P$, this defines a face $\tau_{J'}$ of $\partial P^*$ in the following way:
$$\tau_{J'} = \{ x \in P^* / \langle m_j, x \rangle =1 \; \forall j \in J' \},$$
this is a face of dimension $(n+1) - \lvert J'\rvert$. It is straightforward to check that for $\lvert J' \rvert \ge 2$, the region $\tilde{\tau}_{J'}$ is contained in the stratum $\overline{T}_{J'}$.
\\The following proposition asserts that the bounded strata of the tropicalization are the faces of $\partial P^*$, while the unbounded ones are obtained from non-maximal faces of $\partial P^*$ and their associated cones:
\begin{prop} Let $I = \{0\} \cup V(P)$, and $J \subset I$ with at least two elements. We let $J' = J \cap V(P)$. 
\begin{itemize}
\item if $0 \in J$, then $\overline{T}_J = \tau_{J'} \subset \partial P^*$;
\item if $0 \notin J$, then $\overline{T}_J = \tilde{\tau}_{J'}$.
\end{itemize}
\end{prop}
Note that in the latter case, since $\lvert J \rvert \ge 2$, the face $\tau_{J'}$ is not a maximal face of $\partial P^*$. The following picture depicts the tropicalization when $n=2$: $\Trop(X)$ is the union of the boundary of the red polytope $P^*$ with the portions of the linear planes passing through pairs of vertices of $P^*$; here we have colored $\tilde{\tau}_J$ for $J \subset \{1,2,3 \}$.
\tdplotsetmaincoords{60}{15}
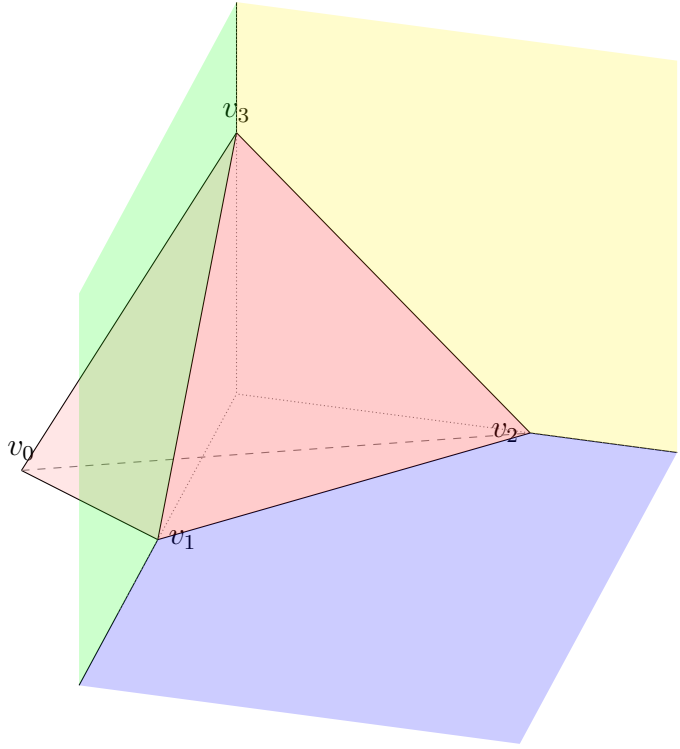
\begin{figure}[H]
    \centering
\begin{tikzpicture}[tdplot_main_coords,scale=4]
\node[left] at (1,0,0) {$v_2$};
\node[right] at (0,-1,0) {$v_1$};
\node[above] at (0,0,1) {$v_3$};
\node[above] at (-1,1,-1) {$v_0$};
\draw (1,0,0)--(1/2,0,1/2)--(0,0,1);
\draw (-1,1,-1)--(0,-1,0);
\draw (-1,1,-1)--(0,0,1);
\draw (0,-1,0)--(0,-2,0);
\draw (0,0,1)--(0,0,3/2);
\draw (1,0,0)--(3/2,0,0);
\draw[gray, dashed] (-1,1,-1)--(1,0,0);
\draw[gray,densely dotted] (0,0,3/2)--(0,0,0);
\draw[gray,densely dotted] (0,-3/2,0)--(0,0,0);
\draw[gray,densely dotted] (3/2,0,0)--(0,0,0);
\draw (1,0,0)--(1/2,-1/2,0)--(0,-1,0);
\draw (0,-1,0)--(0,-1/2,1/2)--(0,0,1);
\fill[fill=blue, fill opacity=0.2] (1,0,0)--(3/2,0,0)--(3/2,-2,0)--(0,-2,0)--(0,-1,0);
\fill[fill=yellow, fill opacity=0.2] (0,0,1)--(0,0,3/2)--(3/2,0,3/2)--(3/2,0,0)--(1,0,0);
\fill[fill=green, fill opacity=0.2] (0,0,3/2)--(0, -2,3/2)--(0,-2,0)--(0,-1,0)--(0,0,1);
\fill[fill=red, fill opacity=0.2] (0,-1,0)--(1,0,0)--(0,0,1);
\fill[fill=red, fill opacity=0.1] (0,-1,0)--(-1,1,-1)--(0,0,1);
\end{tikzpicture}
\caption{The tropicalization of a maximally degenerate K3 surface}
\end{figure}
We now move to the description of the tropical contraction. If $\tau$ is a face of $\partial P^*$, we set $U_{\tau} := \Int(\widetilde{\Star}(a_\tau))$ the open star of the point $a_\tau \in \Int(\tau)$ with respect to $\widetilde{\mathscr{P}}$. For $\tau' \subset \tau$ a smaller face of $\partial P^*$, set:
$$W_{\tau, \tau'} := U_{\tau} \cap U_{\tau'}$$
and $Y_{\tau, \tau'} := W_{\tau, \tau'} + \overline{\sigma}(\tau').$ Then
$$Y_{\tau} := \big( \bigcup_{\tau' \subset \tau} Y_{\tau, \tau'} \big) \cap \overline{\Trop}(X),$$
where the union ranges over the subfaces of $\tau$. For instance, if $\tau \subset \partial P^*$ is a maximal face, we simply have $Y_{\tau} = U_{\tau} = \Int(\tau)$.
\\Then by \cite[lem. 5.15]{Ya}, the union of $Y_{\tau}$ cover $\overline{\Trop}(X)$ when $\tau$ ranges over the faces of $\partial P^*$, so that it is enough to define local tropical contractions:
$$\delta_{\tau} : Y_{\tau} \fl U_{\tau},$$
which glue over the intersections $U_{\tau_1} \cap U_{\tau_2}$ by \cite[lem. 5.14]{Ya}. Then if $x \in Y_{\tau, \tau'}$, i.e. $x = w + x'$, with $w \in W_{\tau, \tau'}$ and $x' \in \overline{\sigma}(\tau')$, we set:
$$\delta_{\tau} (x) = w.$$
In other words, the local tropical contraction contracts the cone direction, and is locally isomorphic over $W_{\tau, \tau'}$ to the canonical projection $p: N_{\Sigma} \fl N_{\Sigma'}$, where $\Sigma'$ is the fan in $N_{\R} / \Span(\sigma(\tau'))$ induced by $\Sigma$. 
\\In the two-dimensional case, the restriction of the tropical contraction to the standard orthant is depicted as follows (we have picked the barycenters of the edges as choice of branch cuts):
\tdplotsetmaincoords{70}{25}
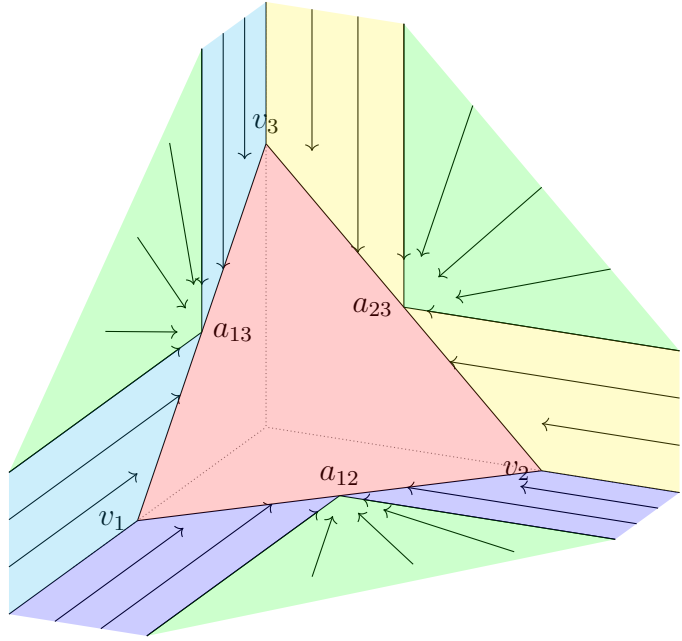
\begin{figure}[H]
    \centering
\begin{tikzpicture}[tdplot_main_coords,scale=4]
\node[left] at (1,0,0) {$v_2$};
\node[left] at (0,-1,0) {$v_1$};
\node[above] at (0,0,1) {$v_3$};
\node[above] at (1/2,-1/2,0) {$a_{12}$};
\node[left] at (1/2,0,1/2) {$a_{23}$};
\node[right] at (0, -1/2, 1/2) {$a_{13}$};
\draw (1,0,0)--(1/2,0,1/2)--(0,0,1);
\draw (0,0,3/2)--(0,0,1);
\draw (1/2, 0, 3/2)--(1/2, 0,1/2);
\draw (3/2, 0, 1/2)--(1/2, 0, 1/2);
\draw[gray,densely dotted] (0,0,3/2)--(0,0,0);
\draw[gray,densely dotted] (0,-3/2,0)--(0,0,0);
\draw[gray,densely dotted] (3/2,0,0)--(0,0,0);
\draw (1,0,0)--(1/2,-1/2,0)--(0,-1,0);
\draw (0,-1,0)--(0,-1/2,1/2)--(0,0,1);
\draw[->] (1/2,-2,0)--(1/2,-2/3,0);
\draw[->] (3/4, -5/4,0)--(9/16,-11/16,0);
\draw[->] (1, -1,0)--(5/8,-5/8,0);
\draw[->] (5/4, -3/4,0)--(11/16,-9/16,0);
\draw[->] (3/2,-1/2,0)--(7/12,-1/2, 0);
\draw[->] (1/3, -2, 0)--(1/3, -2/3, 0);
\draw[->] (1/6, -2, 0)--(1/6, -1, 0);
\draw[->] (3/2,-1/3 ,0)--(2/3, -1/3, 0);
\draw[->] (3/2,-1/6 ,0)--(1, -1/6, 0);
\draw[->] (1/2,0,3/2)--(1/2,0,2/3);
\draw[->] (3/4, 0,5/4)--(9/16,0, 11/16);
\draw[->] (1, 0,1)--(5/8,0, 5/8);
\draw[->] (5/4, 0, 3/4)--(11/16,0, 9/16);
\draw[->] (3/2,0, 1/2)--(7/12,0, 1/2);
\draw[->] (1/3, 0, 3/2)--(1/3, 0, 2/3);
\draw[->] (1/6, 0, 3/2)--(1/6, 0, 1);
\draw[->] (3/2,0, 1/3)--(2/3, 0, 1/3);
\draw[->] (3/2,0, 1/6)--(1, 0, 1/6);
\draw[->] (0, -1/6, 3/2)--(0, -1/6, 1);
\draw[->] (0, -1/3, 3/2)--(0, -1/3, 2/3);
\draw[->] (0, -1/2, 3/2)--(0, -1/2, 2/3);
\draw[->] (0, -2, 1/6)--(0, -1, 1/6);
\draw[->] (0, -2, 1/3)--(0, -2/3, 1/3);
\draw[->] (0, -2, 1/2)--(0, -2/3, 1/2);
\draw[->] (0, -3/4,5/4)--(0,-9/16, 11/16);
\draw[->] (0, -1,1)--(0, -5/8, 5/8);
\draw[->] (0,-5/4, 3/4)--(0, -11/16, 9/16);
\draw (0,-1,0)--(0,-2,0);
\draw (1/2,-1/2,0)--(1/2,-2,0);
\draw (1,0,0)--(3/2,0,0);
\draw (1/2,-1/2,0)--(3/2,-1/2,0);
\draw (0, -1/2, 3/2)--(0, -1/2, 1/2);
\draw (0, -2, 1/2)--(0, -1/2, 1/2);
\fill[fill=cyan, fill opacity=0.2] (0, -1/2, 1/2)--(0,-1,0)--(0,-2,0)--(0,-2,1/2);
\fill[fill=cyan, fill opacity=0.2] (0, -1/2, 1/2)--(0,0,1)--(0,0,3/2)--(0,-1/2,3/2);
\fill[fill=green, fill opacity=0.2] (0, -2, 1/2)--(0, -1/2, 1/2)--(0, -1/2, 3/2);
\fill[fill=green, fill opacity=0.2] (1/2,-1/2,0)--(1/2,-2,0)--(3/2,-1/2,0);
\fill[fill=blue, fill opacity=0.2] (0,-1,0)--(0,-2,0)--(1/2,-2,0)--(1/2,-1/2,0);
\fill[fill=blue, fill opacity=0.2] (1,0,0)--(3/2,0,0)--(3/2,-1/2,0)--(1/2,-1/2,0);
\fill[fill=green, fill opacity=0.2] (1/2,0, 1/2)--(1/2,0,3/2)--(3/2,0, 1/2);
\fill[fill=yellow, fill opacity=0.2] (0,0,1)--(0,0,3/2)--(1/2,0,3/2)--(1/2,0,1/2);
\fill[fill=yellow, fill opacity=0.2] (1,0,0)--(3/2,0,0)--(3/2,0,1/2)--(1/2,0, 1/2);
\fill[fill=red, fill opacity=0.2] (1,0,0)--(0,-1,0)--(0,0,1);
\end{tikzpicture}
\caption{The tropical contraction in dimension 2} \label{pic cont}
\end{figure}
If $\tau = e_{12}$ is an edge, then $W_{\tau, v_i} = \Conv(a_{12}, v_i)$ for $i=1, 2$, and the contractions over $W_{\tau, v_i}$ are the projection $N_{\R} \fl N_{\R} / \R v_i$ (the blue regions); while $W_{\tau, \tau} = a_{12}$ and $Y_{\tau, \tau}$ is the lower green region, contracted to the point $a_{12}$.
\begin{prop} Let $\tau$ be a maximal face of $\Sk(\X)$. Then the retraction $\rho_{\underline{a}} : X^{\an} \fl \Sk(\X)$ is an affinoid torus fibration over $\Int(\tau)$.
\end{prop}
\begin{proof}
Let $\tau$ be a maximal face of $\Sk(\X) \simeq \partial P^*$, and let $p \in \X_0$ be the corresponding zero-dimensional stratum. Then from the construction of $\delta$, we have $Y_{\tau} = U_{\tau} = \Int(\tau)$, so that $\delta^{-1} (\Int(\tau)) = \Int(\tau)$, and $\rho^{-1}((\Int(\tau))= \val_{\Sigma}^{-1} ((\Int(\tau))$. Moreover, by the proof of prop. \ref{prop sk N}, if:
$$\tau = \Conv(v_{l_1},..., v_{l_n})$$
and $\val_{\Sigma}(v) = \sum_{i=1}^{n+1} \lambda_i v_{l_i}$, then $\lambda_i = v(\frac{z_i}{z_0})$. In particular, if $\val_{\Sigma}(v) \in \Int(\tau)$, then $v(\frac{z_i}{z_0})>0$ for $i=1,..., (n+1)$, or in other words $z_i =0$ at the center $c_{\X}(v)$. This implies that $\val_{\Sigma}(v) \in \Int(\tau)$, if and only $c_{\X}(v) =p$.
\\We infer that $\rho^{-1}(\Int(\tau))=\fX_p^{\eta} :=(\widehat{\X}_{/p})^{\eta}$, and that under the identification $\Sk(\X) = \partial P^*$, we have $\rho_{| \fX_p^{\eta}} = \rho_{\fX_p}$, which is indeed an affinoid torus fibration by example \ref{ex:affinoid torus over max face}.
\end{proof}
In what follows, we will be mostly interested about what happens near a vertex $v_i \in \Sk(\X)$, in which case the following holds:
\begin{prop}{\cite[lem. 5.12]{Ya}} \label{prop cont local}
Let $\tau = v_i$ be a vertex of $\Sk(\X)$. Then $Y_{\tau} = U_{\tau} + \overline{\R}_{\ge 0} v_i$, and the local tropical contraction satisfies:
$$\delta_{\tau} (w + \alpha v_i) = w,$$
for $w \in U_{\tau}$ and $\alpha \in [0, + \infty]$.
\\In other words, writing $p_i : N_{\Sigma} \fl N_{\Sigma_i}$ the canonical projection, the following diagram commutes:
\begin{center}
    \begin{tikzcd}
    Y_{e_i} \arrow[r, "\delta_i"] \arrow[rd, "p_i"]  & U_i \arrow[d, "q_i"] & \subset N_{\R} \\
    & V_i & \subset N_{\R}/ \R e_i,
   \end{tikzcd}
\end{center}
where $q_i =(p_i)_{| U_i}$ is the homeomorphism from $U_i$ onto its image $V_i$ by the projection $p_i$.
\end{prop}
\subsection{Small resolution at a vertex} \label{local models}
We consider a point in ${\X}^{\text{sing}} \cap D_1 \cap ...\cap D_n$, the singular points in the other strata curves can be treated analogously. Etale locally around such a point, $\X$ is isomorphic to the toric variety $ \U:= V(z_1...z_n - wt) \subset \A^{n+1}_R$, where the $D_i|_{\U}= \{z_i=t=0\}$ are the components of the special fiber $\cU_0$. We still denote these by $D_i$; they form the toric boundary of $\cU$ together with the :
$$D_i'|_{\U}=\{z_i=w=0\}.$$
If $J \subset \{1,..., n\}$, we let $D_J = \cap_{j \in J} D_j$. The singular locus ${\U}^{\text{sing}}$ is of codimension $3$, and consists of the union over pairs $i \neq j$ of indices of the:
$$ V_{ij}|_{\U}=\{z_i=z_j=w=t=0\}=D_i \cap D_j \cap D_i' \cap D_j',$$
which all intersect at the torus invariant point $p := \{z_1=...=z_n=w=t=0\}$.
\begin{prop} \label{prop local res} The composition of successive toric blow-ups:
$$\cU(n-1) = \Bl_{D_{n-1}} \cU(n-2) \fl \cU(n-2) =... \fl \cU(1) = \Bl_{D_1} \cU  \fl \cU,$$
where we inductively abusively denote $D_{j+1} \subset \cU(j+1)$ the strict transform of $D_{j+1} \subset \cU(j)$, is a small log resolution of $(\cU, \cU_0)$, such that the strict transform of $D_n$ in $\cU(n-1)$ is isomorphic to $D_n$.
\end{prop}
When $n=3$, this was done in \cite[§4.2]{MPS}, which provides a picture of the slice of the fan of $\cU$ and the subdivisions performed.
\begin{proof} Let $N = \Z^{n+2}$, with the standard basis written as $(e_1,...,e_n, e_t, e_w)$, and $N' \subset N$ the codimension 1 sublattice:
$$N' = \{ u_1+...+ u_n = u_t + u_w \} \subset N,$$
where $(u_1,...,u_n, u_t, u_w)$ is the dual basis in $M$. Then the primitive generators of the rays of the fan $\Sigma$ of $\mathscr{U}$ in $N'_{\R}$ are the $v_i = e_i + e_t$ (corresponding to the boundary component $D_i = \{ x_i = t =0 \}$) and the $v'_i= e_i +e_w$ (corresponding to the boundary component $D'_i = \{ x_i = w =0 \}$), for $i =1,...,n$. The fan $\Sigma$ has a unique maximal cone $\sigma = \sum_{i =1}^n \R_{\ge 0} v_i + \sum_{i =1}^n \R_{\ge 0} v'_i$, it is the cone over the product of an $(n-1)$-simplex and a segment.
\\We set $\mathscr{I}_i := \mathscr{I}_{\mathscr{U}(i-1)}(D_i)$ the ideal of the strict transform of $D_i$ inside $\cU(i-1)$, so that $\cU(i)$ is the blow-up of $\cU(i-1)$ along the toric ideal $\mathscr{I}_i$. 
\\We claim that the blow-up $\cU(i) \fl \cU(i-1)$ has no exceptional divisors, and that $\mathscr{I}_{i+1} = (x_i, t)$. We prove this by induction on $i \in \{ 1,..., n-1 \}$.
\\For $i=1$, we are blowing-up the toric ideal $\mathscr{I} =(x_1, t)$ inside $\cU$, this amounts to subdividing the maximal cone $\sigma$ of $\Sigma$ into the two subcones $\sigma \cap \{ u_1 \ge u_t \}$ and $\sigma \cap \{ u_1 \le u_t \}$ - the piecewise-linear function attached to $\mathscr{I}$ is $\phi_{\mathscr{I}} = \min \{ u_1, u_t \}$. It is straightforward to check that this not subdivide any 2-dimensional cone of $\Sigma$, so that the subdivision does not induce any new ray, and $\cU(1)$ has no exceptional divisors. This also implies that the zero locus of $\mathscr{I}_2 = (x_2, t)$ is the strict transform $D_2 \subset \cU(1)$.
\\Now assume that the claim holds for $i \ge 1$, we want to prove that $\cU(i+1) \fl \cU(i)$ has no exceptional divisors and that $\mathscr{I}_{i+2} = (x_{i+1}, t)$. Since we are blowing-up the ideal $\mathscr{I}_{i+1} = (x_i, t)$ inside $\cU(i)$, this means we are subdividing the fan of $\cU(i)$ along the hyperplane $\{ u_i = u_t \}$, which is once again easily seen not to subdivide the 2-dimensional cones, hence no rays are added. By the same argument as above, this implies that inside $\cU(i+1)$, the zero locus $V(x_{i+1}, t)$ is the strict transform of $D_{i+1}$. 
We now prove that $\cU(n-1)$ is regular. The maximal cones of the fan $\cU(n-1)$ are the intersection of $\sigma$ with $(n-1)$ halfspaces of the form $\{ \pm (u_i -u_t) \ge 0 \}$; by direct computation there are $n$ maximal cones $\sigma_1,..., \sigma_n$, where:
$$\sigma_i = \sigma \cap \{ u_i \ge u_t \} \cap \{ u_{i-1} \le u_t \},$$
with $\sigma_1 = \sigma \cap \{ u_i \ge u_t \}$. This yields:
$$\sigma_i = \R_{\ge 0} v_1 +...+ \R_{\ge 0} v_i + \R_{\ge 0} v'_{i}+...+ \R_{\ge 0} v'_n,$$
which is a regular cone as $(v_1,...,v_i, v'_i,...,v'_n)$ is a basis of $N$, so that the pair $(\cU(n-1), \cU_0(n-1))$ is snc.
\\Finally, we prove that the strict transform of $D_n$ in $\cU(n-1)$ is isomorphic to $D_n$. The ray $ \rho_n = \R_{\ge 0} v_n$ viewed in the fan of $\cU(n-1)$ is contained in only one maximal cone, namely $\sigma_n$. As a result, the fan of the strict transform of $D_n$ is given by the maximal cone $\sigma_n / \R v_n$ and its faces, so that the strict transform $\tilde{D}_n$ of $D_n$ is isomorphic to $\A^n \simeq D_n$. Since the restriction of the composition of blow-ups $\tilde{D}_n \fl D_n$ induces the identity on the tori, it is an isomorphism, and this concludes the proof.
\end{proof} 
Back to the global setting, we obtain:
\begin{cor} \label{cor global res} Let $\X \subset \CP^{n+1}_R$ be as above, and $i \in \{ 0,..., n+1 \}$. Labelling the elements of $\{ 0,..., n+1 \} \setminus i$ as $\{l_1,...,l_n \}$, the sequence of successive blow-ups:
$$\X(n) = \Bl_{D_{l_n}} \X(n-1) \fl \X(n-1) =... \fl \X(1) = \Bl_{D_{l_1}} \X  \fl \X,$$
yields an snc model $\X'/R$, such that $\Sk(\X') = \Sk(\X)$ and the strict transform $D'_i \simeq D_i$.
\end{cor}
The upshot of this is that we obtain a Berkovich retraction $\rho_{\X_i} : X^{\an} \fl \Sk(\X_i) = \Sk(\X)$, which is an affinoid torus fibration over $\Star(v_i)$ by theorem \ref{THEO TORIC STRAT}. Moreover, the integral affine structure induced on $\Star(v_i)$ matches the one induced by the chart $\phi_{v_i}$ from definition \ref{def Z aff}. 
\\Let us denote by $\fX_i = \widehat{\X}_{i, /D_i}$ the formal completion. By theorem \ref{THEO TORIC STRAT}, writing $\mathcal{N} \fl \A^1_k$ the total space of normal bundle of $D_i$ inside $\X_i$ and $\cN = \mathcal{N} \times_{\A^1} R$, we have a formal isomorphism:
$$f : \fX_i \xrightarrow{\sim} \widehat{\cN}_{/D_i},$$
where $D_i \subset \cN$ denotes the zero section.
\\Thus, passing to the generic fiber, $\fX_i^{\eta}$ embeds analytically as a open subset of a torus:
$$\fX_i^{\eta} \hookrightarrow \T^{\an}_{\mathscr{N}},$$
where $\T_{\mathscr{N}} = \cN \times_R K$ is the generic fiber of $\cN$.
\begin{lem} There is a canonical isomorphism of $K$-tori:
$$\T_{\cN} \simeq \T_{D_i} \times_k K.$$
\end{lem}
\begin{proof}
Since $\mathcal{N}$ is trivial over $\T_{D_i}$, we have a natural morphism $\T_{D_i} \times \G_{m,k} \fl \mathcal{N}$, whose image is precisely the $k$-torus of $\mathcal{N}$. Moreover, from the description of the fan of $\mathcal{N}$ (see for instance \cite[§2.1]{MPS}), the morphism $\mathcal{N} \fl \A^1_k$ restricts to the projection onto the second factor, hence the result.
\end{proof}
We write $\Sigma_i$ for the fan of $D_i$ inside $N_{\R}/ \R v_i$, and let $p_i : N_{\Sigma} \fl N_{\Sigma_i}$ be the canonical projection - on $N_{\R}$, this is simply the quotient map $N_{\R} \fl N_{\R}/ \R v_i$.
\\The above lemma provides us with a map $\val_{D_i} : \fX_i^{\eta} \fl  N_{\R}/\R v_i$. Note that up to identification, the map $\val_{D_i}$ is none other than the Berkovich retraction $\rho_{\fX_i}$ - more precisely, viewing $\Sk(\X) \subset N_{\R}$ via prop. \ref{prop sk N}, the relation $\val_{D_i} = p_i \circ \rho_{\fX_i}$ holds.
\begin{prop} Let $\val_{\CP} : \CP^{n+1, \an} \fl N_{\Sigma}$. Then we have:
$$\val_{D_i} = p_i \circ \val_{\CP}.$$
\end{prop}
\begin{proof}
We assume $i=(n+1)$ for convenience and pick standard coordinates $(\frac{z_1}{z_0},...,\frac{z_{n+1}}{z_0})$ on the torus of $\CP^{n+1}$; then the projection $p_i: \R^{n+1} \fl \R^n$ is given by omitting the last coordinate.  Recall the formal isomorphism:
$$f : \widehat{\X}_{i, /D_i} \xrightarrow{\sim} \widehat{\cN}_{/D_i}$$
was constructed cone by cone, in the following way. We let $\sigma \in \Sigma_i$ be a maximal cone in the fan of $D_i$, and we assume that $\sigma = \Cone(e_1,...,e_n)$ is the standard orthant. Then the restriction $f_{\sigma}$ of $f$ to $\fX_{\sigma}$ was defined by the following property: if $s_l$ is a trivializing section of $\gO_{\X}(D_0 -D_l)$ on $\fX_{\sigma}$, restricting to the meromorphic function $z_l/z_0$ on $D_i$, then $s_l = f_{\sigma}^*(\frac{z_l}{z_0})$ (with the notation of \cite[§2.2]{MPS}, $W^{\sigma}_l = (D_0 -D_l))$. Thus, if $x \in \fX_i^{\eta}$ - which means $c_{\X_i}(v_x) \subset D_i$, then $v_x(s_l)= v_x(\frac{z_l}{z_0})$, hence:
$$\val_{\T_{\cN}}(x) = (v_x(\frac{z_1}{z_0}),...,v_x(\frac{z_n}{z_0})) \in N_{\R}/ \R e_{n+1}$$
while:
$$\val_{\T_{\CP}}(x) = \big(v_x(\frac{z_1}{z_0}),...,v_x(\frac{z_n}{z_0}), v_x(\frac{z_{n+1}}{z_0}) \big) \in N_{\R},$$
hence the result.
\end{proof}
\begin{cor} For any choice of branch cuts $\underline{a}$, the equality $\delta_{\underline{a}} \circ \val_{\CP} = \rho_{\fX_i}$ holds on $\rho_{\fX_i}^{-1} (\widetilde{\Star}(v_i)) \subset \fX_i^{\eta}$.
\end{cor}
\begin{proof} Under the identification $\Sk(\X) \subset N_{\Sigma}$, we have the equality $q_i \circ \rho_{\fX_i} = \val_{D_i}$, where $\val_{D_i}$ is the composition : $\fX_i^{\eta} \hookrightarrow \T^{\an}_{D_i} \fl N_{\R}(D_i)$. By the above proposition, we have $q_i \circ \rho_{\fX_i} = p_i \circ \val_{\CP}$, and by prop. \ref{prop cont local}, the equality $ \delta_{\underline{a}} =(q_i)^{-1} \circ p_i$ holds over $\widetilde{\Star}(v_i)$, hence the result.
\end{proof}
Since $\rho_{\fX_i}$ is an affinoid torus fibration over $\Star(v_i)$ and in particular over $\widetilde{\Star}(v_i)$, this implies that $\rho_{\underline{a}} = \delta_{\underline{a}} \circ \val_{\CP}$ is an affinoid torus fibration away from $\Gamma_{\underline{a}}$, and this concludes the proof of theorem \ref{theo affinoid Yuto}.
\section{The Fermat case} \label{sec Fermat}
\subsection{The non-archimedean Monge-Ampère equation}
Let $X \subset \CP^{n+1} \times \D^*$ be the Fermat family of Calabi-Yau hypersurfaces:
$$X = \{ z_0 ... z_{n+1} + t (z_0^{n+2} +...+ z_{n+1}^{n+2}) \} =0,$$
and $\X \subset \CP^{n+1} \times \D$ its closure. We use the same notation for the respective base change to $K$ and $R$. We are in the setting of the previous section, as the Fermat polynomial $F = z_0^{n+2}+...+ z_{n+1}^{n+2}$ is such that condition \ref{condition} holds. The family $X$ is endowed with the polarization $L = \gO_X(n+2)$, and with the family of Calabi-Yau measures:
$$\nu_t := i^{n^2} \Omega_t \wedge \bar{\Omega}_t,$$
where $(\Omega_t)_{t \in \D^*}$ is a holomorphic trivialization of the relative canonical bundle $K_{X/\D^*}$. We set:
$$\mu_t := C_t \nu_t,$$
where $C_t>0$ is the unique constant ensuring that $\int_{X_t} \mu_t = L_t^n$. By Theorem \ref{theo cv mesures}, the family $(\mu_t)_{t \in \D^*}$ converges weakly on $X^{\hyb}$ to the Lebesgue measure $\mu_0$ on $\Sk(X)$.
\\Each fiber $X_t$ is endowed with the unique Kähler Ricci-flat metric $\omega_t \in c_1(L)$, satisfying the complex Monge-Ampère equation:
$$\omega_{CY, t}^n = \mu_t.$$
Since for all $t$, the Kähler form $\omega_t \in c_1(L_t)$, there exists a family of Hermitian metrics $\phi_{CY,t}$ on $L_t$ such that $\omega_{CY, t} = dd^c \phi_{CY, t}$ - note however that this family need not vary in a subharmonic way with respect to the direction of the base.
\\Recall that $P \subset M_{\R}$ is the convex polytope associated to the toric boundary $D$ of $\CP^{n+1}$, so that $L= \gO_{\CP}(D)$ is a toric line bundle, and psh toric metrics on $(\CP^{n+1, \an}, L^{\an})$ are in one-to-one correspondence with convex $P$-admissible functions on $N_{\R}$, by theorem \ref{theo toric field}. We now have the following:
\begin{theo} \label{Fermat sol}
There exists a convex $P$-admissible function $u: N_{\R} \fl \R$, and induced continuous psh toric metric $\phi_u \in \CPSH(\CP^{n+1, \an}_K, L^{\an})$, such that the restriction $\psi_u = (\phi_u)_{| X^{\an}}$ solves the non-archimedean Monge-Ampère equation:
$$\MA(\psi_u) = \mu_0$$
on $X^{\an}$.
\end{theo}
A similar, more general statement was recently obtained independently in \cite{HJMM} for hypersurfaces of the form considered in section \ref{sec super setup}, building on the construction of a solution to the tropical Monge-Ampère equation on $\Sk(X)$.
\\The fact that the function $u$ solves the real Monge-Ampère equation $\M(u) = \mu_0$ on $\Sk(X) \setminus \Gamma$ was established in \cite{Li1}, however since the solution to this equation may not be unique, the convex function $u$ was only obtained along a subsequence, and could depend on the choice thereof. However the solution to the NAMA equation is unique, so that as observed in \cite[rem. 5.7]{Li1}, we get the following strenghtening of \cite[thm. 5.1]{Li1}:
\begin{cor}
The locally convex function $u_{\infty}$ obtained in \cite[thm. 5.1]{Li1}, as well as the generic SYZ fibration from \cite[thm. 5.14]{Li1}, do not depend on the choice of a subsequence.
\end{cor}
We now move on to the proof of theorem \ref{Fermat sol}, and start by explaining how the function $u$ is obtained. 
Since in the sequel it will be more convenient to work with potentials instead of metrics, we fix once and for all the reference hybrid metric $\phi_{\FS}$ on $\CP^{n+1, \hyb}$, such that:
$$\phi_{FS, t} =\frac{(n+2)}{2} \log(\lvert z_0\rvert^2 +... + \lvert z_{n+1} \rvert^2),$$
and:
$$\phi_{FS, 0} = (n+2) \max ( \log \lvert z_0 \rvert,..., \log \lvert z_{n+1} \rvert);$$
it is a continuous plurisubharmonic metric on $\CP^{n+1, \hyb}$ by \cite[ex. 3.7]{PShyb}; we use the same notation for its restriction to $X^{\an}$. This allows us to identify the Calabi-Yau metric with its relative potential:
$$\varphi_{CY, t} = \phi_{CY, t} - \phi_{FS, t},$$
normalized by the convention $\sup_{X_t} \varphi_{CY, t}=0$, for all $t \in \bar{\D}_r$. Then Li produces, via double Legendre transform, $P$-admissible convex functions $u_{CY, t}$ on $N_{\R}$ which approximate the Calabi-Yau potential, and which satisfy suitable Lipschitz bounds. This implies that for any sequence $(t_k)_{k \in \N}$ going to zero inside $\D^*$, there exists a subsequence (which we will omit from notation) such that the admissible convex functions $u_{CY, t_k}$ converge locally uniformly to a continuous, convex admissible function $u \in \Ad_P(N_{\R})$ - which \emph{a priori} depends on a choice of subsequence. By Theorem \ref{theo toric} and in particular example \ref{ex toric prod}, this induces a continuous hybrid metric $\Phi_u \in \CPSH( \CP^{n+1, \hyb}, L^{\hyb})$, and by restriction to $X^{\hyb}$ a continuous psh hybrid metric $\Psi_u$ on $(X, L)$. We write, following the notation of \cite{Li1} (with a different sign convention), $s(t) = \frac{-1}{\log \lvert t \rvert}$, and $\Log_s(z) := \Log(sz)$ on the torus of $\CP^{n+1}$, so that the restriction of $\Psi_u$ to $X_t$ is equal to $s^{-1} (u \circ \Log_s)$. 
\begin{prop} \label{prop cv Fermat} As $k \rightarrow \infty$ (up to extracting further), the sequence of complex Monge-Ampère measures $( s_{k}^{-n}dd^c (u \circ \Log_{s_k})^n)_{k \in \N}$ on $X^{\hyb}$ converge weakly to the Lebesgue measure $\mu_0$ on the essential skeleton.
\end{prop}
Before moving to the proof of the proposition, we explain how this implies Theorem \ref{Fermat sol}. Writing $\psi_u = \Psi_u^{\NA}$ the restriction of $\Psi_u$ to $X^{\an}$, it follows from Theorem \ref{theo cont MA} that the sequence of measures $(s_k^{-n} dd^c (u \circ \Log_{s_k})^n)_{k \in \N}$ on $X^{\hyb}$ converges weakly to $\MA(\psi_u)$, hence the equality. The uniqueness of $u$ now follows from uniqueness of the solution to the non-archimedean Monge-Ampère equation.
\begin{proof}[Proof of prop. \ref{prop cv Fermat}]
By compactness of $X^{\hyb}$ and after further extraction, we may assume that the sequence of measures $\theta_k := s_k^{-n} dd^c (u \circ \Log_{t_k})^n$ on $X^{\hyb}$ converge weakly to a limit measure $\theta_{\infty}$ supported on $X^{\an}$, of total mass $L^n$. Let $f \in \mathcal{C}^0(X^{\hyb})$, then by \cite[lem. 5.3, proof of thm. 5.1]{Li1}, as $k \rightarrow \infty$ (up to extracting further) we have:
$$\big\lvert \int_{U_{s_k}} f\mu_{s_k} - \int_{U_{s_k}}f\theta_k \big\rvert \xrightarrow[k \rightarrow \infty]{} 0,$$
where $U_{s_k}$ denotes the union of the 'toric regions' $U^{s_k,*}_w$ from \cite[§4.5]{Li1}. Moreover, we have $\mu_{s_k}(X_{s_k} \setminus U_{s_k}) \fl 0$ when $k  \rightarrow \infty$, by \cite[prop. 3.14]{Li1}. Thus, if $f \ge 0$, we infer that $\int_{X^{\an}} f \theta_{\infty} \ge \int_{X^{\an}} f \mu_0$, hence $\theta_{\infty} \ge \mu_0$ as measures on $X^{\an}$. However both measures have the same mass $L^n$, so that they are equal, which concludes the proof.
\end{proof}
\subsection{The comparison property}
In this section we prove the following:
\begin{theo} \label{theo comp prop} Let $X \subset \CP^{n+1} \times \D^*$ be a degeneration of Calabi-Yau hypersurfaces such that condition \ref{condition} holds, and set $L = \gO_{\CP}(n+2)$.
\\We write $(\X, \Ld)$ the dlt model of $X$ obtained by taking the closure of $X$ inside $\CP^{n+1} \times \D^*$.
\\Writing the solution $\phi$ to the NAMA equation as:
$$\phi = \phi_{\Ld} + \psi,$$
we have that $\psi = \psi \circ \rho$ over $\Sk(X) \setminus \Gamma$, where $\rho : X^{\an} \fl \Sk(X)$ is the admissible retraction constructed in Theorem \ref{theo affinoid Yuto}, and with choice of branch cuts such that $a_{\tau}$ is the barycenter of $\tau$, for each face $\tau \subset \Sk(\X)$.
\end{theo}
Recall that the solution of the non-archimedean Monge-Ampère is explicitly given as the restriction to $X^{\an}$ of the toric metric $\phi_u \in \CPSH(\CP_K^{n+1, \an}, L^{\an})$ - in the Fermat case, this follows from Theorem \ref{Fermat sol}, while this follows from \cite[thm. 8.1]{HJMM} in general. The model metric $\phi_{\Ld}$ is also toric, and we have $\phi_{\Ld} = \phi_{u_{\FS}}$, where:
$$u_{\FS}(x) = \max_{m \in V(P)} \langle m, x \rangle$$
is the support function of the boundary of $\CP^{n+1}$. We infer that the relative toric potential $\psi = (u - u_{\FS}) \circ \val_{\Sigma}$ on $\CP^{n+1}$, hence also on $X$. We set $u' = (u - u_{\FS})$ which is a continuous function on $N_{\R}$, extending continuously to $N_{\Sigma}$.
\\We need to prove that if $\rho(x) \in \Sk(X) \setminus \Gamma$, then $\psi(x) = \psi(\rho(x))$. It is enough to treat the two following cases: $\rho(x) \in \Int(\tau)$ for a maximal face $\tau$ of $\Sk(\X)$, or $\rho(x) \in U_i = \Int(\widetilde{\Star}(v_i))$ for a vertex $v_i$ of the skeleton.
\\We start with the former. If $\rho(x) \in \Int(\tau)$, then $\rho(x) = \val_{\Sigma}(x)$ under the identification $\Sk(X) = \partial P^*$, hence $\psi (\rho(x)) = u' (\val_{\Sigma} (x)) = \psi(x)$ as $\phi$ is a toric metric on $\CP^{n+1}$.
\\The second case relies on the following observation:
\begin{lem}{\cite[cor. 3.28]{Li1}} Let $u : N_{\R} \fl \R$ be the convex admissible function from the statement of theorem \ref{Fermat sol}, and let $v_i \in \Sk(\X)$ be a vertex. Then on the region $Y_{v_i} = (U_i + \R_{\ge 0} v_i) \cap \Trop(X)$, the equality:
$$u(w+ \alpha v_i) = u(x) + \alpha$$
holds for any $w \in U_i$, $\alpha \in \R_{\ge 0}$.
\end{lem}
This is proved only in the case of the Fermat family in \cite{Li1}, but by \cite[thm. 8.1]{HJMM} in the general case, the function $u : N_{\R} \fl \R$ is the canonical extension $v^{**}$ of its restriction $v:= u_{| \Sk(X)}$. Hence the proof of \cite[cor. 3.28]{Li1} applies without changes to any degeneration of hypersurfaces such that condition \ref{condition} holds, as the proof only uses $\Trop(X)$ and not the hypersurface itself.
\\This implies that if $\rho(x) \in U_i$ - which means $\val_{\Sigma}(x) \in X_{v_i}$ - then $u'(x) = u'(\delta(x))$. Indeed, writing $x = w+ \alpha v_i$ with $x \in U_i$, we have:
$$u_{\FS} (x) = \max_{m \in V(P)} \langle m, x \rangle = u_{\FS}(w) + \max_{m \in V(P)} \langle m, \alpha v_i \rangle = u_{\FS}(w) + \alpha,$$
since $v_i$ is a vertex of $P^* =\{ \langle P, \cdot \rangle \le 1 \}$. This yields $u'(x) =u'(w) = u'(\delta(x))$ by prop. \ref{prop cont local}, and since $\psi = u' \circ \val_{\Sigma}$, we get that:
$$\psi (\rho(x)) = \psi(x)$$
whenever $\rho(x) \in U_i$. This concludes the proof of theorem \ref{theo comp prop}.
\begin{rem} \label{rem MA hyper} Let $X \subset \CP^{n+1} \times \D^*$ be a maximally degenerate hypersurface as considered in theorem \ref{theo affinoid Yuto}, and assume that for a certain choice $\underline{a}$ of branch cuts, the solution $\psi \in \mathcal{C}^0(X^{\an})$ of the non-archimedean Monge-Ampère equation:
$$\MA(\phi_{\Ld} + \psi) = \mu_0$$
satisfies the comparison property $\psi = \psi \circ \rho_{\underline{a}}$ over $\Sk(X) \setminus \Gamma$. We write $\psi_{\FS} : N_{\R} \fl \R$ the toric potential for the Fubini-Study metric on the ambient $\CP^{n+1}$ with the anticanonical polarization, and use the same notation for its restriction to $\Sk(X)$.
\\Then the collection of local functions:
$$\psi_{\tau} := \psi \; \text{on} \; \Int(\tau),$$
$$\psi_i := \psi + \psi_{\FS} -m_i \; \text{on} \; U_i,$$
where $m_i \in M$ satisfies $\langle m_i, w \rangle =1$, are convex in affine coordinates on $\Sk(X) \setminus \Gamma$, and satisfy the real Monge-Ampère equation:
$$n!\M(\psi) = \mu_0,$$
away from $\Gamma$. Note that the local potential $\psi_i$ equates the $u_m$ from \cite[def. 3.22]{Li1}.
\\This holds over the maximal faces of $\Sk(\X)$ by theorem \ref{theo Vil}, so that we focus on an open subset $U_i$. It is enough to prove that under the formal isomorphism:
$$\widehat{\X_i}_{/D_i} \simeq \widehat{\cN}_{/D_i}$$
provided by theorem \ref{THEO TORIC STRAT}, the metric $\phi = \phi_{\Ld} + \psi$ is toric and has toric potential $\psi_i$ on $U_i$. The first part follows from the assumption $\psi = \psi \circ \rho$, while the second one follows from the fact that the toric potential for $(\phi_{\Ld})_{| \T_{D_i}}$ is $(\psi_{\FS} -m_i)$ - recall that the torus $\T_{D_i, K}$ of $\cN$ is the quotient of $\T_{\CP}$ by the $m_i$-direction.
\end{rem}
\subsection{The hybrid SYZ fibration}
While the results from \cite{Li1} do not yield a global SYZ fibration on the Fermat hypersurfaces $X_t$ for $\lvert t \rvert \ll1$, there does exist a special Lagrangian fibration on a region of $X_t$ whose Calabi-Yau measure is arbitrarily close to $1$ as $t \rightarrow 0$. We will show that this fibration converges to the retraction $\rho$ from thm. \ref{Fermat sol} in the hybrid topology, this is similar to the results from \cite{GO} for finite quotients of abelian varieties. 
\\We start by recalling the precise statement on the existence of an SYZ fibration in the generic region. By the local theory for the real Monge-Ampère equation, the locally convex $u$ on $\Sk(X) \setminus \Gamma$ is smooth and strictly convex on an open subset $\mathcal{R} \subset \big( \Sk(X) \setminus \Gamma \big)$, whose complement has $(n-1)$-Hausdorff measure zero; in particular $\mathcal{R}$ is connected. The estimates from \cite{Li1} imply that the Calabi-Yau potential converges in a natural, $\mathcal{C}^{\infty}$-sense to $u$ locally over $\mathcal{R}$, so that the results from \cite{ZhSL} can be used to obtain the following:
\begin{theo}{\cite[thm. 5.13]{Li1}} Let $K \subset \mathcal{R}$ be a compact subset.
\\There exists an open neighbourhood $V_{K} \subset \mathcal{R}$ of $K$, and writing $U_{t, K} = \Log_t^{-1}(V_K)$, a special Lagrangian fibration:
$$f_t : U_{t, K} \fl \Sk(X)$$
that is a $\mathcal{C}^{\infty}$-perturbation of the map $\Log_t$ as $t \rightarrow 0$. Moreover, if $(K_n)_{n \in \N}$ is a compact exhaustion of $\mathcal{R}$, then the $U_{t, K_n}$ have Calabi-Yau measure arbitrarily close to $1$ as $t \rightarrow 0$, $n \gg1$.
\end{theo}
Here by $\mathcal{C}^{\infty}$-perturbation we mean that the $\mathcal{C}^{\infty}$-norm of $(\Log_t - f_t)$ is arbitrary small as $t \rightarrow 0$, in the scale where the torus fibers have bounded diameter. We claim that these generic SYZ fibrations converge in the hybrid topology to $\rho$:
\begin{prop} Let $U_{0, K} = \rho^{-1}(V_K)$ and:
$$U^{\hyb}_K = \val_{\hyb}^{-1}(V_K) = \bigsqcup_{ t \in \bar{\D}_r} U_{t, K} \subset X^{\hyb}$$
for $r\ll 1$ depending on $K$. Then the hybrid SYZ fibration:
$$f : U^{\hyb}_K \fl K \subset \Sk(X)$$
defined by $f_t$ on $U_{t, K}$ for $t \neq 0$ and $\rho$ on $U_{0, K}$ is continuous.
\end{prop}
\begin{proof} 
By \cite[§2.7, thm. 5.13]{Li1}, for $\lvert t \rvert \ll 1$, the SYZ fibration $f_t$ over $K$ is $\mathcal{C}^{\infty}$-close to the map $\Log_t$ in the scale where the torus fibers have diameters of size $O(1)$, which easily implies our statement, as $\Log_t$ converges to $\rho$ over $K$ in the hybrid topology. 
\end{proof}
\bibliographystyle{alpha}
\bibliography{biblio.bib}
\end{document}

%% file: macros.tex
\newcommand{\ddbar}{\partial\bar{\partial}}
\newcommand \Z{\mathbb Z}
\newcommand \cZ {\mathscr Z}
\newcommand \N{\mathbb N}
\newcommand \C{\mathbb C}
\newcommand \R{\mathbb R}
\newcommand \Q{\mathbb Q}
\newcommand \gO {\mathcal O}

\newcommand \M {\mathscr {M}}
\newcommand \Ld {\mathscr {L}}
\newcommand \D {\mathbb{D}}

\newcommand \U {\mathscr U}
\newcommand \cU {\mathscr U}
\newcommand \CP {\mathbb P}
\newcommand \A {\mathbb A}

\newcommand \T {\mathbb{T}}

\newcommand \G {\mathbb{G}}
\newcommand \X {\mathscr{X}}

\newcommand \fX {\mathfrak{X}}

\newcommand \bS {\mathbb S}
\newcommand \cD {\mathscr{D}}
\newcommand \cN {\mathscr{N}}

\newcommand \fl {\longrightarrow}
\newcommand \eps {\varepsilon}

\newcommand \De {\Delta}

\DeclareMathOperator{\Sk}{Sk}

\DeclareMathOperator{\Star}{Star}

\DeclareMathOperator{\ord}{ord}
\DeclareMathOperator{\Hom}{Hom}
\DeclareMathOperator{\Bl}{Bl}

\DeclareMathOperator \Id {Id}

\DeclareMathOperator \Div {Div}

\DeclareMathOperator{\PSH}{PSH}
\DeclareMathOperator{\CPSH}{CPSH}

\DeclareMathOperator{\Sp}{Spec}
\DeclareMathOperator{\Spf}{Spf}
\DeclareMathOperator{\Log}{Log}
\DeclareMathOperator{\Sing}{Sing}
\DeclareMathOperator{\MA}{MA}
\DeclareMathOperator{\red}{red}
\DeclareMathOperator{\val}{val}
\DeclareMathOperator{\Val}{Val}

\DeclareMathOperator{\an}{an}
\DeclareMathOperator{\An}{An}
\DeclareMathOperator{\hyb}{hyb}

\DeclareMathOperator{\Conv}{Conv}
\DeclareMathOperator{\Aff}{Aff}
\DeclareMathOperator{\sg}{sg}
\DeclareMathOperator{\ev}{ev}

\DeclareMathOperator{\FS}{FS}
\DeclareMathOperator{\DFS}{DFS}
\DeclareMathOperator{\hol}{hol}

\DeclareMathOperator{\Ad}{Ad}
\DeclareMathOperator{\NA}{NA}

\DeclareMathOperator{\Frac}{Frac}

\DeclareMathOperator{\Trop}{Trop}

\DeclareMathOperator{\wt}{wt}
\DeclareMathOperator{\sch}{sch}

\DeclareMathOperator{\sm}{sm}
\DeclareMathOperator{\Int}{Int}
\DeclareMathOperator{\Cone}{Cone}
\DeclareMathOperator{\Span}{Span}
\DeclareMathOperator{\xdiv}{div}
\DeclareMathOperator{\can}{can}
\newtheorem*{thmA}{Theorem A}
\newtheorem*{thmB}{Theorem B}
\newtheorem*{thmC}{Theorem C}